\documentclass[11pt,a4paper]{amsart}

\usepackage{amsmath}
\usepackage{mathtools}
\usepackage{amsthm}
\usepackage{amssymb}
\usepackage{hyperref}

\usepackage{booktabs}

\newcommand*\fullref[3][\relax]{%
  \ifdefined\hyperref%
    {\hyperref[#3]{#2\penalty 200\ \ref*{#3}#1}}%
  \else%
    {#2\penalty 200\ \relax\ref{#3}#1}%
  \fi%
}

\usepackage{tikz}
\usetikzlibrary{arrows.meta}
\usetikzlibrary{calc}
\usetikzlibrary{decorations.pathmorphing}
\usetikzlibrary{decorations.text}

\usetikzlibrary{shapes.geometric} 
\usetikzlibrary{shapes.misc}  

\tikzset{
  bst/.style={
    standard/.style={
      font=\small,
      draw=gray,
      rounded rectangle,
      minimum width=4.5mm,
      minimum height=4.5mm,
      inner xsep=0mm,
      inner ysep=1mm,
      outer sep=0mm,
      line width=.5pt,
    },
    empty/.style={
      minimum width=3mm,
      minimum height=3mm,
    },
    triangle/.style={
      isosceles triangle,
      isosceles triangle apex angle=60,
      shape border rotate=90,
      rounded corners=2mm,
      minimum width=8mm,
      inner xsep=0mm,
      inner ysep=.5mm
    },
    blank/.style={
      draw=none,
    },
    nodecount/.style={
      blank,
      font=\scriptsize,
    },
    every node/.style={standard},
    every child/.style={draw=black,line width=.6pt},
    level distance=10mm,
    level 1/.style={sibling distance=60mm},
    level 2/.style={sibling distance=30mm},
    level 3/.style={sibling distance=15mm},
  },
  medbst/.style={
    bst,
    level distance=10mm,
    level 1/.style={sibling distance=15mm},
    level 2/.style={sibling distance=15mm},
    level 3/.style={sibling distance=15mm},
  },
  smallbst/.style={
    bst,
    level distance=8mm,
    level 1/.style={sibling distance=10mm},
    level 2/.style={sibling distance=10mm},
    level 3/.style={sibling distance=10mm},
  },
  tinybst/.style={
    bst,
    level distance=5mm,
    level 1/.style={sibling distance=8mm},
    level 2/.style={sibling distance=8mm},
    level 3/.style={sibling distance=8mm},
    every node/.append style={
      font=\footnotesize,
    },
    triangle/.append style={
      rounded corners=1mm,
      minimum width=7mm,
      inner xsep=-.5mm,
    },
  },
  microbst/.style={
    bst,
    standard/.append style={
      font=\scriptsize,
      minimum width=3mm,
      minimum height=3mm,
      inner ysep=.25mm,
    },
    level distance=3mm,
    level 1/.style={sibling distance=6mm},
    level 2/.style={sibling distance=6mm},
    level 3/.style={sibling distance=6mm},
  },
  nanobst/.style={
    bst,
    standard/.append style={
      font=\tiny,
      minimum width=2mm,
      minimum height=2mm,
      inner ysep=.25mm,
    },
    level distance=2mm,
    level 1/.style={sibling distance=4mm},
    level 2/.style={sibling distance=4mm},
    level 3/.style={sibling distance=4mm},
  },
}
\usetikzlibrary{intersections}

\makeatletter

\pgfkeys{/pgf/.cd,
  rectangle corner radius/.initial=3pt
}
\newif\ifpgf@rectanglewrc@donecorner@
\def\pgf@rectanglewithroundedcorners@docorner#1#2#3#4{%
  \edef\pgf@marshal{%
    \noexpand\pgfintersectionofpaths
      {%
        \noexpand\pgfpathmoveto{\noexpand\pgfpoint{\the\pgf@xa}{\the\pgf@ya}}%
        \noexpand\pgfpathlineto{\noexpand\pgfpoint{\the\pgf@x}{\the\pgf@y}}%
      }%
      {%
        \noexpand\pgfpathmoveto{\noexpand\pgfpointadd
          {\noexpand\pgfpoint{\the\pgf@xc}{\the\pgf@yc}}%
          {\noexpand\pgfpoint{#1}{#2}}}%
        \noexpand\pgfpatharc{#3}{#4}{\cornerradius}%
      }%
    }%
  \pgf@process{\pgf@marshal\pgfpointintersectionsolution{1}}%
  \pgf@process{\pgftransforminvert\pgfpointtransformed{}}%
  \pgf@rectanglewrc@donecorner@true
}
\pgfdeclareshape{rectangle with rounded corners}
{
  \inheritsavedanchors[from=rectangle] 
  \inheritanchor[from=rectangle]{north}
  \inheritanchor[from=rectangle]{north west}
  \inheritanchor[from=rectangle]{north east}
  \inheritanchor[from=rectangle]{center}
  \inheritanchor[from=rectangle]{west}
  \inheritanchor[from=rectangle]{east}
  \inheritanchor[from=rectangle]{mid}
  \inheritanchor[from=rectangle]{mid west}
  \inheritanchor[from=rectangle]{mid east}
  \inheritanchor[from=rectangle]{base}
  \inheritanchor[from=rectangle]{base west}
  \inheritanchor[from=rectangle]{base east}
  \inheritanchor[from=rectangle]{south}
  \inheritanchor[from=rectangle]{south west}
  \inheritanchor[from=rectangle]{south east}

  \savedmacro\cornerradius{%
    \edef\cornerradius{\pgfkeysvalueof{/pgf/rectangle corner radius}}%
  }

  \backgroundpath{%
    \northeast\advance\pgf@y-\cornerradius\relax
    \pgfpathmoveto{}%
    \pgfpatharc{0}{90}{\cornerradius}%
    \northeast\pgf@ya=\pgf@y\southwest\advance\pgf@x\cornerradius\relax\pgf@y=\pgf@ya
    \pgfpathlineto{}%
    \pgfpatharc{90}{180}{\cornerradius}%
    \southwest\advance\pgf@y\cornerradius\relax
    \pgfpathlineto{}%
    \pgfpatharc{180}{270}{\cornerradius}%
    \northeast\pgf@xa=\pgf@x\advance\pgf@xa-\cornerradius\southwest\pgf@x=\pgf@xa
    \pgfpathlineto{}%
    \pgfpatharc{270}{360}{\cornerradius}%
    \northeast\advance\pgf@y-\cornerradius\relax
    \pgfpathlineto{}%
  }

  \anchor{before north east}{\northeast\advance\pgf@y-\cornerradius}
  \anchor{after north east}{\northeast\advance\pgf@x-\cornerradius}
  \anchor{before north west}{\southwest\pgf@xa=\pgf@x\advance\pgf@xa\cornerradius
    \northeast\pgf@x=\pgf@xa}
  \anchor{after north west}{\northeast\pgf@ya=\pgf@y\advance\pgf@ya-\cornerradius
    \southwest\pgf@y=\pgf@ya}
  \anchor{before south west}{\southwest\advance\pgf@y\cornerradius}
  \anchor{after south west}{\southwest\advance\pgf@x\cornerradius}
  \anchor{before south east}{\northeast\pgf@xa=\pgf@x\advance\pgf@xa-\cornerradius
    \southwest\pgf@x=\pgf@xa}
  \anchor{after south east}{\southwest\pgf@ya=\pgf@y\advance\pgf@ya\cornerradius
    \northeast\pgf@y=\pgf@ya}

  \anchorborder{%
    \pgf@xb=\pgf@x
    \pgf@yb=\pgf@y%
    \southwest%
    \pgf@xa=\pgf@x
    \pgf@ya=\pgf@y%
    \northeast%
    \advance\pgf@x by-\pgf@xa%
    \advance\pgf@y by-\pgf@ya%
    \pgf@xc=.5\pgf@x
    \pgf@yc=.5\pgf@y%
    \advance\pgf@xa by\pgf@xc
    \advance\pgf@ya by\pgf@yc%
    \edef\pgf@marshal{%
      \noexpand\pgfpointborderrectangle
      {\noexpand\pgfqpoint{\the\pgf@xb}{\the\pgf@yb}}
      {\noexpand\pgfqpoint{\the\pgf@xc}{\the\pgf@yc}}%
    }%
    \pgf@process{\pgf@marshal}%
    \advance\pgf@x by\pgf@xa%
    \advance\pgf@y by\pgf@ya%
    \pgfextract@process\borderpoint{}%
    \pgf@rectanglewrc@donecorner@false
    \southwest\pgf@xc=\pgf@x\pgf@yc=\pgf@y
    \advance\pgf@xc\cornerradius\relax\advance\pgf@yc\cornerradius\relax
    \borderpoint
    \ifdim\pgf@x<\pgf@xc\relax\ifdim\pgf@y<\pgf@yc\relax
      \pgf@rectanglewithroundedcorners@docorner{-\cornerradius}{0pt}{180}{270}%
    \fi\fi
    \ifpgf@rectanglewrc@donecorner@\else
      \southwest\pgf@yc=\pgf@y\relax\northeast\pgf@xc=\pgf@x\relax
      \advance\pgf@xc-\cornerradius\relax\advance\pgf@yc\cornerradius\relax
      \borderpoint
      \ifdim\pgf@x>\pgf@xc\relax\ifdim\pgf@y<\pgf@yc\relax
       \pgf@rectanglewithroundedcorners@docorner{0pt}{-\cornerradius}{270}{360}%
      \fi\fi
    \fi
    \ifpgf@rectanglewrc@donecorner@\else
      \northeast\pgf@xc=\pgf@x\relax\pgf@yc=\pgf@y\relax
      \advance\pgf@xc-\cornerradius\relax\advance\pgf@yc-\cornerradius\relax
      \borderpoint
      \ifdim\pgf@x>\pgf@xc\relax\ifdim\pgf@y>\pgf@yc\relax
       \pgf@rectanglewithroundedcorners@docorner{\cornerradius}{0pt}{0}{90}%
      \fi\fi
    \fi
    \ifpgf@rectanglewrc@donecorner@\else
      \northeast\pgf@yc=\pgf@y\relax\southwest\pgf@xc=\pgf@x\relax
      \advance\pgf@xc\cornerradius\relax\advance\pgf@yc-\cornerradius\relax
      \borderpoint
      \ifdim\pgf@x<\pgf@xc\relax\ifdim\pgf@y>\pgf@yc\relax
       \pgf@rectanglewithroundedcorners@docorner{0pt}{\cornerradius}{90}{180}%
      \fi\fi
    \fi
  }
}

\makeatother

\makeatletter

\@ifpackagelater{tikz}{2013/12/01}{
  \newcommand{\convexpath}[2]{
    [
    create hullcoords/.code={
      \global\edef\namelist{#1}
      \foreach [count=\counter] \nodename in \namelist {
        \global\edef\numberofnodes{\counter}
        \coordinate (hullcoord\counter) at (\nodename);
      }
      \coordinate (hullcoord0) at (hullcoord\numberofnodes);
      \pgfmathtruncatemacro\lastnumber{\numberofnodes+1}
      \coordinate (hullcoord\lastnumber) at (hullcoord1);
    },
    create hullcoords
    ]
    ($(hullcoord1)!#2!-90:(hullcoord0)$)
    \foreach [
    evaluate=\currentnode as \previousnode using \currentnode-1,
    evaluate=\currentnode as \nextnode using \currentnode+1
    ] \currentnode in {1,...,\numberofnodes} {
      let \p1 = ($(hullcoord\currentnode) - (hullcoord\previousnode)$),
      \n1 = {atan2(\y1,\x1) + 90},
      \p2 = ($(hullcoord\nextnode) - (hullcoord\currentnode)$),
      \n2 = {atan2(\y2,\x2) + 90},
      \n{delta} = {Mod(\n2-\n1,360) - 360}
      in
      {arc [start angle=\n1, delta angle=\n{delta}, radius=#2]}
      -- ($(hullcoord\nextnode)!#2!-90:(hullcoord\currentnode)$)
    }
  }
}{
  \newcommand{\convexpath}[2]{
    [
    create hullcoords/.code={
      \global\edef\namelist{#1}
      \foreach [count=\counter] \nodename in \namelist {
        \global\edef\numberofnodes{\counter}
        \coordinate (hullcoord\counter) at (\nodename);
      }
      \coordinate (hullcoord0) at (hullcoord\numberofnodes);
      \pgfmathtruncatemacro\lastnumber{\numberofnodes+1}
      \coordinate (hullcoord\lastnumber) at (hullcoord1);
    },
    create hullcoords
    ]
    ($(hullcoord1)!#2!-90:(hullcoord0)$)
    \foreach [
    evaluate=\currentnode as \previousnode using \currentnode-1,
    evaluate=\currentnode as \nextnode using \currentnode+1
    ] \currentnode in {1,...,\numberofnodes} {
      let \p1 = ($(hullcoord\currentnode) - (hullcoord\previousnode)$),
      \n1 = {atan2(\x1,\y1) + 90},
      \p2 = ($(hullcoord\nextnode) - (hullcoord\currentnode)$),
      \n2 = {atan2(\x2,\y2) + 90},
      \n{delta} = {Mod(\n2-\n1,360) - 360}
      in
      {arc [start angle=\n1, delta angle=\n{delta}, radius=#2]}
      -- ($(hullcoord\nextnode)!#2!-90:(hullcoord\currentnode)$)
    }
  }
}

\makeatother%

\usetikzlibrary{matrix}

\tikzset{
  pretableaumatrix/.style={
    ampersand replacement=\&,
    matrix of math nodes,
    outer sep=1mm,
    inner sep=0mm,
    anchor=center,
    row sep={between borders,-\pgflinewidth},
    column sep={between borders,-\pgflinewidth},
    dottedentry/.style={densely dotted},
    spaceentry/.style={draw=none,execute at begin node=\null},
  },
  pretableaunode/.style={
    font=\small,
    draw=gray,
    sharp corners,
    rectangle,
    anchor=base,
    text height=3.75mm,
    text depth=1.25mm,
    minimum height=5mm,
    minimum width=5mm,
    inner sep=0mm,
    outer sep=0mm,
  },
  tableaumatrix/.style={
    pretableaumatrix,
    every node/.append style={
      pretableaunode,
    },
  },
  medtableaumatrix/.style={
    pretableaumatrix,
    every node/.append style={
      pretableaunode,
      font=\footnotesize,
      text height=2.75mm,
      text depth=.75mm,
      minimum height=3.5mm,
      minimum width=3.5mm
    },
  },
  smalltableaumatrix/.style={
    pretableaumatrix,
    every node/.append style={
      pretableaunode,
      font=\scriptsize,
      text height=1.85mm,
      text depth=.15mm,
      minimum height=2.5mm,
      minimum width=2.5mm,
    },
  },
  tinytableaumatrix/.style={
    pretableaumatrix,
    every node/.append style={
      pretableaunode,
      font=\tiny,
      text height=1.25mm,
      text depth=.15mm,
      minimum height=1.75mm,
      minimum width=1.75mm
    },
  },
  tableau/.style={
    baseline=-1.25mm,
    every matrix/.style={tableaumatrix},
  },
  medtableau/.style={
    baseline=-1.25mm,
    every matrix/.style={medtableaumatrix},
  },
  smalltableau/.style={
    baseline=-1.25mm,
    every matrix/.style={smalltableaumatrix},
  },
  preshapetableaumatrix/.style={
    pretableaumatrix,
    execute at end cell={\strut},
    every node/.append style={
      draw=black,
      anchor=base,
      inner sep=0mm,
      outer sep=0mm,
    },
    shadedentry/.style={fill=gray},
    darkshadedentry/.style={fill=darkgray},
  },
  medshapetableaumatrix/.style={
    preshapetableaumatrix,
    every node/.append style={
      text height=2.75mm,
      text depth=.75mm,
      minimum height=3.5mm,
      minimum width=3.5mm
    },
  },
  shapetableaumatrix/.style={
    ampersand replacement=\&,
    matrix of math nodes,
    outer sep=0mm,
    inner sep=0mm,
    anchor=base,
    row sep={between borders,-\pgflinewidth},
    column sep={between borders,-\pgflinewidth},
    execute at begin cell={\strut},
    every node/.append style={draw,anchor=base,text height=1mm,text depth=.5mm,minimum size=1.5mm,inner sep=0mm,outer sep=0mm},
  },
  shapetableau/.style={
    every matrix/.style={shapetableaumatrix},
  },
  topalign/.style={
    every matrix/.append style={name=maintableau,anchor=maintableau-1-1.base},
    baseline,
  },
}

\newcommand*\tableau[2][]{\tikz[tableau,#1]\matrix{#2};}
\newcommand*\smalltableau[2][]{\tikz[smalltableau,#1]\matrix{#2};}
\theoremstyle{definition}
\newtheorem{definition}{Definition}[section]

\newtheorem{algorithm}[definition]{Algorithm}
\newtheorem{conjecture}[definition]{Conjecture}
\newtheorem{example}[definition]{Example}
\newtheorem{method}[definition]{Method}

\newtheorem{question}[definition]{Question}

\theoremstyle{plain}

\newtheorem{lemma}[definition]{Lemma}
\newtheorem{proposition}[definition]{Proposition}
\newtheorem{theorem}[definition]{Theorem}

\numberwithin{equation}{section}
\newcommand*{\textparens}[1]{\textup{(}#1\textup{)}}

\newcommand*{\defterm}[1]{\emph{#1}}

\makeatletter
\newcommand\chyph{\penalty\@M-\hskip\z@skip}
\makeatother

\DeclarePairedDelimiter{\parens}{\lparen}{\rparen}

\DeclarePairedDelimiter{\set}{\{}{\}}
\DeclarePairedDelimiterX{\gset}[2]{\{}{\}}{\,#1:#2\,}

\makeatletter
\newcommand{\sizeddelimiter}[2]{\bBigg@{#1}#2}
\makeatother

\newcommand*{\nset}{\mathbb{N}}

\newcommand*{\emptyword}{\varepsilon}

\DeclarePairedDelimiter{\gen}{\langle}{\rangle}

\DeclarePairedDelimiterX{\pres}[2]{\langle}{\rangle}{#1\,\delimsize\vert\,\mathopen{}#2}

\newcommand*{\drel}[1]{\mathcal{#1}}
\newcommand*{\cgen}[1]{#1^{\#}}

\newcommand*{\aA}{\mathcal{A}}

\newcommand*{\evlit}{{\mathrm{ev}}}
\newcommand*{\ev}[2][]{\evlit\parens[#1]{#2}}
\newcommand*{\evrel}{\equiv_\evlit}

\newcommand*{\cochseq}{\mathrm{cochseq}}

\newcommand*{\tree}[1]{#1}

\newcommand*{\plac}{{\mathsf{plac}}}
\newcommand*{\hypo}{{\mathsf{hypo}}}
\newcommand*{\sylv}{{\mathsf{sylv}}}

\newcommand*{\baxt}{{\mathsf{baxt}}}
\newcommand*{\stal}{{\mathsf{stal}}}
\newcommand*{\taig}{{\mathsf{taig}}}

\newcommand*{\placcong}{\equiv_\plac}
\newcommand*{\hypocong}{\equiv_\hypo}
\newcommand*{\stalcong}{\equiv_\stal}
\newcommand*{\sylvcong}{\equiv_\sylv}

\newcommand*{\taigcong}{\equiv_\taig}
\newcommand*{\baxtcong}{\equiv_\baxt}

\newcommand*{\Plit}{\mathrm{P}}

\renewcommand*{\P}[2][]{\Plit\parens[#1]{#2}}

\newcommand*{\colreading}[2][]{\mathrm{C}\parens[#1]{#2}}
\newcommand*{\rowreading}[2][]{\mathrm{R}\parens[#1]{#2}}

\newcommand*{\plit}{\mathrm{P}}

\newcommand*{\pplac}[2][]{\plit_{\plac}\parens[#1]{#2}}

\newcommand*{\phypo}[2][]{\plit_{\hypo}\parens[#1]{#2}}

\newcommand*{\psylv}[2][]{\plit_{\sylv}\parens[#1]{#2}}

\newcommand*{\pbaxt}[2][]{\plit_{\baxt}\parens[#1]{#2}}

\newcommand*{\ptaig}[2][]{\plit_{\taig}\parens[#1]{#2}}

\newcommand*{\pstal}[2][]{\plit_{\stal}\parens[#1]{#2}}

\tikzset{
  olsubword/.style={
    every node/.append style={
      name=mainnode,
      draw=darkgray,
      rounded corners=.5mm,
      inner sep=.5mm,
    },
    baseline=(mainnode.base),
  }
}

\newcommand*\olsubword[1]{\tikz[olsubword]\node{#1};}

\tikzset{
  mogrifyarrow/.style={
    ->,
    >/.tip=Computer Modern Rightarrow,
    decorate,
    decoration={
      zigzag,
      amplitude=0.2em,
      segment length=0.35em,
      pre length=0.35em,
      post length=0.35em,
    },
  },
}

\tikzset{
  bstoutline/.style={darkgray,thick}
}

\tikzset{
  lcomment/.style={
    align=right,
    font=\scriptsize,
    inner sep=.5mm,
    anchor=east,
  },
  rcomment/.style={
    align=left,
    font=\scriptsize,
    inner sep=.5mm,
    anchor=west,
  },
  tcomment/.style={
    align=center,
    font=\scriptsize,
    inner sep=.5mm,
    anchor=south,
  },
}

\newcommand{\cyc}{\sim}

\begin{document}

\title[Combinatorics of cyclic shifts]{Combinatorics of cyclic shifts in plactic, hypoplactic, sylvester, Baxter, and related monoids}

\author[A.J. Cain]{Alan J. Cain}
\address{%
Centro de Matem\'{a}tica e Aplica\c{c}\~{o}es (CMA)\\
Faculdade de Ci\^{e}ncias e Tecnologia\\
Universidade Nova de Lisboa\\
2829--516 Caparica\\
Portugal
}
\email{%
a.cain@fct.unl.pt
}

\thanks{The first author was supported by an Investigador {\sc FCT} fellowship ({\sc IF}/01622/2013/{\sc CP}1161/{\sc
    CT}0001).}

\author[A. Malheiro]{Ant\'{o}nio Malheiro}
\address{%
Centro de Matem\'{a}tica e Aplica\c{c}\~{o}es (CMA)\\
Faculdade de Ci\^{e}ncias e Tecnologia\\
Universidade Nova de Lisboa\\
2829--516 Caparica\\
Portugal
}
\address{%
Departamento de Matem\'{a}tica\\
Faculdade de Ci\^{e}ncias e Tecnologia\\
Universidade Nova de Lisboa\\
2829--516 Caparica\\
Portugal
}
\email{%
ajm@fct.unl.pt
}

\thanks{For both authors, this work was partially supported by by the Funda\c{c}\~{a}o para
a Ci\^{e}ncia e a Tecnologia (Portuguese Foundation for Science and Technology) through the project {\sc UID}/{\sc
  MAT}/00297/2013 (Centro de Matem\'{a}tica e Aplica\c{c}\~{o}es) and the project {\scshape PTDC}/{\scshape
      MHC-FIL}/2583/2014.}

\begin{abstract}
  The cyclic shift graph of a monoid is the graph whose vertices are elements of the monoid and whose edges link
  elements that differ by a cyclic shift. This paper examines the cyclic shift graphs of `plactic-like' monoids, whose
  elements can be viewed as combinatorial objects of some type: aside from the plactic monoid itself (the monoid of
  Young tableaux), examples include the hypoplactic monoid (quasi-ribbon tableaux), the sylvester monoid (binary search
  trees), the stalactic monoid (stalactic tableaux), the taiga monoid (binary search trees with multiplicities), and the
  Baxter monoid (pairs of twin binary search trees). It was already known that for many of these monoids, connected
  components of the cyclic shift graph consist of elements that have the same evaluation (that is, contain the same
  number of each generating symbol). This paper focusses on the maximum diameter of a connected component of the cyclic
  shift graph of these monoids in the rank-$n$ case. For the hypoplactic monoid, this is $n-1$; for the sylvester and
  taiga monoids, at least $n-1$ and at most $n$; for the stalactic monoid, $3$ (except for ranks $1$ and $2$, when it is
  respectively $0$ and $1$); for the plactic monoid, at least $n-1$ and at most $2n-3$. The current state of knowledge,
  including new and previously-known results, is summarized in a table.
\end{abstract}

\maketitle

\tableofcontents

\section{Introduction}

In a monoid $M$, two elements $s$ and $t$ are related by a cyclic shift, denoted $s \sim t$, if and only if there exist
$x,y \in M$ such that $s = xy$ and $t = yx$. In the plactic monoid (the monoid of Young tableaux, here denoted $\plac$),
elements that have the same evaluation (that is, elements that contain the same number of each generating symbol) can be
obtained from each other by iterated application of cyclic shifts \cite[\S~4]{lascoux_plaxique}. Furthermore, in the
plactic monoid of rank $n$ (denoted $\plac_n$), it is known that $2n-2$ applications of cyclic shifts are sufficient
\cite[Theorem~17]{choffrut_lexicographic}.

To restate these results in a new form, define the \defterm{cyclic shift graph} $K(M)$ of a monoid $M$ to be the
undirected graph with vertex set $M$ and, for all $s,t \in M$, an edge between $s$ and $t$ if and only if $s \sim t$.
Connected components of $K(M)$ are $\sim^*$-classes (where $\sim^*$ is the reflexive and transitive closure of $\sim$),
since they consist of elements that are related by iterated cyclic shifts. Thus the results discussed above say that
each connected component of $K(\plac)$ consists of precisely the elements with a given evaluation, and that the diameter
of a connected component of $K(\plac_n)$ is at most $2n-2$. Note that there is no bound on the number of elements in a
connected component, despite there being a bound on diameters that is dependent only on the rank.

This paper studies the cyclic shift graph for analogues of the plactic monoid in which other combinatorial objects have
the role that Young tableaux play for the plactic monoid. For each monoid there are two central questions, motivated by
the results for the plactic monoid: (i) whether connected components consist of precisely the elements with a given
evaluation; (ii) what the maximum diameter of a connected component is in the rank $n$ case.

The monoids considered are the plactic monoid, which is celebrated for its ubiquity, arising in such diverse contexts as
symmetric functions \cite{macdonald_symmetric}, representation theory and algebraic combinatorics
\cite{fulton_young,lothaire_algebraic}, and musical theory \cite{jedrzejewski_plactic}; the hypoplactic monoid, whose
elements are quasi-ribbon tableaux and which arises in the theory of quasi-symmetric functions
\cite{krob_noncommutative4,krob_noncommutative5,novelli_hypoplactic}; the sylvester monoid, whose elements are binary
search trees \cite{hivert_sylvester}; the taiga monoid, whose elements are binary search trees with multiplicities
\cite{priez_lattice}; the stalactic monoid, whose elements are stalactic tableaux
\cite{hivert_commutative,priez_lattice}; and the Baxter monoid, whose elements are pairs of twin binary search trees
\cite{giraudo_baxter,giraudo_baxter2}, and which is linked to the theory of Baxter permutations. (See
\fullref{Table}{tbl:monoids}.) Each of these monoids arises by factoring the free monoid $\aA^*$ over the ordered
alphabet $\aA = \set{1 < 2 < \ldots}$ by a congruence $\equiv$ that can be defined in two equivalent ways:
\begin{itemize}
\item[C1] \textit{Insertion.} $\equiv$ relates those words that yield the same combinatorial object as the result of
  some insertion algorithm.
\item[C2] \textit{Defining relations.} $\equiv$ is defined to be the congruence generated by some set of defining
  relations $\drel{R}$.
\end{itemize}
Each of these monoids also has a rank-$n$ version (where $n \in \nset$), which arises by factoring the free monoid
$\aA_n^*$ over the finite ordered alphabet $\aA_n = \set{1 < 2 < \ldots < n}$ by the natural restriction of
$\equiv$. Each of these monoids is discussed in its own section, and the equivalent definitions will be recalled at the
start of the relevant section. For the present, note that these monoids are \emph{multihomogeneous}: if two words
over $\aA^*$ represent the same element of the monoid (that is, are related by $\equiv$) then they have the same
evaluation (that is, they contain the same number of each symbol in $\aA$ and, in particular, have the same
length). Thus it is sensible to consider the evaluation of an \emph{element} of the monoid to be the evaluation of any
word that represents it. The relation $\evrel$ holds between elements that have the same evaluation; clearly
$\evrel$ is an equivalence relation and $\evrel$-classes are finite.

\begin{table}[t]
  \centering
  \caption{Monoids and corresponding combinatorial objects.}
  \label{tbl:monoids}
  \begin{tabular}{llll}
    \toprule
    \textit{Monoid} & \textit{Sym.} & \textit{Combinatorial object}          & \textit{See}              \\
    \midrule
    Plactic         & $\plac$         & Young tableau                          & \fullref{\S}{sec:plactic}     \\
    Hypoplactic     & $\hypo$         & Quasi-ribbon tableau                   & \fullref{\S}{sec:hypoplactic} \\
    Sylvester       & $\sylv$         & Binary search tree                     & \fullref{\S}{sec:sylvester}   \\
    Stalactic       & $\stal$         & Stalactic tableau                      & \fullref{\S}{sec:stalactic}   \\
    Taiga           & $\taig$         & Binary search tree with multiplicities & \fullref{\S}{sec:taiga}   \\
    Baxter          & $\baxt$         & Pair of twin binary search trees       & \fullref{\S}{sec:baxter}      \\
    \bottomrule
  \end{tabular}
\end{table}

This paper shows that in the cyclic shift graph of the hypoplactic, sylvester, and taiga monoids, each connected component
does consist of precisely those elements with a given evaluation. In the case of the stalactic monoid, an alternative
characterization is given of when two elements lie in the same connected component. (The result for the sylvester monoid was
previously proved in a different way by the present authors \cite[Theorem~3.4]{cm_conjugacy}.)

Furthermore, just as for the rank-$n$ plactic monoid, there is a bound on the maximum diameters of connected components
in the cyclic shift graphs of the rank-$n$ hypoplactic, sylvester, and taiga monoids, and this bound is only dependent on
$n$. It is worth emphasizing how remarkable this is: although there is no global bound on the number of elements in a
component (or, equivalently, which have the same evaluation), any two elements in the same component are related by a
number of cyclic shifts that is dependent only on $n$. \fullref{Table}{tbl:components} shows the current state of
knowledge for all the monoids considered in this paper. All the exact values and bounds shown in this table are new
results (although the upper bound of $2n-3$ in the case of the plactic monoid follows by a minor modification of the
reasoning that yields the Choffrut--Merca\c{s} bound of $2n-2$).

Experimentation using the computer algebra software Sage \cite{sage} strongly suggests that in the cases of the rank-$n$
plactic, sylvester, and taiga monoids, the maximum diameter of a connected component is $n-1$.

Also, although the monoids considered are all multihomogeneous, \fullref{Section}{sec:multihomogeneous} exhibits a
rank~$4$ multihomogeneous monoid for which there is no bound on the diameter of connected components in its cyclic shift
graph. Thus, the bound on diameters is \emph{not} a general property of multihomogeneous monoids: rather, it seems to
dependent on the underlying combinatorial objects.  This also is of interest because cyclic shifts are a possible
generalization of conjugacy from groups to monoids; thus the combinatorial objects are here linked closely to the
algebraic structure of the monoid.

\begin{table}[t]
  \centering
  \caption{Maximum diameter of a connected component of cyclic shift graph for rank-$n$ monoids.}
  \label{tbl:components}
  \begin{tabular}{lccccc}
    \toprule
                    &                       & \multicolumn{4}{c}{\textit{Maximum diameter}}                                          \\
    \cmidrule(lr){3-6}
                    &                       &                      &                     & \multicolumn{2}{c}{\textit{Known bounds}} \\
    \cmidrule(lr){5-6}
    \textit{Monoid} & ${\cyc^*} = {\evrel}$ & \textit{Known value} & \textit{Conjecture} & \textit{Lower} & \textit{Upper}           \\
    \midrule
    $\plac_n$       & Y                     & ?                    & $n-1$               & $n-1$          & $2n-3$                   \\
    $\hypo_n$       & Y                     & $n-1$                & ---                 & ---            & ---                      \\
    $\sylv_n$       & Y                     & ?                & $n-1$               & $n-1$          & $n$                      \\
    $\stal_n$       & N                     & $\begin{cases} n-1 & \text{if $n < 3$} \\ 3 & \text{if $n \geq 3$}\end{cases}$    & --- & --- & ---                      \\
    $\taig_n$       & Y                     & ?                    & $n-1$               & $n-1$          & $n$                      \\
    $\baxt_n$       & N                     & ?                    & ?                   & ?              & ?                        \\
    \bottomrule
  \end{tabular}
\end{table}

[The results in this article have already been announced in the conference paper \cite{cm_cyclicshifts1}.]

\section{Preliminaries}

\subsection{Alphabets and words}

This subsection recalls some terminology and fixes notation for presentations. For background on the free monoid, see
\cite{howie_fundamentals}; for semigroup presentations, see \cite{higgins_techniques,ruskuc_phd}.

For any alphabet $X$, the free monoid (that is, the set of all words, including the empty word) on the
alphabet $X$ is denoted $X^*$. The empty word is denoted $\emptyword$. For any $u \in X^*$, the length of $u$ is denoted
$|u|$, and, for any $x \in X$, the number of times the symbol $x$ appears in $u$ is denoted $|u|_x$.

The \defterm{evaluation} (also called the \defterm{content}) of a word $u \in X^*$, denoted $\ev{u}$, is the $|X|$-tuple
of non-negative integers, indexed by $X$, whose $x$-th element is $|u|_x$; thus this tuple describes the number of each
symbol in $X$ that appears in $u$. If two words $u,v \in X^*$ have the same evaluation, this is denoted $u \evrel v$.
Notice that $\evrel$ is an equivalence relation (and indeed a congruence). Note further that there are clearly only
finitely many words with a given evaluation, and so $\evrel$-classes are finite.

When $X$ represents a generating set for a monoid $M$, every element of $X^*$ can be interpreted either as a word or as
an element of $M$. For words $u,v \in X^*$, write $u=v$ to indicate that $u$ and $v$ are equal as words and
$u \equiv_M v$ to denote that $u$ and $v$ represent the same element of the monoid $M$.  A presentation is a pair
$\pres{X}{\drel{R}}$ where $\drel{R}$ is a binary relation on $X^*$, which defines [any monoid isomorphic to]
$X^*/\cgen{\drel{R}}$, where $\cgen{\drel{R}}$ denotes the congruence generated by $\drel{R}$.

The presentation $\pres{X}{\drel{R}}$ is \defterm{homogeneous} (respectively, \defterm{multihomogeneous}) if for every
$(u,v) \in \drel{R}$ we have $|u| = |v|$ (respectively, $u \evrel v$). That is, in a homogeneous presentation,
defining relations preserve length of words; in a multihomogenous presentation, defining relations preserve
evaluations. A monoid is \defterm{homogeneous} (respectively, \defterm{multihomogeneous}) if it admits a homogeneous
(respectively, multihomogeneous) presentation. Suppose $M$ is a multihomogeneous monoid defined by a multihomogeneous
presentation $\pres{X}{\drel{R}}$. Since every word in $X^*$ that represents a given element of $M$ has the same
evaluation, it makes sense to define the evaluation of an element of $M$ to be the evaluation of any word representing
it, and to write $s \evrel t$ if $s,t \in M$ have the same evaluation.

\subsection{`Plactic-like' monoids}

Throughout the paper, $\aA$ denotes the infinite ordered alphabet $\set{1 < 2 < \ldots}$ (that is, the set of natural
numbers, viewed as an alphabet), $\aA_n$ the finite ordered alphabet $\set{1 < 2 < \ldots < n}$ (that is, the first $n$
natural numbers, viewed as an alphabet). A word $u \in \aA_n^*$ is \defterm{standard} if $u \evrel 123\cdots |u|$. That
is, $u$ is standard if it contains each symbol in $\set{1,\ldots,|u|}$ exactly once.

This paper is mainly concerned with `plactic-like' monoids, whose elements can be identified with some kind of
combinatorial object. Each such monoid $\mathsf{M}$ has an associated insertion algorithm, which takes a combinatorial
object of the relevant type and a letter of the alphabet $\aA$ and computes a new combinatorial object. Thus one can
compute from a word $u \in \aA^*$ a combinatorial object $\plit_{\mathsf{M}}(u)$ of the type associated to $\mathsf{M}$ by starting with
the empty combinatorial object and inserting the symbols of $u$ one-by-one using the appropriate insertion algorithm and
proceeding through the word $u$ either left-to-right or right-to-left. (The procedure is slightly different for the
Baxter monoid; this will be discussed in \fullref{Section}{sec:baxter}.) One then defines a relation
$\equiv_{\mathsf{M}}$ as the kernel of the map $u \mapsto \plit_{\mathsf{M}}(u)$. In each case, the relation
$\equiv_{\mathsf{M}}$ is a congruence, and $\mathsf{M}$ is the factor monoid $\aA^*\!/{\equiv_{\mathsf{M}}}$; the rank-$n$
analogue is the factor monoid $\aA_n^*/{\equiv_{\mathsf{M}}}$, where $\equiv_{\mathsf{M}}$ is naturally restricted to
$\aA_n^* \times \aA_n^*$. Since each element of $\mathsf{M}$ is an equivalence class of words that give the same
combinatorial object, elements of $\mathsf{M}$ can be identified with the corresponding combinatorial objects.

Each of the combinatorial objects and insertion algorithms considered in this paper is such that the number of each
symbol from $\aA$ in the word $u$ is the same as the number of symbols $\aA$ in the object $\plit_{\mathsf{M}}(u)$. It
follows that each of the corresponding monoids is multihomogeneous, and it makes sense to define an element of a
rank-$n$ monoid to be \defterm{standard} if it is represented by a standard word (and thus \emph{only} by standard
words).

\subsection{Cyclic shifts}

Recall that two elements $s,t \in M$ are related by a cyclic shift, denoted $s \cyc t$, if and only if there exist
$x,y \in M$ such that $s = xy$ and $t = yx$. If $X$ represents a generating set for $M$, then $s \cyc t$ if and only if
there exist $u,v \in X^*$ such that $uv$ represents $s$ and $vu$ represents $t$. Notice that the relation $\cyc$ is
reflexive (because it is possible that $x$ or $y$ in the definition of $\cyc$ can be the identity) and symmetric. For
$k \in \nset$, let $\cyc^k$ be the $k$-fold composition of the relation $\cyc$: that is,
\[
\cyc^k = \underbrace{\cyc \circ \cyc \circ \ldots \circ \cyc}_{\text{$k$ times}}.
\]
Note that $\cyc^*$, the transitive closure of $\cyc$, is $\bigcup_{k=1}^\infty \cyc^k$. Note further that
${\cyc^k} \subseteq {\cyc^{k+1}}$ since $\cyc$ is reflexive. Thus $\cyc^{k}$ relates elements of $M$ that differ by
at most $k$ cyclic shifts.

The following result is immediate from the definition of $\cyc$:

\begin{lemma}
  \label{lem:basicprops}
  In any multihomogeneous monoid, ${\cyc^*} \subseteq {\evrel}$.
\end{lemma}

For any monoid $M$, define the \defterm{cyclic shift graph} $K(M)$ to be the undirected graph with vertex set $M$ and,
for all $s,t \in M$, an edge between $s$ and $t$ if and only if $s \cyc t$. (See \cite{harary_graph} for
graph-theoretical definitions and terminology.) Two elements $s,t \in M$ are a distance at most $k$ apart in $K(M)$ if
and only if $s \cyc^k t$. Connected components of $K(M)$ are $\cyc^*$-classes, since they consist of elements that are
related by iterated cyclic shifts. The connected component of $K(M)$ containing an element $s \in M$ is denoted
$K(M,s)$. If $M$ is multihomogenous, ${\cyc^*} \subseteq {\evrel}$, and thus connected components of $K(M)$ are finite.

Since $\cyc$ is reflexive, there is a loop at every vertex of $K(M)$. Throughout the paper, illustrations of graphs
$K(M)$ will, for clarity, omit these loops.

\subsection{Cocharge sequences}
\label{subsec:cocharge}

This subsection introduces `cocharge sequences', which will be used in several places to establish lower bounds on the
maximum diameters of connected components of the cyclic shift graphs of finite-rank monoids.

Let $u \in \aA^*$ be a standard word. The \defterm{cocharge sequence} of $u$, denoted $\cochseq(u)$, is a sequence (of
length $|u|$) calculated from $u$ as follows:
\begin{enumerate}
\item Draw a circle, place a point $*$ somewhere on its circumference, and, starting from $*$, write $u$ anticlockwise around the circle.
\item Label the symbol $1$ with $0$.
\item Iteratively, after labelling some $i$ with $k$, proceed clockwise from $i$ to the symbol $i+1$:
  \begin{itemize}
  \item if the symbol $i+1$ is reached \emph{before} $*$, label $i+1$ by $k+1$;
  \item if the symbol $i+1$ is reached \emph{after} $*$, label $i+1$ by $k$.
  \end{itemize}
\item The sequence $\cochseq(u)$ is the sequence whose $i$-th term is the label of $i$.
\end{enumerate}

Note that at steps 2 and 3, the symbols $1$ and $i+1$ are known to be in $u$ because $u$ is a standard word.

For example, for the word $1246375$, the labelling process gives:
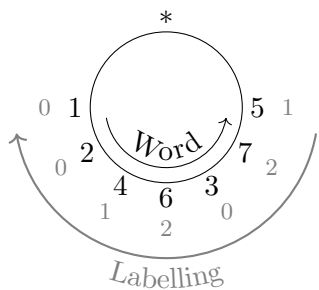
\[
\begin{tikzpicture}
  \draw (0,0) circle[radius=10mm];
  \draw[->] (-170:8mm) arc[radius=8mm,start angle=-170,end angle=-10];
  \draw[decorate,decoration={text effects along path,text={Word},text color=gray,text align={align=center}},text effects={text along path,character widths={inner xsep =.2mm}}] (-170:7mm) arc[radius=7mm,start angle=-170,end angle=-10];
  \draw[gray,thick,->] (-10:20mm) arc[radius=20mm,start angle=-10,end angle=-170];
  \draw[gray,decorate,decoration={text along path,text={Labelling},text color=gray,text align={align=center}}] (-170:24mm) arc[radius=24mm,start angle=-170,end angle=-10];
  \foreach\i/\ilabel in {0/*,9/1,8/2,7/4,6/6,5/3,4/7,3/5} {
    \node at ($ (90-\i*30:12mm) $) {$\ilabel$};
  };
  \foreach\i/\ilabel in {9/0,8/0,5/0,7/1,3/1,6/2,4/2} {
    \node[font=\footnotesize,gray] at ($ (90-\i*30:16mm) $) {$\ilabel$};
  };
\end{tikzpicture}
\]
and it follows that $\cochseq(u) = (0,0,0,1,1,2,2)$. Notice that the first term of a cocharge sequence is always $0$,
and that each term in the sequence is either the same as its predecessor or greater by $1$. Thus the $i$-th term in the
sequence always lies in the set $\set{0,1,\ldots,i-1}$.

The usual notion of cocharge is obtained by summing the cocharge sequence (see \cite[\S~5.6]{lothaire_algebraic}). Note,
however, that cocharge is defined for all words, whereas this section defines the cocharge sequence only for standard words.

\begin{lemma}
  \label{lem:cochseqcycles}
  \begin{enumerate}
  \item Let $u \in \aA^*$ and $a \in \aA_n \setminus \set{1}$ be such that $ua$ is a standard word. Then $\cochseq(ua)$ is
    obtained from $\cochseq(au)$ by adding $1$ to the $a$-th component.
  \item Let $xy \in \aA^*$ be a standard word such that $x$ does not contain the symbol $1$. Then $\cochseq(yx)$ is
    obtained from $\cochseq(xy)$ by adding $1$ to the $a$-th component for each symbol $a$ that appears in $x$.
  \item Let $xy \in \aA^*$ be a standard word such that $y$ does not contain the symbol $1$. Then $\cochseq(yx)$ is
    obtained from $\cochseq(xy)$ by subtracting $1$ from the $a$-th component for each symbol $a$ that appears in $y$.
  \end{enumerate}
\end{lemma}

\begin{proof}
  Consider how $a$ is labelled during the calculation of $\cochseq(ua)$ and $\cochseq(au)$:
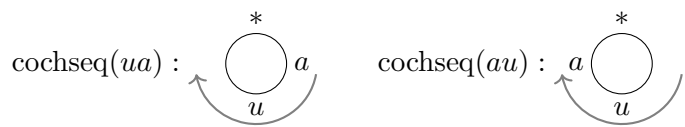
  \[
  \cochseq(ua):
  \begin{tikzpicture}[baseline=-1mm]
    \draw (0,0) circle[radius=4mm];
    \draw[gray,thick,->] (-10:8mm) arc[radius=8mm,start angle=-10,end angle=-170];
    %
    \foreach\i/\ilabel in {0/*,6/u,3/a} {
      \node at ($ (90-\i*30:6mm) $) {$\ilabel$};
    };
  \end{tikzpicture}
  \qquad
  \cochseq(au):
  \begin{tikzpicture}[baseline=-1mm]
    \draw (0,0) circle[radius=4mm];
    \draw[gray,thick,->] (-10:8mm) arc[radius=8mm,start angle=-10,end angle=-170];
    %
    \foreach\i/\ilabel in {0/*,9/a,6/u} {
      \node at ($ (90-\i*30:6mm) $) {$\ilabel$};
    };
  \end{tikzpicture}
  \]
  In the calculation of $\cochseq(ua)$, the symbol $a-1$ receives a label $k$, and then $a$ is reached \emph{after} $*$
  is passed; hence $a$ also receives the label $k$. (If a symbol $a+1$ is present, it receives the label $a+1$.) In the
  calculation of $\cochseq(au)$, the symbols $1,\ldots,a-1$ receive the same labels as they do in the calculation of
  $\cochseq(ua)$, but after labelling $a-1$ by $k$ the symbol $a$ is reached \emph{before} $*$ is passed; hence $a$
  receives the label $k+1$ (and if a symbol $a+1$ is present, it also receives the label $k+1$ since it is reached after
  $*$ is passed); after this point, labelling proceeds in the same way. Parts~2) and~3) are now immediate consequences
  of part~1).
\end{proof}

\section{General multihomogeneous monoids}
\label{sec:multihomogeneous}

In order to set in context the results below on the diameters of connected components of cyclic shift graphs, this
section gives an example of a multihomogeneous monoid for which the connected components of the cyclic shift graph have
unbounded diameter. This shows that the results for the `plactic-like' monoids discussed in the rest of the paper are
not simply consequences of some more general result that holds for all multihomogeneous monoids.

\begin{example}
  Let $M$ be the monoid defined by the presentation
  \[
  \pres{a,b,x,y}{(bxy,xyb),(byx,yxb),(axyb,byxa)}.
  \]
  Notice that $M$ is multihomogeneous. Let $\alpha \in \nset$. Then
  \begin{align*}
    a(xy)^\alpha b \equiv_M{}& axyb(xy)^{\alpha-1} \\
    \equiv_M{}& byxa(xy)^{\alpha-1} \displaybreak[0]\\
    \cyc{}& yxa(xy)^{\alpha-1}b \displaybreak[0]\\
    \equiv_M{}& yxaxyb(xy)^{\alpha-2} \displaybreak[0]\\
    \equiv_M{}& yxbyxa(xy)^{\alpha-2} \displaybreak[0]\\
    \equiv_M{}& b(yx)^2a(xy)^{\alpha-2} \displaybreak[0]\\
    \cyc{}& (yx)^2a(xy)^{\alpha-2}b \\
    & \quad\vdots \\
    \cyc{}& b(yx)^{\alpha}a
  \end{align*}
  Thus the elements $a(xy)^\alpha b$ and $b(yx)^\alpha a$ lie in the same connected component of $K(M)$. The aim is now
  to prove that the distance between $a(xy)^\alpha b$ and $b(yx)^\alpha a$ in $K(M)$ is at least $\alpha$. This
  necessitates defining an invariant, reminiscent of cocharge, but tailored specifically to the monoid $M$.

  Let $L = \gset[\big]{u \in \set{a,b,xy,yx}^*}{|u|_a = |u|_b = 1}$. Define a map
  $\mu : L \to \nset\cup\set{0}$, where $\mu(u)$ is calculated as follows:
  \begin{enumerate}
  \item Draw a circle, place a point $*$ somewhere on itse circumference, and write $u$ anticlockwise aroung the circle.
  \item Temporarily ignoring the symbol $b$, let $k(u)$ be the number of consecutive cyclic factors $xy$ following
    $a$. (Equivalently, let $k(u)$ be the number of consecutive words $xy$ (ignoring $b$) following the symbol $a$,
    proceeding anticlockwise around the circle.)
  \item Let $\mu(u)$ be $k(u)$ if starting from $a$ and proceeding anticlockwise, one encounters $b$ before $*$, and
    otherwise let $\mu(u)$ be $k(u)+1$.
  \end{enumerate}
  It is now necessary to show that $\mu(\cdot)$ is invariant under applications of defining relations. Clearly, the set
  $L$ is closed under applying defining relations. Further, applying a defining relation $(bxy,xyb)$ or $(byx,yxb)$ does
  not alter $k(\cdot)$, and does not alter the relative positions of $a$, $b$, and $*$. So applying these defining
  relations does not alter $\mu(\cdot)$. So suppose that $u,v \in L$ differ by a single application of a defining relation
  $(axyb,byxa)$. Interchanging $u$ and $v$ if neceessary, suppose $u = paxybq$ and $v = pbyxaq$. Let
  $m \in \nset\cup\set{0}$ be maximal such that $(xy)^m$ is a prefix of $qp$; thus either $qp = (xy)^m$ or
  $qp = (xy)^myxw$ for some $w \in \set{xy,yx}^*$. Applying the procedure above to $u$ and $v$, one sees that
  $k(u) = m + 1 = k(v) + 1$, but $b$ is encountered before $*$ for $u$ but not for $v$; hence
  $\mu(u) = k(u) = k(v) + 1 = \mu(v)$.

  Thus if two words in $L$ represent the same element of $M$, they have the same image under $\mu$. If $u,v \in L$ are
  such that $u \cyc v$, then $k(u) = k(v)$, since applying $\cyc$ does not alter the number of cyclic factors $xy$
  following $a$. Thus $\mu(u)$ and $\mu(v)$ differ by at most one. Hence, since $\mu(a(xy)^\alpha b) = \alpha$ and
  $\mu(b(yx)^\alpha a) = 1$, it follows that $a(xy)^\alpha b$ and $b(yx)^\alpha a$ are a distance at least $\alpha-1$
  apart in $K(M)$.

  Since $\alpha$ was arbitrary, this shows that there is no bound on the diameters of connected components in $K(M)$.
\end{example}

\section{Plactic monoid}
\label{sec:plactic}

This section recalls the essential facts about the plactic monoid; for proofs and further reading, see
\cite[Ch.~5]{lothaire_algebraic}.

A \defterm{Young tableau} is an array with rows left aligned and of non-increasing length from top to bottom, filled
with symbols from $\aA$ so that the entries in each row are non-decreasing from left to right, and the entries in each
column are [strictly] increasing from top to bottom. An example of a Young tableau is
\begin{equation}
\label{eq:youngtableaueg}
\tableau{
1 \& 2 \& 2 \& 2 \& 4 \\
2 \& 3 \& 5 \\
4 \& 4 \\
5 \& 6 \\
}.
\end{equation}

The associated insertion algorithm is as follows:

\begin{algorithm}[Schensted's algorithm]
\label{alg:placinsertone}
~\par\nobreak
\textit{Input:} A Young tableau $T$ and a symbol $a \in \aA$.

\textit{Output:} A Young tableau $T \leftarrow a$.

\textit{Method:}
\begin{enumerate}

\item If $a$ is greater than or equal to every entry in the topmost row of $T$, add $a$ as an entry at the rightmost end of
  $T$ and output the resulting tableau.

\item Otherwise, let $z$ be the leftmost entry in the top row of $T$ that is strictly greater than $a$. Replace $z$ by
  $a$ in the topmost row and recursively insert $z$ into the tableau formed by the rows of $T$ below the topmost. (Note
  that the recursion may end with an insertion into an `empty row' below the existing rows of $T$.)

\end{enumerate}
\end{algorithm}

Thus one can compute, for any word $u = u_1\cdots u_k \in \aA^*$, a Young tableau $\pplac{u}$ by starting with an empty
tableau and successively inserting the symbols of $u$, proceeding left-to-right through the word. Define the relation
$\placcong$ by
\[
u \placcong v \iff \pplac{u} = \pplac{v}
\]
for all $u,v \in \aA^*$. The relation $\placcong$ is a congruence, and the \defterm{plactic monoid}, denoted $\plac$, is
the factor monoid $\aA^*\!/{\placcong}$; the \defterm{plactic monoid of rank $n$}, denoted $\plac_n$, is the factor
monoid $\aA_n^*/{\placcong}$ (with the natural restriction of $\placcong$). Each element $[u]_{\placcong}$ (where
$u \in \aA^*$) can be identified with the Young tableau $\pplac{u}$.

The monoid $\plac$ is presented by
$\pres{\aA}{\drel{R}_\plac}$, where
\begin{align*}
\drel{R}_\plac ={}& \gset{(acb,cab)}{a,b,c \in \aA, a \leq b < c} \\
&\cup \gset{(bac, bca)}{a,b,c \in \aA, a < b \leq c};
\end{align*}
the defining relations in $\drel{R}_\plac$ are often called the \defterm{Knuth relations}
\cite[\S~5.2]{lothaire_algebraic}. The monoid $\plac_n$ is presented by $\pres{\aA_n}{\drel{R}_\plac}$, where the set of
defining relations $\drel{R}_\plac$ is naturally restricted to $\aA_n^*\times \aA_n^*$. Notice in particular that
$\plac$ and $\plac_n$ are multihomogeneous.

Lascoux \& Sch\"{u}tzenberger proved that if two elements of $\plac$ have the same evaluation, then it is possible to
transform one to the other using cyclic shifts \cite[\S~4]{lascoux_plaxique}. In the terms of this paper, they proved
that ${\evrel} \subseteq {\cyc^*}$ in $\plac$. Since the opposite inclusion holds in general, it follows that
${\evrel} = {\cyc^*}$ in $\plac$. Thus connected components of the cyclic shift graph $K(\plac)$ are $\evrel$-classes.

Choffrut \& Merca\c{s} showed that if two elements of $\plac_n$ have the same evaluation, then it is possible to
transform one to the other using at most $2n-2$ cyclic shifts \cite[Theorem~17]{choffrut_lexicographic}. Thus the maximum
diameter of a connected component of $K(\plac_n)$ is at most $2n-2$.

This section gives a lower bound on the maximum diameter of a connected component of $K(\plac_n)$, makes a slight
improvement on the upper bound of Choffrut \& Merca\c{s}, and conjecture the exact value on the basis of experimentation
using computer algebra software.

Establish the lower bound requires the use of cocharge sequences (defined in \fullref{Subsection}{subsec:cocharge}
above). It is first of all necessary to show that cocharge sequences are well-defined on standard elements of $\plac$:

\begin{proposition}
  \label{prop:cochseqplactic}
  Let $w,w' \in \aA^*$ be standard words such that $w \placcong w'$. Then $\cochseq(w) = \cochseq(w')$.
\end{proposition}

\begin{proof}
  It suffices to prove the result when $w$ and $w'$ differ by a single application of a defining relation. Assume that
  $w$ and $w'$ differ by an application of a defining relation $(acb,cab) \in \drel{R}_\plac$ where $a \leq b < c$. So
  $w = pacbq$ and $w' = pcabq$, where $p,q \in \aA_n^*$ and $a,b,c \in \aA_n$ with $a \leq b < c$. Since $w$ and $w'$ are
  standard words, $a < b$.

  Consider how labels are assigned to the symbols $a$, $b$, and $c$ when calculating $\cochseq(w)$ as described in
  \fullref{Subsection}{subsec:cocharge}:
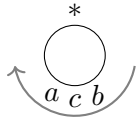
  \[
  \begin{tikzpicture}
    \draw (0,0) circle[radius=4mm];
    \draw[gray,thick,->] (-10:8mm) arc[radius=8mm,start angle=-10,end angle=-170];
    %
    \foreach\i/\ilabel in {0/*,7/a,6/c,5/b} {
      \node at ($ (90-\i*30:6mm) $) {$\ilabel$};
    };
  \end{tikzpicture}
  \]
  Among these three symbols, $a$ will receive a label first, then $b$, then $c$. Thus, after $a$, the labelling process
  will pass $*$ at least once to visit $b$ and only then visit $c$. Thus interchanging $a$ and $c$ does not alter
  the resulting labelling. Hence $\cochseq(w) = \cochseq(w')$. Similar reasoning shows that if $w$ and $w'$ differ by an
  application of a defining relation $(bac,bca) \in \drel{R}_\plac$, then $\cochseq(w) = \cochseq(w')$.
\end{proof}

For any standard tableau $T$ in $\plac$, define $\cochseq(T)$ to be $\cochseq(u)$ for any standard word $u
\in \aA^*$ such that $T = \P{u}$. By \fullref{Proposition}{prop:cochseqplactic}, $\cochseq(T)$ is well-defined.

\begin{proposition}
  \label{prop:placticbounds}
  \begin{enumerate}
  \item Connected components of $K(\plac)$ coincide with $\evrel$-classes of $\plac$.
  \item Let $k_n$ be the maximum diameter of a connected component of $K(\plac_n)$. Then $n-1 \leq k_n \leq 2n-3$.
  \end{enumerate}
\end{proposition}

\begin{proof}
  The first part is the result of Lascoux \& Sch\"utzenberger \cite[\S~4]{lascoux_plaxique}, restated in the terms used
  of this paper. It remains to prove the second part.

  To establish the lower bound on $k_n$, it is necessary to exhibit a pair of elements of $\plac_n$ that lie in the same
  connected component of $K(\plac_n)$ but are a distance at least $n-1$ apart.

  Let $t = 12\cdots (n-1)n$ and $u = n(n-1)\cdots 21$, and let
  \[
  T = \pplac{t}= \tableau{ 1 \& 2 \& |[dottedentry]| \null \& n \\}
  \quad\text{ and }
  U = \pplac{u} = \tableau{ 1 \\ 2 \\ |[dottedentry]| \null \\ n \\}
  \]
  Since $T \evrel U$, the elements $T$ and $U$ are in the same connected component of $K(\plac_n)$.  Let
  $T = T_0,T_1,\ldots,T_{m-1},T_m = U$ be a path in $K(\plac_n)$ from $T$ to $U$. Then
  \[
  T = T_0 \cyc T_1\cyc \ldots \cyc T_{m-1} \cyc T_m = U.
  \]
  Thus for $i = 0,\ldots,m-1$, there are words $u_i,v_i \in \aA_n^*$ such that $T_i = \pplac{u_iv_i}$ and
  $T_{i+1} = \pplac{v_iu_i}$. For each $i$, at least one of $u_i$ and $v_i$ does not contain the symbol $1$, and by parts~2)
  and~3) of \fullref{Lemma}{lem:cochseqcycles}, $\cochseq(T_i)$ and $\cochseq(T_{i+1})$ differ by adding $1$ to certain
  components or subtracting $1$ from certain components. Hence corresponding components of $\cochseq(T)$ and
  $\cochseq(U)$ differ by at most $m$. Since $\cochseq(T) = (0,0,\ldots,0,0)$ and $\cochseq(U) = (0,1,\ldots,n-2,n-1)$,
  it follows that $m \geq n-1$. Hence $T$ and $U$ are a distance at least $n-1$ apart in $K(\plac_n)$.

  To establish the upper bound for $k_n$, we proceed as follows. Let $T,U \in \plac_n$ be such that $T \evrel U$. Let
  $R$ be the unique tableau of row shape such that $T \evrel R \evrel U$. By \cite[Proof of
  Theorem~17]{choffrut_lexicographic}, $T \cyc^{n-1} R$ and $U \cyc^{n-1} R$. Thus there exist $T',U' \in \plac_n$ such
  that $T \cyc^{n-2} T'$ and $U \cyc^{n-2} U'$, and $T' \cyc R \cyc U'$.  Since there is only one word $w \in \aA_n^*$
  that represents the row $R$, it follows that there are words $t$ and $u$ that are cyclic shift of $w$ with
  $T' = \pplac{t}$ and $U' = \pplac{u}$. Since $t$ and $u$ are both cyclic shifts of $w$, they are cyclic shifts of each
  other and so $T' \cyc U'$. Hence $T \cyc^{2n-3} U$.
\end{proof}

Experimentation using the computer algebra software Sage \cite{sage} suggests the following conjecture on maximum
diameters of connected components of $K(\plac_n)$; see \fullref{Figure}{fig:kplac12345}:

\begin{conjecture}
  The maximum diameter of a connected component of $K(\plac_n)$ is $n-1$.
\end{conjecture}

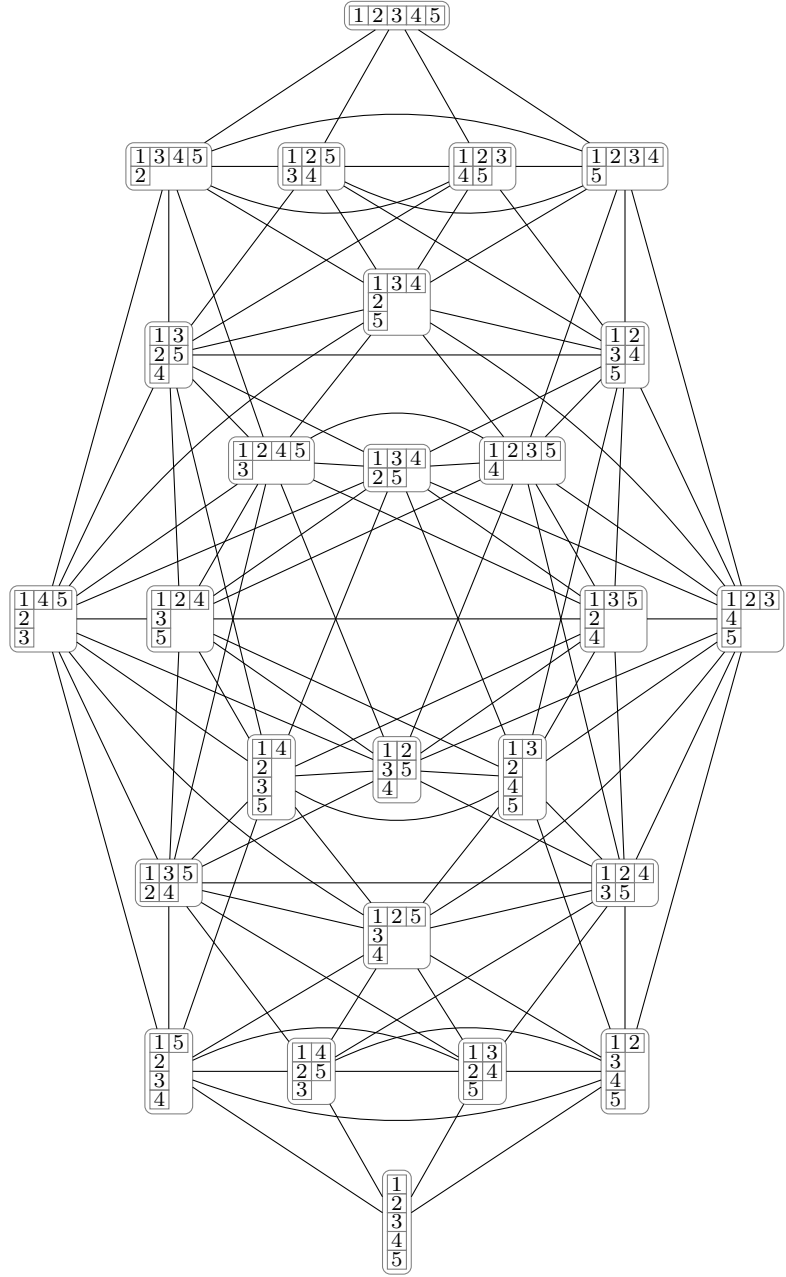
\begin{figure}[tp]
\centering
\begin{tikzpicture}[x=15mm,y=20mm]
  %
  \begin{scope}[vertex/.style={draw=gray,rectangle with rounded corners,inner sep=-.5mm}]
    \node[vertex] (12345) at (0,4) {\tikz[smalltableau]\matrix{1 \& 2 \& 3 \& 4 \& 5\\};};
    \node[vertex] (21345) at (-2,3) {\tikz[smalltableau]\matrix{1 \& 3 \& 4 \& 5\\ 2 \\};};
    \node[vertex] (34125) at (-0.75,3) {\tikz[smalltableau]\matrix{1 \& 2 \& 5\\3 \& 4 \\};};
    \node[vertex] (45123) at (0.75,3) {\tikz[smalltableau]\matrix{1 \& 2 \& 3\\4 \& 5 \\};};
    \node[vertex] (51234) at (2,3) {\tikz[smalltableau]\matrix{1 \& 2 \& 3 \& 4\\5 \\};};
    \node[vertex] (42513) at (-2,1.75) {\tikz[smalltableau]\matrix{1 \& 3\\ 2 \& 5 \\4 \\};};
    \node[vertex] (52134) at (0,2.1) {\tikz[smalltableau]\matrix{1 \& 3 \& 4\\ 2 \\ 5 \\};};
    \node[vertex] (53412) at (2,1.75) {\tikz[smalltableau]\matrix{1 \& 2\\ 3 \& 4 \\ 5 \\};};
    \node[vertex] (31245) at (-1.1,1.05) {\tikz[smalltableau]\matrix{1 \& 2 \& 4 \& 5\\3 \\};};
    \node[vertex] (25134) at (0,1) {\tikz[smalltableau]\matrix{1 \& 3 \& 4\\ 2 \& 5 \\};};
    \node[vertex] (41235) at (1.1,1.05) {\tikz[smalltableau]\matrix{1 \& 2 \& 3 \& 5\\ 4 \\};};
    \node[vertex] (32145) at (-3.1,0) {\tikz[smalltableau]\matrix{1 \& 4 \& 5\\ 2 \\ 3 \\};};
    \node[vertex] (53124) at (-1.9,0) {\tikz[smalltableau]\matrix{1 \& 2 \& 4\\ 3 \\ 5 \\};};
    \node[vertex] (42135) at (1.9,0) {\tikz[smalltableau]\matrix{1 \& 3 \& 5\\ 2 \\ 4 \\};};
    \node[vertex] (54123) at (3.1,0) {\tikz[smalltableau]\matrix{1 \& 2 \& 3\\ 4 \\ 5 \\};};
    \node[vertex] (53214) at (-1.1,-1.05) {\tikz[smalltableau]\matrix{1 \& 4 \\ 2 \\ 3 \\ 5\\};};
    \node[vertex] (43512) at (0,-1) {\tikz[smalltableau]\matrix{1 \& 2\\ 3 \& 5 \\ 4 \\};};
    \node[vertex] (54213) at (1.1,-1.05) {\tikz[smalltableau]\matrix{1 \& 3\\ 2 \\ 4 \\ 5 \\};};
    \node[vertex] (24135) at (-2,-1.75) {\tikz[smalltableau]\matrix{1 \& 3 \& 5\\ 2 \& 4 \\};};
    \node[vertex] (43125) at (0,-2.1) {\tikz[smalltableau]\matrix{1 \& 2 \& 5\\ 3 \\ 4 \\};};
    \node[vertex] (35124) at (2,-1.75) {\tikz[smalltableau]\matrix{1 \& 2 \& 4\\ 3 \& 5 \\};};
    \node[vertex] (43215) at (-2,-3) {\tikz[smalltableau]\matrix{1 \& 5\\ 2 \\ 3 \\ 4 \\};};
    \node[vertex] (32514) at (-0.75,-3) {\tikz[smalltableau]\matrix{1 \& 4\\ 2 \& 5 \\ 3 \\};};
    \node[vertex] (52413) at (0.75,-3) {\tikz[smalltableau]\matrix{1 \& 3\\ 2 \& 4 \\ 5 \\};};
    \node[vertex] (54312) at (2,-3) {\tikz[smalltableau]\matrix{1 \& 2\\ 3 \\ 4 \\ 5 \\};};
    \node[vertex] (54321) at (0,-4) {\tikz[smalltableau]\matrix{1 \\ 2 \\ 3 \\ 4 \\ 5\\};};
  \end{scope}
  \begin{scope}
    \draw (52413) edge (54312);
    \draw (52413) edge (54321);
    \draw (54312) edge (54321);
    \draw (53124) edge (54213);
    \draw (53124) edge (53214);
    \draw (45123) edge (51234);
    \draw (45123) edge (53412);
    \draw (45123) edge (52134);
    \draw (53412) edge (54123);
    \draw (53412) edge (54213);
    \draw (31245) edge (42513);
    \draw (31245) edge (52134);
    \draw (31245) edge (32145);
    \draw (31245) edge[bend left=30] (41235);
    \draw (31245) edge (43512);
    \draw (31245) edge (53124);
    \draw (31245) edge (42135);
    \draw (32145) edge (42513);
    \draw (32145) edge[bend right=10] (43125);
    \draw (32145) edge[bend left=10] (52134);
    \draw (32145) edge (43215);
    \draw (32145) edge (43512);
    \draw (32145) edge (53124);
    \draw (32145) edge (53214);
    \draw (41235) edge (51234);
    \draw (41235) edge (53412);
    \draw (41235) edge (54123);
    \draw (41235) edge (52134);
    \draw (41235) edge (43512);
    \draw (41235) edge (53124);
    \draw (41235) edge (42135);
    \draw (42135) edge (53412);
    \draw (42135) edge (54123);
    \draw (42135) edge (54213);
    \draw (42135) edge (43512);
    \draw (42135) edge (53124);
    \draw (42135) edge (53214);
    \draw (42513) edge (53412);
    \draw (42513) edge (52134);
    \draw (42513) edge (45123);
    \draw (42513) edge (53124);
    \draw (42513) edge (53214);
    \draw (54213) edge (54312);
    \draw (21345) edge (31245);
    \draw (21345) edge[bend left=20] (51234);
    \draw (21345) edge (42513);
    \draw (21345) edge (52134);
    \draw (21345) edge (32145);
    \draw (21345) edge[bend right=25] (45123);
    \draw (21345) edge (34125);
    \draw (43512) edge (54123);
    \draw (43512) edge (54213);
    \draw (43512) edge (53124);
    \draw (43512) edge (53214);
    \draw (35124) edge (54123);
    \draw (35124) edge (43125);
    \draw (35124) edge (54213);
    \draw (35124) edge (54312);
    \draw (35124) edge (43512);
    \draw (35124) edge (41235);
    \draw (35124) edge (52413);
    \draw (35124) edge (42135);
    \draw (25134) edge (31245);
    \draw (25134) edge (53412);
    \draw (25134) edge (42513);
    \draw (25134) edge (54123);
    \draw (25134) edge (54213);
    \draw (25134) edge (32145);
    \draw (25134) edge (41235);
    \draw (25134) edge (53124);
    \draw (25134) edge (42135);
    \draw (25134) edge (53214);
    \draw (43125) edge[bend right=10] (54123);
    \draw (43125) edge (54213);
    \draw (43125) edge (54312);
    \draw (43125) edge (43215);
    \draw (43125) edge (52413);
    \draw (43125) edge (53214);
    \draw (53214) edge[bend right=30] (54213);
    \draw (12345) edge (51234);
    \draw (12345) edge (21345);
    \draw (12345) edge (45123);
    \draw (12345) edge (34125);
    \draw (51234) edge (53412);
    \draw (51234) edge (54123);
    \draw (51234) edge (52134);
    \draw (32514) edge (35124);
    \draw (32514) edge (43125);
    \draw (32514) edge (43215);
    \draw (32514) edge[bend left=25] (54312);
    \draw (32514) edge (54321);
    \draw (32514) edge (52413);
    \draw (34125) edge (53412);
    \draw (34125) edge[bend right=25] (51234);
    \draw (34125) edge (42513);
    \draw (34125) edge (52134);
    \draw (34125) edge (45123);
    \draw (52134) edge (53412);
    \draw (52134) edge[bend left=10] (54123);
    \draw (24135) edge (31245);
    \draw (24135) edge (35124);
    \draw (24135) edge (43125);
    \draw (24135) edge (32145);
    \draw (24135) edge (43215);
    \draw (24135) edge (43512);
    \draw (24135) edge (52413);
    \draw (24135) edge (32514);
    \draw (24135) edge (53124);
    \draw (24135) edge (53214);
    \draw (43215) edge[bend right=20] (54312);
    \draw (43215) edge (54321);
    \draw (43215) edge[bend left=25] (52413);
    \draw (43215) edge (53214);
    \draw (54123) edge (54213);
    \draw (54123) edge (54312);
  \end{scope}
\end{tikzpicture}
\caption{The connected component $K(\plac_5,\pplac{12345})$ of the cyclic shift graph. Note that its diameter is
  $4$. (As discussed following the definition of the cyclic shift graph, the loops at each vertex are not shown.)}
\label{fig:kplac12345}
\end{figure}

\section{Hypoplactic monoid}
\label{sec:hypoplactic}

Following the course of the previous section, only the essential facts about the hypoplactic monoid are recalled here;
for proofs and further reading, see \cite{novelli_hypoplactic}.

A \defterm{quasi-ribbon tableau} is a finite array of symbols from $\aA$, with rows non-decreasing from left to right
and columns strictly increasing from top to bottom, that does not contain any $2\times 2$ subarray (that is, of the form
$\tikz[shapetableau]\matrix{ \null \& \null \\ \null \& \null \\};$). An example of a quasi-ribbon tableau is:
\begin{equation}
\label{eq:qrteg}
\tableau
{
 1 \& 1 \& 2 \&   \&   \&   \\
   \&   \& 3 \& 4 \& 4 \&   \\
   \&   \&   \&   \& 5 \&   \\
   \&   \&   \&   \& 6 \& 6 \\
}.
\end{equation}
Notice that the same symbol cannot appear in two different rows of a quasi-ribbon tableau.

The insertion algorithm is as follows:

\begin{algorithm}[{\cite[\S~7.2]{krob_noncommutative4}}]
\label{alg:hypoinsertone}
~\par\nobreak
\textit{Input:} A quasi-ribbon tableau $T$ and a symbol $a \in \aA$.

\textit{Output:} A quasi-ribbon tableau $T\leftarrow a$.

\textit{Method:} If there is no entry in $T$ that is less than or equal to $a$, output the quasi-ribbon tableau obtained
by creating a new entry $a$ and attaching (by its top-left-most entry) the quasi-ribbon tableau $T$ to the bottom of $a$.

If there is no entry in $T$ that is greater than $a$, output the quasi-ribbon tableau obtained by creating a new entry $a$ and attaching
(by its bottom-right-most entry) the quasi-ribbon tableau $T$ to the left of $a$.

Otherwise, let $x$ and $z$ be the adjacent entries of the quasi-ribbon tableau $T$ such that $x \leq a < z$.
(Equivalently, let $x$ be the right-most and bottom-most entry of $T$ that is less than or equal to $a$, and let $z$ be
the left-most and top-most entry that is greater than $a$. Note that $x$ and $z$ could be either horizontally or
vertically adjacent.) Take the part of $T$ from the top left down to and including $x$, put a new entry $a$ to
the right of $x$ and attach the remaining part of $T$ (from $z$ onwards to the bottom right) to the bottom of the new
entry $a$, as illustrated here:
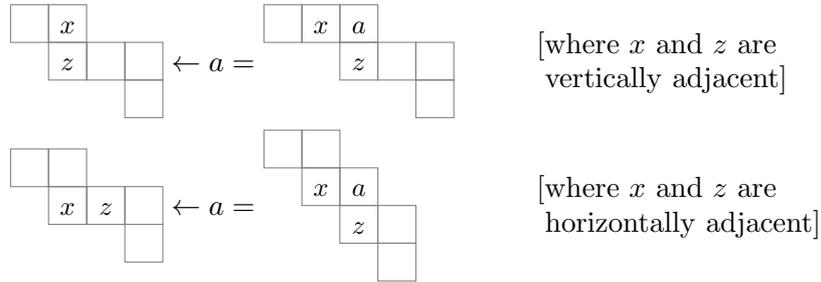
\begin{align*}
\tikz[tableau]\matrix
{
 \null \& x \&   \& \\
 \& z \& \null \& \null \\
 \&   \&   \& \null \\
}; \leftarrow a &=
\tikz[tableau]\matrix
{
 \null \& x \& a  \& \\
 \& \& z \& \null \& \null \\
 \& \& \&   \& \null \\
}; &&
\begin{array}{l}
\text{[where $x$ and $z$ are} \\
\text{\phantom{[}vertically adjacent]}
\end{array}
\displaybreak[0]\\
\tikz[tableau]\matrix
{
 \null \& \null \&   \& \\
 \& x \& z \& \null \\
 \&   \&   \& \null \\
}; \leftarrow a &=
\tikz[tableau]\matrix
{
 \null \& \null \\
 \& x \& a \\
 \&  \& z \& \null \\
 \& \& \& \null \\
}; &&
\begin{array}{l}
\text{[where $x$ and $z$ are} \\
\text{\phantom{[}horizontally adjacent]}
\end{array}
\end{align*}
Output the resulting quasi-ribbon tableau.
\end{algorithm}

Thus one can compute, for any word $u \in \aA^*$, a quasi-ribbon tableau $\phypo{u}$ by starting with an
empty quasi-ribbon tableau and successively inserting the symbols of $u$, proceeding left-to-right through the word. Define the relation
$\hypocong$ by
\[
u \hypocong v \iff \phypo{u} = \phypo{v}
\]
for all $u,v \in \aA^*$. The relation $\hypocong$ is a congruence, and the \defterm{hypoplactic monoid}, denoted
$\hypo$, is the factor monoid $\aA^*\!/{\hypocong}$; the \defterm{hypoplactic monoid of rank $n$}, denoted $\hypo_n$, is
the factor monoid $\aA_n^*/{\hypocong}$ (with the natural restriction of $\hypocong$). Each element $[u]_{\hypocong}$
(where $u \in \aA^*$) can be identified with the quasi-ribbon tableau $\phypo{u}$. For two quasi-ribbon tableaux $T$ and
$U$, denote the product of $T$ and $U$ in $\hypo$ by $T \circ U$.

The monoid $\hypo$ is presented by
$\pres{A}{\drel{R}_\hypo}$, where
\begin{align*}
\drel{R}_\hypo ={}& \drel{R}_\plac \\
&\cup \gset[\big]{(cadb,acbd)}{a \leq b < c \leq d} \\
&\cup \gset[\big]{(bdac,dbca)}{a < b \leq c < d};
\end{align*}
see \cite[\S~4.1]{novelli_hypoplactic} or \cite[\S~4.8]{krob_noncommutative4}. The monoid $\hypo_n$ is presented by
$\pres{\aA_n}{\drel{R}_\hypo}$, where the set of defining relations $\drel{R}_\hypo$ is naturally restricted to
$\aA_n^*\times \aA_n^*$. Notice that $\hypo$ and $\hypo_n$ are multihomogeneous.

It seems that cyclic shifts of elements of $\hypo$ have not been explicitly discussed in the existing
literature. However, it is clear from the defining relations that $\hypo$ is a quotient of $\plac$ under the natural
homomorphism. If two elements are $\cyc$-related in $\plac$, they are $\cyc$-related in $\hypo$. Furthermore, since
an element of $\plac$ and its image in $\hypo$ have the same evaluation, and since connected components of $K(\plac)$
coincide with $\evrel$-classes, it follows that connected components of $K(\hypo)$ also coincide with
$\evrel$-classes. That is, ${\cyc^*} = {\evrel}$ in $\hypo$.

In contrast to $K(\plac_n)$, it is possible to give an exact value for the maximum diameter of a connected component in
$K(\hypo_n)$: the aim is to prove that it is $n-1$. The proof that it cannot be smaller than $n-1$ is similar to
the proof of the lower bound in \fullref{Proposition}{prop:placticbounds}. Again, cocharge sequences are the key:

\begin{proposition}
  \label{prop:cochseqhypoplactic}
  Let $u,v \in \aA^*$ be standard words such that $u \hypocong v$. Then $\cochseq(u) = \cochseq(v)$.
\end{proposition}

\begin{proof}
  It suffices to prove the result when $u$ and $v$ differ by a single application of a defining relation. Assume that
  $u$ and $v$ differ by an application of a defining relation $(cadb,acbd) \in \drel{R}_\hypo \setminus \drel{R}_\plac$
  where $a \leq b < c \leq d$. The reasoning for defining relations of the form $(bdac,dbca)$ is similar, and for
  defining relations in $\drel{R}_\plac$, one can proceed in the same way as in the proof of
  \fullref{Proposition}{prop:cochseqplactic}. So $u = pcadbq$ and $v = pacbdq$, where $p,q \in \aA_n^*$ and
  $a,b,c,d \in \aA_n$ with $a \leq b < c \leq d$. Since $u$ and $v$ are standard words, $a < b$ and $c < d$.

  Consider how labels are assigned to the symbols $a$, $b$, $c$, and $d$ when calculating $\cochseq(w)$ as described in
  \fullref{Subsection}{subsec:cocharge}:
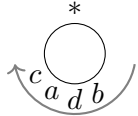
  \[
  \begin{tikzpicture}
    \draw (0,0) circle[radius=4mm];
    \draw[gray,thick,->] (-10:8mm) arc[radius=8mm,start angle=-10,end angle=-170];
    %
    \foreach\i/\ilabel in {0/*,8/c,7/a,6/d,5/b} {
      \node at ($ (90-\i*30:6mm) $) {$\ilabel$};
    };
  \end{tikzpicture}
  \]
  Among these four symbols, $a$ will receive a label first, then $b$, then $c$, then $d$. Thus, after $a$, the labelling
  process will pass $*$ at least once to visit $b$ and (perhaps after passing $*$ more times) only then visit $c$. After
  visiting $b$, it must visit $c$ first and thus must pass $*$ at least once before visiting $d$. Thus interchanging
  $a$ and $c$ and interchanging $b$ and $d$ does not alter the resulting labelling. Hence $\cochseq(u) = \cochseq(v)$.
\end{proof}

For any standard quasi-ribbon tableau $T$ in $\hypo$, define $\cochseq(T)$ to be $\cochseq(u)$ for any standard word $u
\in \aA^*$ such that $T = \phypo{u}$. By \fullref{Proposition}{prop:cochseqhypoplactic}, $\cochseq(T)$ is well-defined.

\begin{lemma}
\label{lem:hypoplacticlowerbound}
There is a connected component in $K(\hypo_n)$ with diameter at least $n-1$.
\end{lemma}

\begin{proof}
  To prove this result, it suffices to exhibit two elements that lie in the same connected component of $K(\hypo_n)$),
  but that are a distance strictly at least $n-1$ apart. Let $t = 12\cdots (n-1)n$ and $u = n(n-1)\cdots 21$, and let
  \[
  T = \phypo{t}= \tikz[tableau]\matrix{ 1 \& 2 \& |[dottedentry]| \& n \\};
  \quad\text{ and }
  U = \phypo{u} = \tikz[tableau]\matrix{ 1 \\ 2 \\ |[dottedentry]| \\ n \\};
  \]
  Then $T \evrel U$ and so $T$ and $U$ are in the same connected component of $K(\hypo_n)$. Reasoning similar to the
  proof of \fullref{Proposition}{prop:placticbounds} shows that a path from $T$ to $U$ in $K(\hypo_n)$ must have length
  at least $n-1$.
\end{proof}

Having shown that there is a connected component of $K(\hypo_n)$ with diameter at least $n-1$, the next step is to prove
that connected components of $K(\hypo_n)$ have diameter at most $n-1$. To do this, some more definitions are necessary:
\begin{itemize}
\item The \defterm{column reading} of a quasi-ribbon tableau $T$, denoted $\colreading{T}$, is the word in $\aA^*$
  obtained by reading each column from bottom to top, proceeding left to right through the columns. So the column
  reading of \eqref{eq:qrteg} is $1\;1\;32\;4\;654\;6$ (where the spaces are simply for clarity, to show readings of
  individual columns):
  
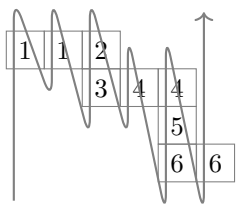
\[
  \begin{tikzpicture}[baseline=0]
    \matrix[tableaumatrix] (qrteg)
    {
      1 \& 1 \& 2 \&   \&   \&   \\
      \&   \& 3 \& 4 \& 4 \&   \\
      \&   \&   \&   \& 5 \&   \\
      \&   \&   \&   \& 6 \& 6 \\
    };
    \draw[gray,thick,rounded corners=3mm,->]
    ($ (qrteg-1-1) + (-1.5mm,-20mm) $) -- ($ (qrteg-1-1) + (-1.5mm,7mm) $) --
    ($ (qrteg-1-2) + (-1.5mm,-7mm) $) -- ($ (qrteg-1-2) + (-1.5mm,7mm) $) --
    ($ (qrteg-2-3) + (-1.5mm,-7mm) $) -- ($ (qrteg-1-3) + (-1.5mm,7mm) $) --
    ($ (qrteg-2-4) + (-1.5mm,-7mm) $) -- ($ (qrteg-2-4) + (-1.5mm,7mm) $) --
    ($ (qrteg-4-5) + (-1.5mm,-7mm) $) -- ($ (qrteg-2-5) + (-1.5mm,7mm) $) --
    ($ (qrteg-4-6) + (-1.5mm,-7mm) $) -- ($ (qrteg-4-6) + (-1.5mm,7mm) $) --
    ($ (qrteg-4-6) + (-1.5mm,20mm) $);
  \end{tikzpicture}
  \]
\item The \defterm{row reading} of a quasi-ribbon tableau $T$, denoted $\rowreading{T}$, is the word in $\aA^*$ obtained
  by reading the each row from left to right, proceeding through the rows from bottom to top. So the row reading of
  \eqref{eq:qrteg} is $66\;5\;344\;112$ (where the spaces are simply to show readings of individual rows):
  
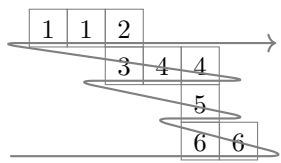
\[
  \begin{tikzpicture}[baseline=0]
    \matrix[tableaumatrix] (qrteg)
    {
      1 \& 1 \& 2 \&   \&   \&   \\
      \&   \& 3 \& 4 \& 4 \&   \\
      \&   \&   \&   \& 5 \&   \\
      \&   \&   \&   \& 6 \& 6 \\
    };
    \draw[gray,thick,rounded corners=3mm,->]
    ($ (qrteg-4-5) + (-25mm,-2mm) $) --
    ($ (qrteg-4-5) + (-7mm,-2mm) $) -- ($ (qrteg-4-6) + (7mm,-2mm) $) --
    ($ (qrteg-3-5) + (-7mm,-2mm) $) -- ($ (qrteg-3-5) + (7mm,-2mm) $) --
    ($ (qrteg-2-3) + (-7mm,-2mm) $) -- ($ (qrteg-2-5) + (7mm,-2mm) $) --
    ($ (qrteg-1-1) + (-7mm,-2mm) $) -- ($ (qrteg-1-3) + (7mm,-2mm) $) --
    ($ (qrteg-1-3) + (20mm,-2mm) $);
  \end{tikzpicture}
  \]
\end{itemize}

Let $w \in \aA^*$ and let $a_1 < a_2 < \ldots < a_k$ be the symbols in $\aA_n$ that appear in $w$. The word $w$ contains
an \defterm{$(a_i,a_{i+1})$-inversion} if it contains a symbol $a_{i+1}$ somewhere to the left of a symbol $a_i$. The
following lemma is immediate from the statement of \fullref{Algorithm}{alg:hypoinsertone}:

\begin{lemma}
  \label{lem:inversions}
  Let $w \in \aA^*$ and let $a_1 < a_2 < \ldots < a_k$ be the symbols in $\aA_n$ that appear in $w$. Then $w$ contains
  an $(a_i,a_{i+1})$-inversion if and only if $a_i$ and $a_{i+1}$ are on different rows of $\phypo{w}$.
\end{lemma}

\begin{proposition}
Let $T$ be a quasi-ribbon tableau. Then $\phypo{\colreading{T}} = \phypo{\rowreading{T}} = T$.
\end{proposition}

[The fact that $\phypo{\colreading{T}} = T$ follows from \cite[Note~4.5]{novelli_hypoplactic}, but the following proof
treats $\phypo{\colreading{T}}$ in parallel.]

\begin{proof}
  Let $a_1 < a_2 < \ldots < a_k$ be the symbols in $\aA_n$ that appear in $T$.  By definition, $\colreading{T}$ and
  $\rowreading{T}$ will contain an $(a_i,a_{i+1})$-inversion if and only if $a_i$ and $a_{i+1}$ are on different rows of
  $T$. Hence, by \fullref{Lemma}{lem:inversions}, symbols $a_i$ and $a_{i+1}$ are on different rows of
  $\phypo{\colreading{T}}$ and $\phypo{\rowreading{T}}$ if and only if they are on different rows of $T$. Furthermore,
  $\phypo{\colreading{T}}$ and $\phypo{\rowreading{T}}$ both have the same evaluation as $T$. Since quasi-ribbon tableau
  is determined by its evaluation and whether adjacent symbols are on different rows, it follows that
  $\phypo{\colreading{T}} = \phypo{\rowreading{T}} = T$.
\end{proof}

\begin{lemma}
  \label{lem:hypoplacticupperbound}
  Every connected component of $K(\hypo_n)$ has diameter at most $n-1$.
\end{lemma}

(While reading this proof, the reader may wish to look ahead to \fullref{Example}{eg:hypoplacticupperbound}, which illustrates
the strategy.)

\begin{proof}
  Let $T$ and $U$ be elements of the same connected component of $K(\hypo_n)$. Then $T \sim_{\evlit} U$. Let
  $a_1 < a_2 < \ldots < a_k$ be the symbols in $\aA_n$ that appear in $T$ and $U$. Since the defining relations in
  $\drel{R}_\hypo$ depend only on the relative order of symbols, it is clear that there is an isomorphism from
  $\gen{a_1,\ldots,a_k}$ to $\hypo_k$ extending $a_i \mapsto i$. Since $k \leq n$, it suffices to prove that the
  distance between $T$ and $U$ is at most $n-1$ when $T$ and $U$ contain every symbol in $\aA_n$.

  The aim is to construct a path in $K(\hypo_n)$ from $T$ to $U$ of length at most $n-1$. For simplicity, the aim is to
  find a sequence $T_0,T_1,\ldots,T_{n-1}$ such that $T_0 = T$ and $T_{n-1} = U$, and $T_{i} \cyc T_{i+1}$ for
  $i = 0,\ldots,n-2$. The construction is inductive, starting from $T_0 = T$ and building each $T_i$ so that the part of
  $T_i$ that contains only symbols from $\set{1,\ldots,i+1}$ is equal to the corresponding part of $U$.

  Set $T_0 = T$. Notice that since $T \sim_{\evlit} U$, both $T$ and $U$ contain the same number of symbols $1$, which
  must all lie at the left of the top rows of the two tableaux. Thus for $i = 0$, it holds that the part of $T_i$ that
  contains only symbols from $\set{1,\ldots,i+1}$ is equal to the corresponding part of $U$.

  Now suppose that for some $i$ with $1 \leq i < n$, it holds that the part of $T_{i-1}$ that contains only symbols $\set{1,\ldots,i}$
  is equal to the corresponding part of $U$.

  Consider the positions of symbols $i$ and $i+1$ in a quasi-ribbon tableau. One of two cases holds:
  \begin{itemize}
  \item[(S)] All symbols $i$ and $i+1$ are on the same row, with the rightmost symbol $i$ immediately to the left of the leftmost symbol $i+1$:
    \[
    \tableau{
      \null \& \null \\
      \& \null \& i \& |[font={\scriptsize}]| i{+}1 \& \null \& \null \\
      \& \& \& \& \& \null \\
      }.
    \]
  \item[(D)] The symbols $i$ and the symbols $i+1$ are on different rows, with the rightmost symbol $i$ immediately above the leftmost symbol $i+1$:
    \[
    \tableau{
      \null \& \null \\
      \& \null \& \rlap{\;\;$i$}\phantom{\text{\scriptsize{$i{+}1$}}} \\
      \& \& |[font={\scriptsize}]| i{+}1 \& \null \& \null \\
      \& \& \& \& \null \\
      }.
    \]
  \end{itemize}

  In each of the quasi-ribbon tableau $T_{i-1}$ and $U$, the symbols $i$ and $i+1$ may be in the same row or in
  different rows. There are thus four possible combinations of cases:
  \begin{enumerate}

  \item Suppose case~(S) holds in both $T_{i-1}$ and $U$. Set $T_i = T_{i-1}$; thus $T_{i-1} \cyc T_{i}$ by the
    reflexivity of $\cyc$. Since $T_i$ and $U$ contain the same number of symbols $i+1$, all of which must lie in the
    same row, the part of $T_i$ that contains only symbols $\set{1,\ldots,i+1}$ is equal to the corresponding part of
    $U$.

  \item Suppose case~(D) holds in both $T_{i-1}$ and $U$. Set $T_i = T_{i-1}$. By the same reasoning as above,
    $T_{i-1} \cyc T_{i}$ and the part of $T_i$ that contains only symbols $\set{1,\ldots,i+1}$ is equal to the
    corresponding part of $U$.

  \item Suppose case~(S) holds in $T_{i-1}$ but case~(D) holds in $U$. Let $\colreading{T_{i-1}} = st$, where $s$ is the
    column reading of $T_{i-1}$ up to and including the whole of the column containing the rightmost symbol $i$, and $t$ is the
    column reading of $T_{i-1}$ from (the whole of) the column containing the leftmost symbol $i+1$. Let $T_i = \phypo{ts}$; then
    $T_{i-1} \cyc T_i$.

    Notice that $s$ contains all the symbols $\set{1,\ldots,i}$ in the word $st$. Applying
    \fullref{Lemma}{lem:inversions} twice, one therefore sees that for any $j < i$, the symbols $j$ and $j+1$ are on
    different rows of $T_{i-1}$ if and only if $s$ contains a $(j+1,j)$-inversion, which is true if and only if $j$ and
    $j+1$ are on different rows of $T_i$. Thus the parts of $T_{i-1}$ and $T_i$ (and thus, by the induction hypothesis,
    $U$) that contain only symbols from $\set{1,\ldots,i}$ are equal. Furthermore, $ts$ also contains an $(i+1,i)$
    inversion, since $t$ contains at least one symbol $i+1$ and $s$ contains at least one symbol $i$, and so by
    \fullref{Lemma}{lem:inversions} the symbols $i$ and $i+1$ are on different rows of $T_i$. Hence the parts of $T_i$
    and $U$ that contain symbols from $\set{1,\ldots,i+1}$ are equal.

  \item Suppose case~(D) holds in $T_{i-1}$ but case~(S) holds in $U$. Let $\rowreading{T_{i-1}} = st$, where $s$ is the
    row reading of $T_{i-1}$ up to and including the whole of the row containing the symbols $i+1$, and $t$ is the row
    reading of $T_{i-1}$ from (the whole of) the row containing the symbols $i$. Let $T_i = \phypo{ts}$;
    then $T_{i-1} \cyc T_i$.

    Notice that $t$ contains all the symbols $\set{1,\ldots,i}$ in the word $st$. Applying
    \fullref{Lemma}{lem:inversions} twice, one therefore sees that for any $j < i$, the symbols $j$ and $j+1$ are on
    different rows of $T_{i-1}$ if and only if $t$ contains a $(j+1,j)$-inversion, which is true if and only if $j$ and
    $j+1$ are on different rows of $T_i$. Thus the parts of $T_{i-1}$ and $T_i$ (and thus, by the induction hypothesis,
    $U$) that contain only symbols from $\set{1,\ldots,i}$ are equal. Furthermore, $ts$ does not contains an $(i+1,i)$
    inversion, since $t$ contains all the symbols $i$ and $s$ contains all the symbols $i+1$, and so by
    \fullref{Lemma}{lem:inversions} the symbols $i$ and $i+1$ are on the same row of $T_i$. Hence the parts of $T_i$
    and $U$ that contain symbols from $\set{1,\ldots,i+1}$ are equal.

  \end{enumerate}
  By induction, the parts of the quasi-ribbon tableaux $T_{n-1}$ and $U$ that contain symbols from $\set{1,\ldots,n}$
  are equal. That is, $T_{n-1} = U$. Therefore $T$ and $U$ are a distance at most $n-1$ apart.
\end{proof}

Combining \fullref{Lemmata}{lem:hypoplacticupperbound} and \ref{lem:hypoplacticlowerbound} gives the following result:

\begin{theorem}
  \label{thm:hypoplacticbounds}
  \begin{enumerate}
  \item Connected components of $K(\hypo)$ coincide with $\evrel$-classes of $\hypo$.
  \item The maximum diameter of a connected component of $K(\hypo_n)$ is $n-1$.
  \end{enumerate}
\end{theorem}

\begin{figure}[tp]
\centering
\begin{tikzpicture}[x=15mm,y=20mm]
  %
  \begin{scope}[vertex/.style={draw=gray,rectangle with rounded corners,inner sep=-.5mm}]
    \node[vertex] (12345) at (0,3) {\tikz[smalltableau]\matrix{1 \& 2 \& 3 \& 4 \& 4 \& 5\\};};
    \node[vertex] (21345) at (-2,2) {\tikz[smalltableau]\matrix{1 \\ 2 \& 3 \& 4 \& 4 \& 5 \\};};
    \node[vertex] (13245) at (-0.75,2) {\tikz[smalltableau]\matrix{1 \& 2 \\\& 3 \& 4 \& 4 \& 5\\};};
    \node[vertex] (12435) at (0.75,2) {\tikz[smalltableau]\matrix{1 \& 2 \& 3\\\& \& 4 \& 4 \& 5 \\};};
    \node[vertex] (12354) at (2,2) {\tikz[smalltableau]\matrix{1 \& 2 \& 3 \& 4\\ \& \& \& 5 \\};};
    \node[vertex] (21354) at (0,1) {\tikz[smalltableau]\matrix{1 \\ 2 \& 3 \& 4 \& 4 \\ \& \& \& 5 \\};};
    \node[vertex] (32145) at (-3,0) {\tikz[smalltableau]\matrix{1 \\ 2 \\ 3 \& 4 \& 4 \& 5\\};};
    \node[vertex] (13254) at (1.5,0) {\tikz[smalltableau]\matrix{1 \& 2 \\ \& 3 \& 4 \& 4 \\ \& \& \& 5 \\};};
    \node[vertex] (21435) at (-1.5,0) {\tikz[smalltableau]\matrix{1 \\ 2 \& 3 \\ \& 4 \& 4 \& 5\\};};
    \node[vertex] (12543) at (3,0) {\tikz[smalltableau]\matrix{1 \& 2 \& 3\\ \& \& 4 \& 4 \\ \& \& \& 5 \\};};
    \node[vertex] (14325) at (0,-1) {\tikz[smalltableau]\matrix{1 \& 2 \\ \& 3 \\ \& 4 \& 4 \& 5\\};};
    \node[vertex] (43215) at (-2,-2) {\tikz[smalltableau]\matrix{1 \\ 2 \\ 3 \\ 4 \& 4 \& 5\\};};
    \node[vertex] (32154) at (-0.75,-2) {\tikz[smalltableau]\matrix{1 \\ 2 \\ 3 \& 4 \& 4 \\ \& \& 5 \\};};
    \node[vertex] (21543) at (0.75,-2) {\tikz[smalltableau]\matrix{1 \\ 2 \& 3 \\ \& 4 \& 4 \\ \& \& 5 \\};};
    \node[vertex] (15432) at (2,-2) {\tikz[smalltableau]\matrix{1 \& 2\\ \& 3 \\ \& 4 \& 4 \\ \& \& 5 \\};};
    \node[vertex] (54321) at (0,-3) {\tikz[smalltableau]\matrix{1 \\ 2 \\ 3 \\ 4 \& 4 \\ \& 5\\};};
  \end{scope}
  \begin{scope}
    \draw (21543) edge (15432);
    \draw (21543) edge (54321);
    \draw (15432) edge (54321);
    \draw (13254) edge (21543);
    \draw (13254) edge[bend left=15] (32154);
    \draw (12435) edge (12354);
    \draw (12435) edge (13254);
    \draw (12435) edge (21354);
    \draw (13254) edge (12543);
    \draw (13245) edge (21435);
    \draw (13245) edge (21354);
    \draw (13245) edge (32145);
    \draw (13245) edge (12435);
    \draw (13245) edge (14325);
    \draw (13245) edge[bend left=15] (13254);
    \draw (32145) edge (21435);
    \draw (32145) edge (14325);
    \draw (32145) edge (21354);
    \draw (32145) edge (43215);
    \draw (32145) edge[bend left=20] (13254);
    \draw (32145) edge (32154);
    \draw (12435) edge (12543);
    \draw (12435) edge (14325);
    \draw (21435) edge (13254);
    \draw (21435) edge[bend right=20] (12543);
    \draw (21435) edge[bend right=15] (21543);
    \draw (21435) edge (14325);
    \draw (21435) edge (32154);
    \draw (21435) edge (21354);
    \draw (21435) edge[bend left=15] (12435);
    \draw (21543) edge (15432);
    \draw (21345) edge (13245);
    \draw (21345) edge[bend left=30] (12354);
    \draw (21345) edge (21435);
    \draw (21345) edge (21354);
    \draw (21345) edge (32145);
    \draw (21345) edge[bend right=30] (12435);
    \draw (14325) edge (12543);
    \draw (14325) edge (21543);
    \draw (14325) edge (13254);
    \draw (14325) edge (32154);
    \draw (13254) edge (12543);
    \draw (13254) edge (15432);
    \draw (21354) edge (12543);
    \draw (21354) edge (21543);
    \draw (21354) edge (13254);
    \draw (21354) edge (32154);
    \draw (14325) edge (15432);
    \draw (14325) edge (43215);
    \draw (32154) edge (21543);
    \draw (12345) edge (12354);
    \draw (12345) edge (21345);
    \draw (12345) edge (12435);
    \draw (12345) edge (13245);
    \draw (12354) edge (13254);
    \draw (12354) edge (12543);
    \draw (12354) edge (21354);
    \draw (32154) edge (43215);
    \draw (32154) edge[bend left=30] (15432);
    \draw (32154) edge (54321);
    \draw (32154) edge (21543);
    \draw (13245) edge[bend right=30] (12354);
    \draw (21435) edge (32145);
    \draw (21435) edge (43215);
    \draw (43215) edge[bend right=30] (15432);
    \draw (43215) edge (54321);
    \draw (43215) edge[bend left=30] (21543);
    \draw (43215) edge (32154);
    \draw (12543) edge (21543);
    \draw (12543) edge (15432);
  \end{scope}
\end{tikzpicture}
\caption{The connected component $K(\hypo_5,\phypo{123445})$. Note that its diameter is $4$.}
\label{fig:khypo12345}
\end{figure}
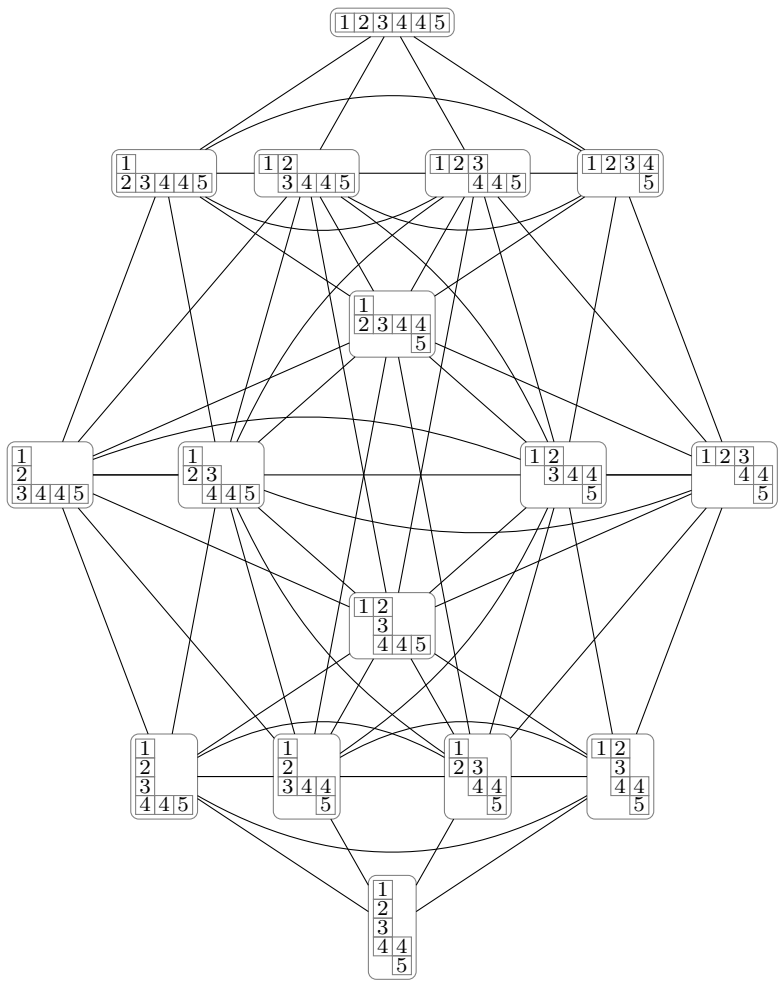

The following example illustrates how the construction in the proof of \fullref{Lemma}{lem:hypoplacticupperbound}:

\begin{example}
\label{eg:hypoplacticupperbound}
Consider the following elements of $\hypo_5$:
\[
T = \tableau
{
 1 \&   \\
 2 \& 3 \\
   \& 4 \& 4 \& 5\\
}\quad\text{ and }
U = \tableau
{
 1 \& 2 \\
   \& 3 \& 4 \& 4 \\
   \&   \&   \& 5 \\
}.
\]
Then $T$ and $U$ have the same evaluation, and so lie in the same connected component of $K(\hypo)$ by
\fullref[(1)]{Theorem}{thm:hypoplacticbounds}, and the distance between them is at most $4$ by
\fullref[(2)]{Theorem}{thm:hypoplacticbounds}. The connected component $K(\hypo_5,T)$ is shown in
\fullref{Figure}{fig:khypo12345}, and the construction in the proof of \fullref{Lemma}{lem:hypoplacticupperbound} yields
the following path between $T$ and $U$:
\begin{align*}
T = T_0 ={}& \tikz[tableau]\matrix
{
 1 \&   \\
 2 \& 3 \\
   \& 4 \& 4 \& 5\\
};\\
\hypocong{}&
\tikz[tableau]\matrix
{
 2 \& 3 \\
   \& 4 \& 4 \& 5\\
};
\circ
\tikz[tableau]\matrix
{
 1 \\
};
\cyc
\tikz[tableau]\matrix
{
 1 \\
};
\circ
\tikz[tableau]\matrix
{
 2 \& 3 \\
   \& 4 \& 4 \& 5\\
};
&&\text{(case 4)}\\
\hypocong T_1 ={}&
\tikz[tableau]\matrix
{
 1 \& 2 \& 3 \\
   \&   \& 4 \& 4 \& 5\\
};\\
\hypocong{}&
\tikz[tableau]\matrix
{
 1 \& 2 \\
};
\circ
\tikz[tableau]\matrix
{
3 \\
4 \& 4 \& 5\\
};
\cyc
\tikz[tableau]\matrix
{
3 \\
4 \& 4 \& 5\\
};
\circ
\tikz[tableau]\matrix
{
 1 \& 2 \\
}; &&\text{(case 3)}\displaybreak[0]\\
\hypocong T_2 ={}&
\tikz[tableau]\matrix
{
 1 \& 2 \\
   \& 3 \\
   \& 4 \& 4 \& 5 \\
};\\
\hypocong{}&
\tikz[tableau]\matrix
{
4 \& 4 \& 5 \\
};
\circ
\tikz[tableau]\matrix
{
 1 \& 2 \\
   \& 3 \\
};
\cyc
\tikz[tableau]\matrix
{
 1 \& 2 \\
   \& 3 \\
};
\circ
\tikz[tableau]\matrix
{
4 \& 4 \& 5 \\
};&&\text{(case 4)}\displaybreak[0]\\
\hypocong T_3 ={}&
\tikz[tableau]\matrix
{
 1 \& 2 \\
   \& 3 \& 4 \& 4 \& 5 \\
};\\
\hypocong{}&
\tikz[tableau]\matrix
{
 1 \& 2 \\
   \& 3 \& 4 \& 4 \\
};
\circ
\tikz[tableau]\matrix
{
5 \\
};
\cyc
\tikz[tableau]\matrix
{
5 \\
};
\circ
\tikz[tableau]\matrix
{
 1 \& 2 \\
   \& 3 \& 4 \& 4 \\
};&&\text{(case 3)}\\
\hypocong T_4 ={}&
\tikz[tableau]\matrix
{
 1 \& 2 \\
   \& 3 \& 4 \& 4 \\
   \&   \&   \& 5 \\
}; = U.
\end{align*}
This gives a path of length $4$ between $T$ and $U$ in $K(\hypo_5)$.

Notice, however, that $T$ and $U$ are actually a distance $1$ apart in $K(\hypo_5)$:
\begin{align*}
T = \tableau
{
 1 \&   \\
 2 \& 3 \\
   \& 4 \& 4 \& 5\\
  }
\hypocong{}& \tableau{2 \& 4 \& 4\\} \circ \tableau{1\\3\\5\\} \cyc \tableau{1\\3\\5\\} \circ \tableau{2 \& 4 \& 4\\} \\
\hypocong{}& \tableau
{
 1 \& 2 \\
   \& 3 \& 4 \& 4 \\
   \&   \&   \& 5 \\
} = U.
\end{align*}
\end{example}

\section{Sylvester monoid}
\label{sec:sylvester}

\subsection{Preliminaries}

Only the necessary facts about the sylvester monoid are recalled here; see \cite{hivert_sylvester} for further
background.

A \defterm{\textparens{right strict} binary search tree} is a labelled rooted binary tree where the label of each node is greater
than or equal to the label of every node in its left subtree, and strictly less than every node in its right subtree. An
example of a binary search tree is:
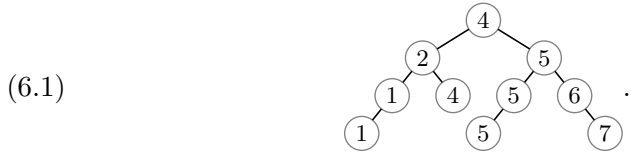
\begin{equation}
\label{eq:bsteg}
\begin{tikzpicture}[tinybst,baseline=-10mm]
  \node (root) {$4$}
    child[sibling distance=16mm] { node (0) {$2$}
      child { node (00) {$1$}
        child { node (000) {$1$} }
        child[missing]
      }
      child { node (01) {$4$} }
    }
    child[sibling distance=16mm] { node (1) {$5$}
      child { node (10) {$5$}
        child { node (100) {$5$} }
        child[missing]
      }
      child { node (11) {$6$}
        child[missing]
        child { node (111) {$7$} }
      }
    };
\end{tikzpicture}.
\end{equation}

\begin{figure}[ht]
  \centering
  \begin{tikzpicture}[x=15mm,y=15mm]
    \useasboundingbox (-.5,-3.5) -- (3.5,4.5);
    %
    \begin{scope}[vertex/.style={draw=gray,rectangle with rounded corners,inner sep=1mm}]
      \node[vertex] (1234) at (1,4) {\tikz[microbst]\draw node {$4$}  child { node {$3$}  child { node {$2$}  child { node {$1$}  } child[missing] } child[missing] } child[missing];};
      \node[vertex] (2341) at (2,4) {\tikz[microbst]\draw node {$1$} child[missing] child { node {$4$}  child { node {$3$}  child { node {$2$}  } child[missing] } child[missing] };};
      \node[vertex] (2314) at (3,3) {\tikz[microbst]\draw node {$4$}  child { node {$1$} child[missing] child { node {$3$}  child { node {$2$}  } child[missing] }  } child[missing];};
      \node[vertex] (2134) at (0,2) {\tikz[microbst]\draw node {$4$}  child { node {$3$}  child { node {$1$} child[missing] child { node {$2$}  }  } child[missing] } child[missing];};
      \node[vertex] (1342) at (1,2) {\tikz[microbst]\draw node {$2$}  child { node {$1$}  }  child { node {$4$}  child { node {$3$}  } child[missing] };};
      \node[vertex] (1243) at (2,2) {\tikz[microbst]\draw node {$3$}  child { node {$2$}  child { node {$1$}  } child[missing] }  child { node {$4$}  };};
      \node[vertex] (2431) at (3,1) {\tikz[microbst]\draw node {$1$} child[missing] child { node {$3$}  child { node {$2$}  }  child { node {$4$}  }  };};
      \node[vertex] (1324) at (3,0) {\tikz[microbst]\draw node {$4$}  child { node {$2$}  child { node {$1$}  }  child { node {$3$}  }  } child[missing];};
      \node[vertex] (3421) at (0,-1) {\tikz[microbst]\draw node {$1$} child[missing] child { node {$2$} child[missing] child { node {$4$}  child { node {$3$}  } child[missing] }  };};
      \node[vertex] (2143) at (1,-1) {\tikz[microbst]\draw node {$3$}  child { node {$1$} child[missing] child { node {$2$}  }  }  child { node {$4$}  };};
      \node[vertex] (1432) at (2,-1) {\tikz[microbst]\draw node {$2$}  child { node {$1$}  }  child { node {$3$} child[missing] child { node {$4$}  }  };};
      \node[vertex] (3241) at (3,-2) {\tikz[microbst]\draw node {$1$} child[missing] child { node {$4$}  child { node {$2$} child[missing] child { node {$3$}  }  } child[missing] };};
      \node[vertex] (3214) at (1,-3) {\tikz[microbst]\draw node {$4$}  child { node {$1$} child[missing] child { node {$2$} child[missing] child { node {$3$}  }  }  } child[missing];};
      \node[vertex] (4321) at (2,-3) {\tikz[microbst]\draw node {$1$} child[missing] child { node {$2$} child[missing] child { node {$3$} child[missing] child { node {$4$}  }  }  };};
    \end{scope}
    \begin{scope}
      \draw (1234) edge (1243);
      \draw (1234) edge (1342);
      \draw (1234) edge (2341);
      \draw (1243) edge[bend right=20] (1324);
      \draw (1243) edge (1342);
      \draw (1243) edge (1432);
      \draw (1243) edge (2314);
      \draw (1243) edge (2341);
      \draw (1243) edge (2431);
      \draw (1324) edge (1432);
      \draw (1324) edge[bend right=20] (2143);
      \draw (1324) edge (2431);
      \draw (1324) edge (3241);
      \draw (1342) edge (2134);
      \draw (1342) edge (2143);
      \draw (1342) edge[bend left=20] (2314);
      \draw (1342) edge (2341);
      \draw (1342) edge[bend right=20] (2431);
      \draw (1342) edge (3421);
      \draw (1432) edge (2143);
      \draw (1432) edge[bend left=20] (2431);
      \draw (1432) edge (3214);
      \draw (1432) edge (3241);
      \draw (1432) edge (4321);
      \draw (2134) edge (2143);
      \draw (2134) edge (3421);
      \draw (2143) edge (3214);
      \draw (2143) edge[bend right=20] (3241);
      \draw (2143) edge (3421);
      \draw (2143) edge (4321);
      \draw (2314) edge (2431);
      \draw (3214) edge (4321);
    \end{scope}
  \end{tikzpicture}
  \caption{The connected component $K(\sylv_4,\psylv{1234})$; note that its diameter is $3$.}
\end{figure}
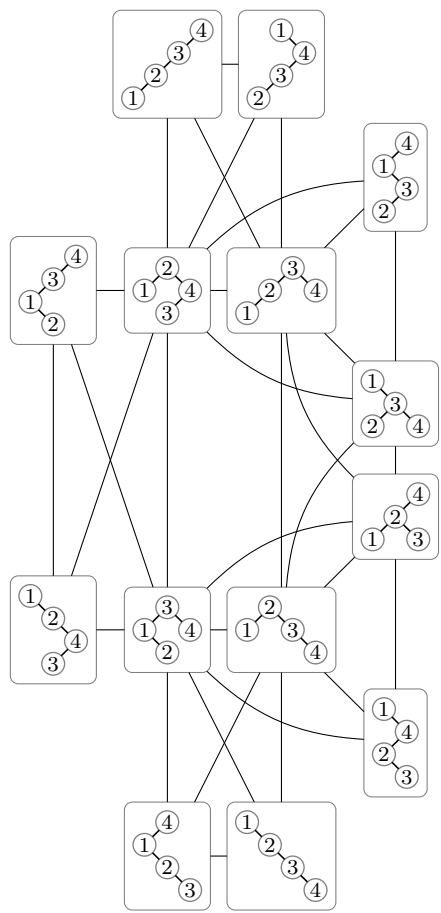


The \defterm{left-to-right postfix traversal}, or simply the \defterm{postfix traversal}, of a rooted binary tree $T$ is
the sequence that `visits' every node in the tree as follows: it recursively perform the postfix traversal
of the left subtree of the root of $T$, then recursively perform the postfix traversal of the right
subtree of the root of $T$, then visits the root of $T$. Thus the postfix traversal of any binary tree
with the same shape as the \eqref{eq:bsteg} visits nodes as follows:
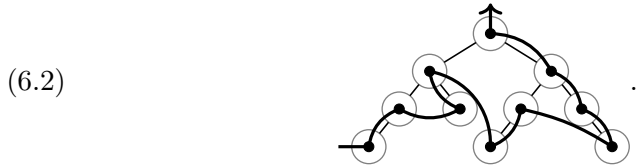
\begin{equation}
\begin{tikzpicture}[tinybst,baseline=-7.5mm]
  \node (root) {}
    child[sibling distance=16mm] { node (0) {}
      child { node (00) {}
        child { node (000) {} }
        child[missing]
      }
      child { node (01) {} }
    }
    child[sibling distance=16mm] { node (1) {}
      child { node (10) {}
        child { node (100) {} }
        child[missing]
      }
      child { node (11) {}
        child[missing]
        child { node (111) {} }
      }
    };
  \begin{scope}[very thick,line cap=round]
    \draw (000.center) edge[bend left=30] (00.center);
    \draw (00.center) edge[bend right=30] (01.center);
    \draw (01.center) edge[bend left=30] (0.center);
    \draw (0.center) edge[bend left=40] (100.center);
    \draw (100.center) edge[bend right=30] (10.center);
    \draw (10.center) edge[bend left=10] (111.center);
    \draw (111.center) edge[bend right=30] (11.center);
    \draw (11.center) edge[bend right=30] (1.center);
    \draw (1.center) edge[bend right=30] (root.center);
  \end{scope}
  \draw[very thick] ($ (000.center) + (-4mm,0) $) -- (000.center);
  \draw[very thick,->] (root.center) -- ($ (root.center) + (0,4mm) $);
  \foreach\x in {000,00,01,0,100,10,111,11,1,root} {
    \draw[draw=black,fill=black] (\x.center) circle (.66mm);
  }
\end{tikzpicture}.
\end{equation}

The insertion algorithm for binary search trees adds the new symbol as a leaf node in the unique place that maintains
the property of being a binary search tree:

\begin{algorithm}[Right strict leaf insertion]
\label{alg:sylvinsertone}
~\par\nobreak
\textit{Input:} A binary search tree $T$ and a symbol $a \in \aA$.

\textit{Output:} A binary search tree $a \rightarrow T$.

\textit{Method:} If $T$ is empty, create a node and label it $a$. If $T$ is non-empty, examine the label $x$ of the root
node; if $a \leq x$, recursively insert $a$ into the left subtree of the root node; otherwise recursively insert $a$
into the right subtree of the root note. Output the resulting tree.
\end{algorithm}

Thus one can compute, for any word $u \in \aA^*$, a binary search tree $\psylv{u}$ by starting with an
empty binary search tree and successively inserting the symbols of $u$, proceeding right-to-left through the word. For
example $\psylv{5451761524}$ is \eqref{eq:bsteg}.

A \defterm{reading} of a binary search tree $\tree{T}$ is a word $u$ such that $\psylv{u} = T$. It is easy to see that a
reading of $T$ is a word formed from the symbols that appear in the nodes of $T$, arranged so that every symbol from a
parent node appears to the right of those from its children. For example, $1571456254$ is a reading of \eqref{eq:bsteg}.

Define the relation
$\sylvcong$ by
\[
u \sylvcong v \iff \psylv{u} = \psylv{v}.
\]
for all $u,v \in \aA^*$. The relation $\sylvcong$ is a congruence, and the \defterm{sylvester monoid}, denoted $\sylv$,
is the factor monoid $\aA^*\!/{\sylvcong}$; the \defterm{sylvester monoid of rank $n$}, denoted $\sylv_n$, is the factor
monoid $\aA_n^*/{\sylvcong}$ (with the natural restriction of $\sylvcong$). Each element $[u]_{\sylvcong}$ (where
$u \in \aA^*$) can be identified with the binary search tree $\psylv{u}$. The words in $[u]_{\sylvcong}$ are precisely
the readings of $\psylv{u}$.

The monoid $\sylv$ is presented by $\pres{\aA}{\drel{R}_\sylv}$, where
\[
\drel{R}_\sylv = \gset[\big]{(cavb,acvb)}{a \leq b < c,\; v \in \aA^*};
\]
the monoid $\sylv_n$ is presented by $\pres{\aA_n}{\drel{R}_\sylv}$, where the set of defining relations
$\drel{R}_\sylv$ is naturally restricted to $\aA_n^*\times \aA_n^*$. Notice that $\sylv$ and $\sylv_n$ are
multihomogeneous.

The present authors proved that the relations ${\evrel}$ and ${\cyc^*}$ coincide in $\sylv$
\cite[Theorem~3.4]{cm_conjugacy}. In the terms of this paper, this proves that connected components of $K(\sylv)$ are
$\evrel$-classes. The aim in this section is to prove that the maximum diameter of a connected component of $K(\sylv_n)$
is either $n-1$ or $n$. \fullref{Subsection}{subsec:sylvlowerbound} shows that $K(\sylv_n)$ has a connected component
with diameter at least $n-1$. \fullref{Subsections}{subsec:sylvupperbound1} to \ref{subsec:sylvupperbound3} show that
connected components of $K(\sylv_n)$ have diameter at most $n$.

\subsection{Lower bound for diameters}
\label{subsec:sylvlowerbound}

As in the cases of the plactic and hypoplactic monoids, cocharge sequences are the key to proving that there is a
connected component of $K(\sylv_n)$ with diameter at least $n-1$. Reasoning similar to the proofs of
\fullref{Propositions}{prop:cochseqplactic} and~\ref{prop:cochseqhypoplactic} establishes the following result:

\begin{proposition}
  \label{prop:cochseqsylvester}
  Let $u,v \in \aA^*$ be standard words such that $u \sylvcong v$. Then $\cochseq(u) = \cochseq(v)$.
\end{proposition}

For any standard binary tree $T$ in $\sylv$, define $\cochseq(T)$ to be $\cochseq(u)$ for any standard word
$u \in \aA^*$ such that $T = \psylv{u}$. By \fullref{Proposition}{prop:cochseqsylvester}, $\cochseq(T)$ is well-defined.

\begin{lemma}
  \label{lem:sylvesterlowerbound}
  There is a connected component in $K(\sylv_n)$ of diameter at least $n-1$.
\end{lemma}

\begin{proof}
  The strategy is the same as in the plactic and hypoplactic monoids: let $t = 12\cdots (n-1)n$ and
  $u = n(n-1)\cdots 21$, and let
\[
T = \psylv{t}=
\begin{tikzpicture}[tinybst,level distance=7mm,baseline=-10.5mm]
  \node {$n$}
  child { node {$n-1$}
    child[dotted] { node[solid] {$2$}
      child[solid] { node {$1$} }
      child[missing]
    }
    child[missing]
  }
  child[missing];
\end{tikzpicture}
\qquad\text{ and }U = \psylv{u} =
\begin{tikzpicture}[tinybst,level distance=7mm,baseline=-10.5mm]
  \node {$1$}
  child[missing]
  child { node {$2$}
    child[missing]
    child[dotted] { node[solid] {$n-1$}
      child[missing]
      child[solid] { node {$n$} }
    }
  };
\end{tikzpicture}
\]
The same reasoning as in the proof of \fullref{Proposition}{prop:placticbounds} shows that $T$ and $U$ are not related
by $\cyc^{\leq n-2}$ in $\sylv_n$.
\end{proof}

\subsection{Upper bound for diameters I: Overview}
\label{subsec:sylvupperbound1}

The proof that every connected component of $K(\sylv_n)$ has diameter at most $n$ is long and complicated.  To
illustrate the strategy, \fullref{Example}{eg:sylvupperbound} explicity constructs a path of length~$5$ between two
elements of $K(\sylv_5)$ that have the same evaluation. Note, however, that these elements are standard, and much of the
complexity of the general proof is due to having to consider multiple appearances of each symbol.

\begin{example}
  Let $T$ and $U$ be the following elements of $\sylv_5$:
  \label{eg:sylvupperbound}
  \[
    T =
    \begin{tikzpicture}[baseline=(0.base)]
      \begin{scope}[tinybst]
        \node (root) {$4$}
        child { node (0) {$2$}
          child { node (00) {$1$} }
          child { node (01) {$3$} }
        }
        child { node (00) {$5$} };
      \end{scope}
    \end{tikzpicture},\quad
    U =
    \begin{tikzpicture}[baseline=(1.base)]
      \begin{scope}[tinybst]
        \node (root) {$1$}
        child[missing]
        child { node (1) {$4$}
          child { node (10) {$3$}
            child { node (100) {$2$} }
            child[missing]
          }
          child { node (11) {$5$} }
        };
      \end{scope}
    \end{tikzpicture}\;\; \in \sylv_5
  \]
  Notice that $T \evrel U$.  The aim is to build a sequence $T = T_0 \cyc T_1 \cyc T_2 \cyc T_3 \cyc T_4 \cyc T_5 = U$.
  The strategy is build the path based upon a left-to-right postfix traversal of $U$. At any point in such a traversal,
  the set of vertices already encountered is a union of subtrees of $U$. By applying an appropriate cyclic shift to
  obtain $T_{i+1}$ from $T_i$ (for each $i$), one gradually builds copies of these subtrees within the $T_i$, arranged
  down the path of left child nodes descending from the root, with the `just completed' subtree at the root itself. In
  the example construction below, the left-hand column shows the $i$-th tree built so far, and the next column shows the
  $i$-th step of the left-to-right postfix traversal. The subtrees of $U$ containing the vertices encountered up to the
  $i$-th step are outlined, as are the corresponding vertices in the tree $T_i$. Note that cyclic shifts never break up
  the subwords (outlined) that represent the already-built subtrees.
  \begin{align*}
  T = T_0 ={}& \psylv{13254} \cyc{} \psylv{54132}\\
  =T_1 ={}&
    \begin{tikzpicture}[baseline=(0.base)]
      \begin{scope}[tinybst]
        \node (root) {$2$}
        child { node (1) {$1$} }
        child { node (1) {$3$}
          child[missing]
          child { node (11) {$4$}
            child[missing]
            child { node (110) {$5$} }
          }
        };
        \draw[bstoutline] (root) circle[radius=3mm];
      \end{scope}
    \end{tikzpicture}
    &
      \begin{tikzpicture}[baseline=(1.base)]
        \begin{scope}[tinybst]
          \node (root) {$1$}
          child[missing]
          child { node (1) {$4$}
            child { node (10) {$3$}
              child { node[fill=gray] (100) {$2$} }
              child[missing]
            }
            child { node (11) {$5$} }
          };
        \end{scope}
        \draw[bstoutline] (100) circle[radius=3mm];
      \end{tikzpicture}
          &
            \quad\parbox{5cm}{\raggedright One cyclic shift moves the first node visited by the traversal to the root. This will be the base of the induction in the proof.}
    \\
    ={}& \psylv{5431\olsubword{2}}  \cyc{} \psylv{1\olsubword{2}543}
    \displaybreak[0]\\
  =T_2 ={}&
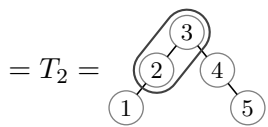
    \begin{tikzpicture}[baseline=(0.base)]
      \begin{scope}[tinybst]
        \node (root) {$3$}
        child { node (0) {$2$}
          child { node (00) {$1$} }
          child[missing]
        }
        child { node (1) {$4$}
          child[missing]
          child { node (11) {$5$} }
        };
      \end{scope}
      \draw[bstoutline] \convexpath{root,0}{3mm};
    \end{tikzpicture}
    &
      
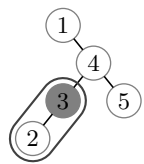
\begin{tikzpicture}[baseline=(1.base)]
        \begin{scope}[tinybst]
          \node (root) {$1$}
          child[missing]
          child { node (1) {$4$}
            child { node[fill=gray] (10) {$3$}
              child { node (100) {$2$} }
              child[missing]
            }
            child { node (11) {$5$} }
          };
        \end{scope}
        \draw[bstoutline] \convexpath{10,100}{3mm};
      \end{tikzpicture}
          &
            \quad\parbox{5cm}{\raggedright The postfix traversal moves from a left child directly to its parent (which has empty right subtree). This will be the induction step, case~3.}
    \\
    ={}& \psylv{541\olsubword{23}}   \cyc{} \psylv{41\olsubword{23}5}
    \displaybreak[0]\\
  =T_3 ={}&
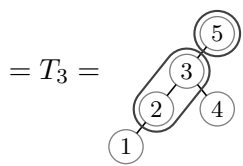
    \begin{tikzpicture}[baseline=(0.base)]
      \begin{scope}[tinybst]
        \node (root) {$5$}
        child { node (0) {$3$}
          child { node (00) {$2$}
            child { node (000) {$1$} }
            child[missing]
          }
          child { node (01) {$4$} }
        }
        child[missing];
      \end{scope}
      \draw[bstoutline] (root) circle[radius=3mm];
      \draw[bstoutline] \convexpath{0,00}{3mm};
    \end{tikzpicture}
    &
      
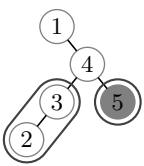
\begin{tikzpicture}[baseline=(1.base)]
        \begin{scope}[tinybst]
          \node (root) {$1$}
          child[missing]
          child { node (1) {$4$}
            child { node (10) {$3$}
              child { node (100) {$2$} }
              child[missing]
            }
            child { node[fill=gray] (11) {$5$} }
          };
        \end{scope}
      \draw[bstoutline] (11) circle[radius=3mm];
      \draw[bstoutline] \convexpath{10,100}{3mm};
      \end{tikzpicture}
          &
            \quad\parbox{5cm}{\raggedright The postfix traversal moves from a left child to a node in the right subtree of its parent. This will be the induction step, case 1.}
    \\
    ={}& \psylv{41\olsubword{23}\olsubword{5}}  \cyc{} \psylv{1\olsubword{23}\olsubword{5}4}
    \displaybreak[0]\\
  =T_4 ={}&
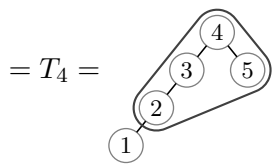
    \begin{tikzpicture}[baseline=(0.base)]
      \begin{scope}[tinybst]
        \node (root) {$4$}
        child { node (0) {$3$}
          child { node (00) {$2$}
            child { node (000) {$1$} }
            child[missing]
          }
          child[missing]
        }
        child { node (1) {$5$} };
      \end{scope}
      \draw[bstoutline] \convexpath{root,1,00}{3mm};
    \end{tikzpicture}
    &
      
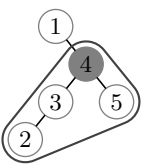
\begin{tikzpicture}[baseline=(1.base)]
        \begin{scope}[tinybst]
          \node (root) {$1$}
          child[missing]
          child { node[fill=gray] (1) {$4$}
            child { node (10) {$3$}
              child { node (100) {$2$} }
              child[missing]
            }
            child { node (11) {$5$} }
          };
        \end{scope}
        \draw[bstoutline] \convexpath{1,11,100}{3mm};
      \end{tikzpicture}
          &
            \quad\parbox{5cm}{\raggedright The postfix traversal moves from a right child to a parent whose left subtree is non-empty. This well be the induction step, case 2.}
    \\
    ={}& \psylv{1\olsubword{2354}}   \cyc{} \psylv{\olsubword{2354}1}
    \displaybreak[0]\\
  =T_5 ={}&
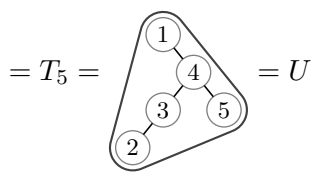
      \begin{tikzpicture}[baseline=(1.base)]
        \begin{scope}[tinybst]
          \node (root) {$1$}
          child[missing]
          child { node (1) {$4$}
            child { node (10) {$3$}
              child { node (100) {$2$} }
              child[missing]
            }
            child { node (11) {$5$} }
          };
        \end{scope}
        \draw[bstoutline] \convexpath{root,11,100}{3mm};
      \end{tikzpicture} = U
    &
      
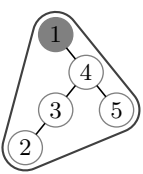
\begin{tikzpicture}[baseline=(1.base)]
        \begin{scope}[tinybst]
          \node[fill=gray] (root) {$1$}
          child[missing]
          child { node (1) {$4$}
            child { node (10) {$3$}
              child { node (100) {$2$} }
              child[missing]
            }
            child { node (11) {$5$} }
          };
        \end{scope}
        \draw[bstoutline] \convexpath{root,11,100}{3mm};
      \end{tikzpicture}
          &
            \quad\parbox{5cm}{\raggedright The postfix traversal moves from a right child to a parent whose left subtree is empty. This well be the induction step, case 4.}
  \end{align*}
\end{example}

Before beginning the proof, it is necessary to set up some terminology and conventions for diagrams of binary search
trees. For brevity, write `the node $x$' instead of `the node labelled by $x$' and `the (sub)tree $\alpha$' instead of
`the (sub)tree with reading $\alpha$'. However, equalities and inequalties always refer to comparisons of labels:
for example, $x = y$ means that the nodes $x$ and $y$ have equal labels, not that they are the same node.

Let $x$ and $y$ be nodes of a binary tree. If $x$ is a descendent of $y$, then $x$ is \defterm{below} $y$ and $y$ is
\defterm{above} $x$. (Note that the terms `above' and `below' do not refer to levels of the tree: the right child of the
root is not above any node in the left subtree.) Let $v$ be the lowest common ancestor of $x$ and $y$. If $x$ is in the
left subtree of $v$ or is $v$ itself, and $y$ is in the right subtree of $v$ or is $v$ itself, and $x$ and $y$ are not
both $v$, then $x$ is \defterm{to the left} of $y$, and $y$ is \defterm{to the right} of $x$. Note that if $x$ is to the
left of $y$, then $x$ is less than or equal to $y$.

Note that it important to distinguish directional terms like `above' and `below' from order terms like `less than' and
`greater than'. The former refer to the position of nodes within the tree, whereas the latter refer to the order of
symbols in the alphabet $\aA$.

As used in this section, a \defterm{subtree} of a binary search tree will always be a rooted subtree. The
\defterm{complete subtree} at a node $x$ is the entire tree below and including $x$. Given a subtree $T'$ of a tree $T$,
the \defterm{left-minimal} subtree of $T'$ in $T$ is the complete subtree at the left child of the left-most node in
$T'$; the \defterm{right-maximal} subtree of $T'$ in $T$ is the complete subtree at the right child of the right-most
node in $T'$.  A node $x$ is \defterm{topmost} if it is above all other nodes labelled $x$. The \defterm{path of left
  child nodes} (respectively, \defterm{path of right child nodes}) from a node $x$ is the path obtained by starting at
the $x$ and entering left (respectively, right) child nodes until a node with no left (respectively, right) child is
encountered.

In diagrams, individual nodes are shown as round, while subtrees as shown as triangles. An edge emerging from the top of
a triangle is the edge running from the root of that subtree to its parent. A vertical edge joining a node to its parent
indicates that the node may be either a left or right child. An edge emerging from the bottom-left of a triangle is the
edge to its left-minimal subtree; an edge emerging from the bottom-right of a triangle is the edge to its right-maximal
subtree. For example, consider the following diagram:
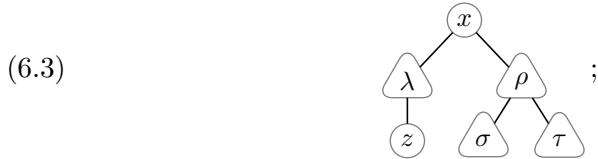
\begin{equation}
\label{eq:bstconv}
\begin{tikzpicture}[smallbst,baseline=-7.5mm]
  \node (root) {$x$}
  child[sibling distance=15mm] { node[triangle] (0) {$\lambda$}
    child { node (00) {$z$} }
  }
  child[sibling distance=15mm] { node[triangle] (1) {$\rho$}
    child { node[triangle] (10) {$\sigma$} }
    child { node[triangle] (11) {$\tau$} }
  };
\end{tikzpicture};
\end{equation}
this shows a tree with root node $x$. Its left subtree consists of the subtree $\lambda$ and a single node $z$ (which
may be a left or right child) whose parent is some node in the subtree $\lambda$. The right subtree of $x$ consists of
subtrees with readings $\rho$, $\sigma$, and $\tau$, with $\sigma$ being the left-minimal subtree of $\rho$ and $\tau$
being the right-maximal subtree of $\rho$. Note that the tree $\rho$ may be deeper than $\sigma$ or $\tau$, in the
sense that the paths leading to its lowest nodes may longer than the paths to the lowest nodes in $\sigma$ or $\tau$.

The strategy in the proof of \fullref{Proposition}{prop:sylvupperbound} below is to pick $T,U \in \sylv_n$ such that
$T \evrel U$ and construct a path of length $n$ from $T$ to $U$ by using a left-to-right postfix traversal of
$U$. Specifically, one considers only the $n$ steps in the left to right postfix traversal that visit the topmost node
labelled by each symbol. (There is a unique topmost node labelled by each symbol by
\fullref{Lemma}{lem:repeatedsymbolssinglepath} below.) Just as in \fullref{Example}{eg:sylvupperbound}, if the node of
$U$ visited at the $h$-th step of the traversal has labelel $x$, then the $h$-th cyclic shift in this path moves a node
$x$ in the $h-1$-th tree in the path to become the root of the tree $h$-th tree. However, the situation is more
complicated because there are other nodes with the same label, and these may be distributed very differently in $T$ and
$U$. For example, the following two binary search trees have the same evaluation, but nodes with the same labels are
distributed differently:
\[
  \begin{tikzpicture}[tinybst,baseline=(000)]
    \node (root) {$7$}
    child { node (0) {$6$}
      child { node (00) {$6$}
        child { node (000) {$2$}
          child { node (0000) {$1$}
            child { node (00000) {$1$}
              child[missing]
              child { node (00001) {$2$} }
            }
            child[missing]
          }
          child { node (0001) {$3$}
            child[missing]
            child { node (00011) {$5$}
              child[missing]
              child { node (000111) {$6$}
                child { node (0001110) {$6$} }
                child[missing]
              }
            }
          }
        }
        child[missing]
      }
      child { node (010) {$7$} }
    }
    child { node (01) {$8$}
      child [missing]
      child { node (011) {$9$} }
    };
  \end{tikzpicture}\;\;,
  \qquad\qquad
  \begin{tikzpicture}[tinybst,baseline=(001)]
    \node (root) {$6$}
    child[sibling distance=20mm] { node (0) {$1$}
      child { node (00) {$1$}
        child[missing]
        child { node (001) {$2$}
          child { node (0010) {$2$} }
          child[missing]
        }
      }
      child { node (01) {$3$}
        child[missing]
        child { node (011) {$6$}
          child { node (0110) {$5$}
            child [missing]
            child { node (01101) {$6$}
              child { node (011010) {$6$} }
              child [missing]
            }
          }
          child[missing]
        }
      }
    }
    child[sibling distance=20mm] { node (1) {$7$}
      child { node (10) {$7$} }
      child { node (11) {$8$}
        child[missing]
        child { node (111) {$9$} }
      }
    };
  \end{tikzpicture}\;\;.
\]

\subsection{Upper bound for diameters II: Properties of trees}
\label{subsec:sylvupperbound2}

This section gathers some properties of trees, and in particular properties of how nodes with the same label are
arranged in binary search trees. These properties are mostly technical but simple to prove.

The \defterm{left-to-right infix traversal} (or simply the \defterm{infix traversal}) of a rooted binary tree $T$ is the
sequence that `visits' every node in the tree as follows: it recursively performs the infix traversal of
the left subtree of the root of $T$, then visits the root of $T$, then recursively performs the infix
traversal of the right subtree of the root of $T$. Thus the infix traversal of any binary tree with the same
shape as the right-hand tree in \eqref{eq:bsteg} visits nodes as follows:
\begin{equation}
\begin{tikzpicture}[tinybst,baseline=-7.5mm]
  \node (root) {}
    child[sibling distance=16mm] { node (0) {}
      child { node (00) {}
        child { node (000) {} }
        child[missing]
      }
      child { node (01) {} }
    }
    child[sibling distance=16mm] { node (1) {}
      child { node (10) {}
        child { node (100) {} }
        child[missing]
      }
      child { node (11) {}
        child[missing]
        child { node (111) {} }
      }
    };
  \begin{scope}[very thick,line cap=round]
    \draw (000.center) edge[bend left=30] (00.center);
    \draw (00.center) edge[bend left=30] (0.center);
    \draw (0.center) edge[bend right=30] (01.center);
    \draw (01.center) edge[bend left=20] (root.center);
    \draw (root.center) edge[bend left=10] (100.center);
    \draw (100.center) edge[bend right=30] (10.center);
    \draw (10.center) edge[bend right=30] (1.center);
    \draw (1.center) edge[bend left=30] (11.center);
    \draw (11.center) edge[bend left=30] (111.center);
  \end{scope}
  \draw[very thick] ($ (000.center) + (-4mm,0) $) -- (000.center);
  \draw[very thick,->] (111.center) -- ($ (111.center) + (4mm,0) $);
  \foreach\x in {000,00,01,0,100,10,111,11,1,root} {
    \draw[draw=black,fill=black] (\x.center) circle (.66mm);
  }
\end{tikzpicture}.
\end{equation}
The following result is immediate from the definition of a binary search tree, but it is used frequently:

\begin{proposition}
  \label{prop:infixreading}
  For any binary search tree $T$, if a node $x$ is encountered before a node $y$ in an infix traversal, then $x \leq y$.
\end{proposition}

Let $U$ be a binary search tree and let $a \in \aA$ be a symbol that appears in $U$.

\begin{lemma}
  \label{lem:repeatedsymbolssinglepath}
  Every node $a$ appears on a single path descending from the root to a leaf; thus there is a unique topmost node $a$.
\end{lemma}

\begin{proof}
  Suppose the node $a$ appeared on two different paths. Let $v$ be the least common ancestor of two such
  appearances. Then $a$ is in both the left and right subtrees of $v$, and so $a \leq v < a$ by the definition of a
  binary search tree, which is a contradiction.
\end{proof}

\begin{lemma}
  \label{lem:onlytopmosthasnonemptyrightsubtree}
  If a node $a$ has a non-empty right subtree, it is the topmost node $a$.
\end{lemma}

\begin{proof}
  Suppose a particular node $a$ has a non-empty right subtree; let $b$ be a symbol in this subtree. Then $a < b$, since
  $b$ is in the right subtree of $a$. If this node $a$ is not the topmost node $a$, then $b$ is also in the left subtree
  of the topmost node $a$ (since the distinguished node $a$ must be in the left subtree of the topmost $a$); hence
  $b \leq a$, which is a contradiction. Thus this node $a$ must be topmost.
\end{proof}

\begin{lemma}
  \label{lem:parentofleftchild}
  Let $x$ be a left child of a parent node $z$. Then the symbol $z$ is the least symbol in the tree $U$ greater than or
  equal to every node in the complete subtree at $x$. Furthermore, if $x$ is not topmost, then $x = z$.
\end{lemma}

\begin{proof}
  In the infix traversal of $U$, the node $z$ is the first node visited after visiting all the nodes in the complete
  subtree at $x$. Since the infix traversal visit nodes in weakly increasing order, $z$ is certainly the least symbol
  greater than or equal to every node in the complete subtree at $x$. Suppose further $x$ is not topmost. Then its right
  subtree is empty and so $z$ is visited immediately after $x$. Since the topmost node $x$ is visited at some later
  point, and since nodes are visited in weakly increasing order, $x \leq z \leq x$ and so $x = z$.
\end{proof}

\begin{lemma}
  \label{lem:leftchildnottopmost}
  Suppose $z$ is not a topmost node. Then it is a left child if and only if its parent is another node $z$.
\end{lemma}

\begin{proof}
  If $z$ is a left child, then its parent is another node $z$ by \fullref{Lemma}{lem:parentofleftchild}.  If the
  parent of $z$ is another node $z$, then the child node must be a left child by the definition of a binary search tree.
\end{proof}

\begin{lemma}
  \label{lem:topmostlessthantopmost}
  Suppose $x$ is a node and let $y$ be a symbol such that $y \leq x$ and $y$ labels some node in the complete
  subtree at $x$. Then the topmost node $y$ is also in the complete subtree at $x$.
\end{lemma}

\begin{proof}
  The result holds trivially if $y = x$, so suppose $y < x$. Then the node $y$ is in the left subtree of $x$ and the
  infix traversal visits the node $y$ in the complete subtree at $x$ before it visits $x$ itself. Since the infix
  traversal visits nodes in weakly increasing order, it must visit the topmost node $y$ before visiting any node
  $x$. Hence the topmost node $y$ must also be in the left subtree of $x$.
\end{proof}

\begin{lemma}
  \label{lem:rightchildnottopmost}
  Suppose $x$ is a topmost node and its right child $z$ is not a topmost node. Then $z$ is the least symbol that is
  greater than every topmost node in the complete subtree at $x$.
\end{lemma}

\begin{proof}
  Since the node $z$ is not topmost, its right subtree is empty. Thus it is the maximum symbol in the complete subtree
  rooted at $x$. For any node $y$ in the left subtree of $x$, the topmost node $y$ is also in the left subtree of $x$ by
  \fullref{Lemma}{lem:topmostlessthantopmost}. For any node $y'$ in the left subtree of $z$, the topmost node $y'$ is
  also in the left subtree of $z$ by \fullref{Lemma}{lem:topmostlessthantopmost}. So $z$ is greater than every topmost
  node in the complete subtree at $x$. Since the symbols labelling these topmost nodes are the ones visited by the infix
  traversal immediately before it visits nodes labelled by $z$, it follows that $z$ is the least symbol that is greater
  than every topmost node in the complete subtree at $x$.
\end{proof}

These results give information about how repeated symbols can appear in a binary search tree. Consider a symbol $z$ that
appears more than once in $U$. If one chooses a node $z$, then one of the following holds:
\begin{itemize}
\item the node $z$ is topmost;
\item the node $z$ is not a topmost node, and is the left child of another node $z$ (by
  \fullref{Lemma}{lem:leftchildnottopmost});
\item the node $z$ is not a topmost node, and the right child of a topmost node $x$, and $z$ is the least symbol greater
  than every topmost node in the complete subtree at $x$ (by \fullref{Lemma}{lem:rightchildnottopmost}).
\end{itemize}

For example, consider the following binary search tree and the repeated symbol $5$:
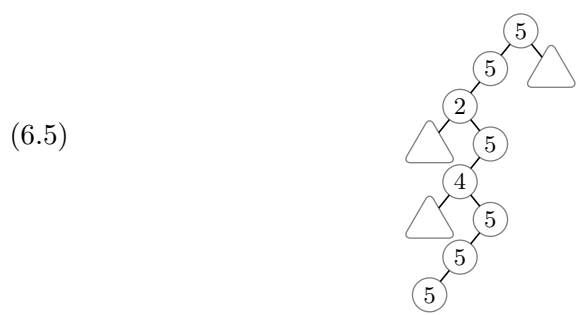
\begin{equation}
\label{eq:bstrepeatedeg}
\begin{tikzpicture}[tinybst,baseline=-15mm]
  \node (root) {$5$}
    child { node (0) {$5$}
      child { node (00) {$2$}
        child { node[triangle] (000) {} }
        child { node (001) {$5$}
          child { node (0010) {$4$}
            child { node[triangle] (00100) {} }
            child { node (00101) {$5$}
              child { node (001010) {$5$}
                child { node (0010100) {$5$} }
                child[missing]
              }
              child[missing]
            }
          }
          child[missing]
        }
      }
      child[missing]
    }
    child { node[triangle] (1) {} };
\end{tikzpicture}
\end{equation}

The \defterm{primary} nodes $a$ are those nodes labelled by $a$ that are consecutive with the topmost node $a$,
including the topmost node itself; in \eqref{eq:bstrepeatedeg} there are two primary nodes $5$. Any node $a$ that has no
children, and any node $a$ consecutive with it, provided they are not primary, are the \defterm{tertiary} nodes $a$; in
\eqref{eq:bstrepeatedeg} there are three tertiary nodes $5$. All other nodes $a$ are \defterm{secondary}. Note that in
each group of consecutive secondary nodes $a$, the uppermost node is a right child of some non-$a$ node, and the
lowermost node has as its left child another non-$a$ node. (Secondary and tertiary nodes always have empty right
subtrees, since they are never topmost.) In \eqref{eq:bstrepeatedeg}, there is one secondary node $5$.

In constructing a path of length $n$ from $T$ to $U$, the $h$-th cyclic shift, corresponding to visiting a topmost node
$y$ of $U$, will move a node $y$ of $T_{h-1}$ to form the root of $T_h$, and must simultaneously deal with any secondary
nodes $z$ that are attached at the right child of $y$ in $U$. Tertiary nodes $z$ either fall into place naturally or are
dealt with during the cyclic shift that moves $z$ to the root. The difficulty is in proving that there always exists a
cyclic shift that performs these tasks. This is the reason for the slightly complicated conditions that form the
inductive statement in the proof of \fullref{Proposition}{prop:sylvupperbound} below.

\subsection{Upper bound for diameters III: Result}
\label{subsec:sylvupperbound3}

This subsection is dedicated entirely to \fullref{Proposition}{prop:sylvupperbound} and its proof. Extensive use will be
made of the concepts and results in \fullref{Subsections}{subsec:sylvupperbound1} and \ref{subsec:sylvupperbound2}.

\begin{proposition}
  \label{prop:sylvupperbound}
  The diameter of any connected component of $K(\sylv_n)$ is at most $n$.
\end{proposition}

\begin{proof}
  Let $T$ and $U$ be in the same connected component of $K(\sylv_n)$. Then $T \evrel U$ and so $T$ and $U$ contain the
  same number of nodes labelled by each symbol. Without loss of generality, assume that every symbol in $\aA_n$ appears
  in $T$ and $U$.

  \bigskip
  \noindent\textit{Preliminaries.}
  Consider the left-to-right postfix traversal of $U$. Modify this traversal so that it only visits topmost nodes; for
  the purposes of this proof, call this the \defterm{topmost traversal}. Since $U$ contains every symbol in $\aA_n$,
  there are exactly $n$ steps in this traversal. Let $u_i$ be the $i$-th node visited in this modified traversal.

  For $h = 1,\ldots,n$, define $U_h = \set{u_1,\ldots,u_h}$ and let $U_h^\uparrow$ be the set of nodes in $U_h$ that do
  not lie below any other node in $U_h$. Since a later step in a left-to-right postfix traversal is never below an
  earlier step, $u_h \in U_h^\uparrow$ for all $h$. (The set $U_h^\uparrow$ will turn out to be the roots of the
  complete subtrees that have been `built' in $T_h$.)

  Let $B_h$ be the complete subtree of $U$ at $u_h$; see \fullref{Figure}{fig:sylvupperbounddefs} for an example of this
  and later definitions.

  Define $m_h$ to be the minimum symbol in $B_h$; thus $m_h$ is the minimum symbol below and including the node $u_h$ in
  $U$. Note that the topmost node $m_h$ is in the subtree $B_h$ by \fullref{Lemma}{lem:topmostlessthantopmost}, and
  must (by minimality) be on the path of left child nodes from $u_h$.

  Define $q_h$ to be the minimum symbol that is greater than every topmost node in $B_h$, with $q_h$ undefined if
  there is no such symbol.
  Note that $q_h$ can be found by following the path from $u_h$ to the root: $q_h$ will be symbol labelling the first
  node entered from a left child, because this node is the node visited by the infix traversal of $U$ immediately after
  the nodes in the subtree $B_h$ have been visited. In particular, $q_h$ is defined precisely when $u_h$ does not lie on
  the path of right child nodes from the root of $U$.

  Define $p_h$ to be the maximum symbol that is less than every symbol in $B_h$, with $p_h$ undefined if there is no
  such symbol.
  Note that $p_h$ can be found by following the path from $u_h$ to the root: $p_h$ will be the symbol labelling the
  first node entered from a right child, because this node is the node visited by the infix traversal of $U$ immediately
  before visiting nodes in the subtree $B_h$. (This process always locates the \emph{topmost} node $p_h$, since it has a
  non-empty right subtree.) In particular, $p_h$ is defined precisely when $u_h$ does not lie on the path of left child nodes
  from the root of $U$.

  Note that $p_h < x \leq q_h$ for all symbols $x$ in $B_h$, ignoring the inequalities involving $p_h$ or $q_h$ when
  these are undefined. Note that the symbol $q_h$ may appear in $B_h$, labelling a secondary or tertiary node. In
  particular, $p_h < m_h \leq u_h < q_h$, since $u_h$ is a topmost symbol.

  Define $C_h$ to the the tree obtained from $B_h$ by deleting all nodes $m_h$ except the topmost. (Note that this
  leaves no `orphaned' subtrees, since only the topmost $m_h$ can have a right subtree, and the bottommost $m_h$ has
  empty left subtree by the minimality of $m_h$.)

  Define $D_h$ to be the tree obtained from $C_h$ by deleting any tertiary nodes $q_h$ from $C_h$.

  If $D_h$ contains no node $q_h$, define $E_h = D_h$. If $D_h$ contains a node $q_h$, define $E_h$ to be the tree
  obtained from $D_h$ by inserting $s$ symbols $q_h$ into $D_h$ using \fullref{Algorithm}{alg:sylvinsertone}, where $s$
  is the difference between the number of nodes $q_h$ in $U$ and the the number of nodes $q_h$ in $D_h$.

  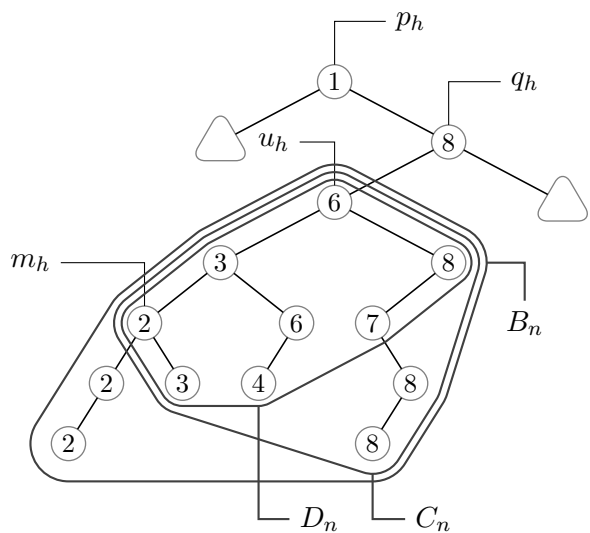
\begin{figure}[t]
    \begin{tikzpicture}
      \begin{scope}[smallbst]
        \node (root) at (0,0) {$1$}
        child[sibling distance=30mm] { node[triangle] {} }
        child[sibling distance=30mm] { node (1) {$8$}
          child[sibling distance=30mm] { node (10)  {$6$}
            child[sibling distance=30mm] { node (100) {$3$}
              child[sibling distance=20mm] { node (1000) {$2$}
                child[sibling distance=10mm] { node (10000) {$2$}
                  child { node (100000) {$2$} }
                  child[missing]
                }
                child[sibling distance=10mm] { node (10001) {$3$} }
              }
              child[sibling distance=20mm] { node (1001) {$6$}
                child[sibling distance=10mm] { node (10010) {$4$} }
                child[missing]
              }
            }
            child[sibling distance=30mm] { node (101) {$8$}
              child[sibling distance=20mm] { node (1010) {$7$}
                child[missing]
                child[sibling distance=10mm] { node (10101) {$8$}
                  child { node (101010) {$8$} }
                  child[missing]
                }
              }
              child[missing]
            }
          }
          child[sibling distance=30mm] { node[triangle] {} }
        };
      \end{scope}
      \node (ph) at ($ (root) + (10mm,8mm) $) {$p_h$};
      \node (zh) at ($ (1) + (10mm,8mm) $) {$q_h$};
      \node (uh) at ($ (10) + (-8mm,8mm) $) {$u_h$};
      \node (mh) at ($ (1000) + (-15mm,8mm) $) {$m_h$};
      \draw (root) |- (ph);
      \draw (1) |- (zh);
      \draw (10) |- (uh);
      \draw (1000) |- (mh);
      \draw[bstoutline] \convexpath{10,101,1010,10010,10001,1000,100}{3mm};
      \draw[bstoutline] \convexpath{10,101,10101,101010,10001,1000,100}{4mm};
      \draw[bstoutline] \convexpath{10,101,10101,101010,100000,1000,100}{5mm};
      \node (bn) at ($ (101) + (10mm,-8mm) $) {$B_n$};
      \node (cn) at ($ (101010) + (8mm,-10mm) $) {$C_n$};
      \node (dn) at ($ (10010) + (8mm,-18mm) $) {$D_n$};
      \draw[bstoutline] ($ (101) + (5mm,0) $) -| (bn);
      \draw[bstoutline] ($ (101010) + (0,-4mm) $) |- (cn);
      \draw[bstoutline] ($ (10010) + (0,-3mm) $) |- (dn);
    \end{tikzpicture}
    \caption{Example tree $U$ illustrating definitions of notation. Here, $u_h$ is the topmost node $5$, and thus
      $p_h = 1$ and $q_h = 7$. The subtree $B_h$ consists of the complete subtree at $u_h$, and the minimum symbol in
      this subtree is $m_h = 2$. The subtree $C_h$ is obtained from $B_h$ by deleting all nodes $2$ except the topmost,
      and $D_h$ is obtained from $C_h$ by deleting the tertiary nodes $8$. Since there is a [secondary] node $8$ in
      $D_h$, the tree $E_h$ is obtained from $D_h$ by inserting three symbols $8$ using
      \fullref{Algorithm}{alg:sylvinsertone}, since there are three nodes $8$ in the tree $U$ outside $D_h$.}
    \label{fig:sylvupperbounddefs}
  \end{figure}

  Notice that the set of symbols in $E_h$ is the same as the set of symbols in $D_h$ (in either case), and is contained
  in the set of symbols in $C_h$ (this containment will be strict if in $C_h$ there are tertiary nodes $q_h$ but no
  secondary nodes $q_h$), which in turn is equal to the set of symbols in $B_h$.

  Suppose $B_h$ contains a node $x$, but that $B_h$ does not contain the topmost node $x$. Then by
  \fullref{Lemma}{lem:repeatedsymbolssinglepath}, $B_h$ is below the topmost node $x$ and must be in its left subtree
  since it contains a node $x$. In following the path from $u_h$ to the root, $q_h$ labels the first node entered from a
  left child. Hence $q_h \leq x$. However, $x$ appears in $B_h$, so $x \leq q_h$. Hence $x = q_h$. Thus $q_h$ is the
  only symbol that can label a node in $B_h$ but whose corresponding topmost node lies outside $B_h$. By
  their definitions, the same applies to $C_h$, $D_h$, and $E_h$.

  \bigskip
  \noindent\textit{Statement of induction.}
  The aim is to construct inductively a sequence $T = T_0, T_1, \ldots, T_n = U$ with $T_i \cyc T_{i+1}$ for
  $i \in \set{1,\ldots,n-1}$. Let $h = 1,\ldots,n$ and suppose $U_h^\uparrow = \set{u_{i_1},\ldots,u_{i_k}}$ (where
  $i_1 < \ldots < i_k = h$). Then the tree $T_h$ will satisfy the following four conditions P1--P4:
  \begin{itemize}
  \item[P1] The subtree $E_{i_k}$ appears at the root of $T_h$.
  \item[P2] The subtrees $E_{i_k},\ldots,E_{i_1}$ appear, in that order, on the path of left child nodes from the root of $T_h$. (These subtrees may be separated by other nodes on the path of left child nodes.)
  \item[P3] For $j = 1,\ldots,k$, every node below $E_{i_j}$ is in its left-minimal or right-maximal subtrees in $T_h$.
  \item[P4] For $j = 1,\ldots,k$, no node $m_{i_j}$ is below a node $p_{i_j}$ in $T_h$.
  \end{itemize}
  (Note that conditions P1--P4 do not apply to $T_0 = T$, which is an arbitrary element of $\sylv_n$.)

  \bigskip
  \noindent\textit{Base of induction.} The base of the induction is to apply a cyclic shift to $T_0$ and obtain a tree
  $T_1$ that satisfies conditions P1--P4.

  Note first that $U_1^\uparrow = U_1 = \set{u_1}$. By the definition of the topmost traversal, there can be no topmost
  nodes below $u_1$ in $U$. Since $m_1 \leq u_1$, the topmost node $m_1$ is in $B_1$ by
  \fullref{Lemma}{lem:topmostlessthantopmost}, and thus $m_1 = u_1$. Thus $B_1$ consists only of symbols $u_1$ and
  possibly tertiary nodes $q_1$. Thus $C_1$ consists only of the topmost node $m_1 = u_1$ and possibly tertiary nodes
  $q_1$, and so $D_1$ consists only of the single node $u_1$. Thus $E_1 = D_1$ since $D_1$ does not contain a symbol
  $q_1$.

  There are two cases to consider:

  \medskip
  \noindent\textit{Case 1.} Suppose that there is some node with label $u_1$ below some node with label $p_1$ in $T_0$. (Note that this case can
  only hold when $p_1$ is defined, or, equivalently, if $u_1$ is not on the path of left child nodes from the root of
  $U$.)

  In $T_0$, distinguish the uppermost node with label $u_1$ that lies below some node with label $p_1$. (Note that
  although $u_1$ and $p_1$ are defined in terms of the tree $U$, here they are used to pick out nodes with the same
  labels in the tree $T_0$.) Since $p_1 < u_1$, this node $u_1$ must be in the right subtree of the node $p_1$. Thus
  this node $p_1$ must be a topmost node since it has non-empty right subtree. As shown in
  \fullref{Figure}{fig:sylvbasecase1}, let $\zeta$ be a reading of the part of $T_0$ outside the complete subtree at the
  topmost node $p_1$, let $\alpha$ be a reading of the left subtree of the topmost node $p_1$, let $\delta$ be a reading
  of the right subtree of the topmost node $p_1$ outside the complete subtree at the distinguished node $u_1$, and let
  $\beta$ and $\gamma$ be readings of the left and right subtrees of this node $u_1$. All nodes $p_1$ other than the
  distinguished topmost node $p_1$ must be in $\alpha$. There are no nodes $u_1$ in $\delta$, by the choice of the
  distinguished node $u_1$. It is possible that there may be nodes $u_1$ in $\beta$ or in $\zeta$. (If there is a node
  $u_1$ in $\zeta$, then the distinguished node $u_1$ is not topmost and so $\gamma$ is empty.)
  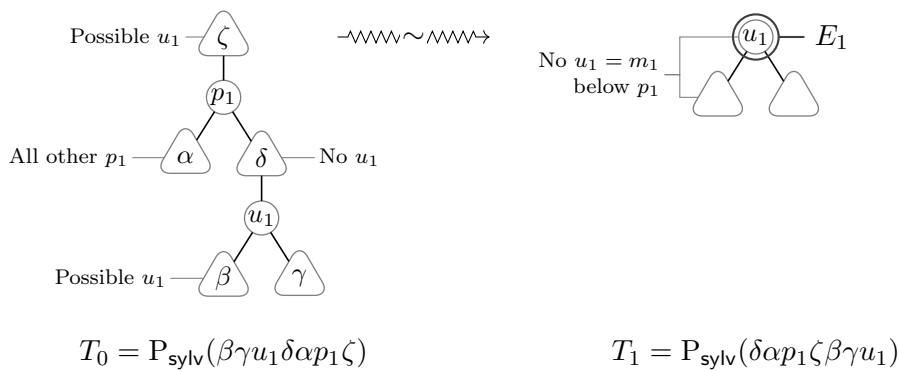
\begin{figure}[tb]
    \centering
    \begin{tikzpicture}
      \begin{scope}[smallbst]
        \node[triangle] (aroot) at (0,0) {$\zeta$}
        child { node (a0) {$p_1$}
          child { node[triangle] (a00) {$\alpha$} }
          child { node[triangle] (a01) {$\delta$}
            child { node (a010) {$u_1$}
              child { node[triangle] (a0100) {$\beta$} }
              child { node[triangle] (a0101) {$\gamma$} }
            }
          }
        };
      \end{scope}
      \node[lcomment] (arootcomment) at ($ (aroot) + (-5mm,0) $) {Possible $u_1$};
      \node[lcomment] (a00comment) at ($ (a00) + (-7mm,0) $) {All other $p_1$};
      \node[lcomment] (a0100comment) at ($ (a0100) + (-7mm,0) $) {Possible $u_1$};
      \node[rcomment] (a01comment) at ($ (a01) + (7mm,0) $) {No $u_1$};
      \draw[gray] (arootcomment) -- (aroot);
      \draw[gray] (a00comment) -- (a00);
      \draw[gray] (a0100comment) -- (a0100);
      \draw[gray] (a01comment) -- (a01);
      \node at ($ (aroot) + (0,-42mm) $) {$T_0 = \psylv{\beta\gamma u_1\delta\alpha p_1\zeta}$};
      %
      %
      \begin{scope}[smallbst]
        \node (broot) at (70mm,0) {$u_1$}
        child { node[triangle] (b0) {} }
        child { node[triangle] (b1) {} };
      \end{scope}
      \node[lcomment] (brootb0comment) at ($ (b0) + (-7mm,3mm) $) {No $u_1 = m_1$\\below $p_1$};
      \draw[gray] (brootb0comment.east) -- ++(2mm,0mm) |- (broot);
      \draw[gray] (brootb0comment.east) -- ++(2mm,0mm) |- (b0);
      \draw[bstoutline] (broot) circle[radius=3mm];
      \node (e1) at ($ (broot) + (10mm,0mm) $) {$E_1$};
      \draw[bstoutline] ($ (broot) + (3mm,0) $) -- (e1);
      \node at ($ (broot) + (0,-42mm) $) {$T_1 = \psylv{\delta\alpha p_1\zeta\beta\gamma u_1}$};
      \draw ($ (aroot) + (15mm,0) $) edge[mogrifyarrow] node[fill=white,anchor=mid,inner xsep=.25mm,inner ysep=1mm] {$\sim$} ($ (broot) + (-35mm,0) $);
    \end{tikzpicture}
    \caption{Base of induction, case 1: some node $u_1$ lies below some node $p_1$ in $T_0$.}
    \label{fig:sylvbasecase1}
  \end{figure}

  Thus $T_0 = \psylv{\beta\gamma u_1\delta\alpha p_1\zeta}$. Let $T_1 = \psylv{\delta\alpha p_1\zeta\beta\gamma u_1}$;
  note that $T_0 \cyc T_1$.

  In computing $T_1$, the symbol $u_1$ is inserted first and becomes the root node. Since other symbols $u_1$ can only
  appear in $\beta$ and $\zeta$, and other symbols $p_1$ can only appear in $\alpha$, all symbols $m_1 = u_1$ are
  inserted before symbols $p_1$. Thus there is no node $m_1$ below a node $p_1$; thus $T_1$ satisfies P4. Since $E_1$
  consists only of a node $u_1$, the tree $E_1$ appears at the root of $T_1$ and so $T_1$ satisfies P1 and
  P2. Furthermore, every node below $E_1$ is either in the left-minimal subtree of $E_1$ (that is, the left subtree of
  the root node $u_1$) or the right-maximal subtree of $E_1$ (that is, right subtree of $u_1$), and so $T_1$ satisfies
  P3.

  \medskip
  \noindent\textit{Case 2.} Suppose that no node labelled $u_1$ lies below a node with label $p_1$ in $T_0$. (This case always holds when
  $p_1$ is undefined.)

  Distinguish the topmost node $u_1$ in $T_0$. As shown in \fullref{Figure}{fig:sylvbasecase2}, let $\zeta$ be a reading
  of the part of $T_0$ outside the complete subtree at the topmost node $u_1$. Since no node $u_1$ lies below a node
  $p_1$, there exists a reading $\beta$ of the left subtree of the topmost node $u_1$ in which all symbols $p_1$ appear
  before all symbols $u_1$. Let $\gamma$ be a reading of the right subtree of the topmost node $u_1$.
  \begin{figure}[tb]
    \centering
    \begin{tikzpicture}
      \begin{scope}[smallbst]
        \node[triangle] (aroot) at (0,0) {$\zeta$}
        child { node (a0) {$u_1$}
          child { node[triangle] (a00) {$\beta$} }
          child { node[triangle] (a01) {$\gamma$} }
        };
      \end{scope}
      \node[lcomment] (arootb0comment) at ($ (a00) + (-7mm,8mm) $) {No $u_1 = m_1$\\below $p_1$};
      \draw[gray] (arootb0comment.east) -- ++(2mm,0mm) |- (aroot);
      \draw[gray] (arootb0comment.east) -- ++(2mm,0mm) |- (a0);
      \draw[gray] (arootb0comment.east) -- ++(2mm,0mm) |- (a00);
      \node at ($ (aroot) + (0,-26mm) $) {$T_0 = \psylv{\beta\gamma u_1\zeta}$};
      %
      %
      \begin{scope}[smallbst]
        \node (broot) at (70mm,0) {$u_1$}
        child { node[triangle] (b0) {} }
        child { node[triangle] (b1) {} };
      \end{scope}
      \node[lcomment] (brootb0comment) at ($ (b0) + (-7mm,3mm) $) {No $u_1 = m_1$\\below $p_1$};
      \draw[gray] (brootb0comment.east) -- ++(2mm,0mm) |- (broot);
      \draw[gray] (brootb0comment.east) -- ++(2mm,0mm) |- (b0);
      \draw[bstoutline] (broot) circle[radius=3mm];
      \node (e1) at ($ (broot) + (10mm,0mm) $) {$E_1$};
      \draw[bstoutline] ($ (broot) + (3mm,0) $) -- (e1);
      \node at ($ (broot) + (0,-26mm) $) {$T_1 = \psylv{\zeta\beta\gamma u_1}$};
      \draw ($ (aroot) + (15mm,0) $) edge[mogrifyarrow] node[fill=white,anchor=mid,inner xsep=.25mm,inner ysep=1mm] {$\sim$} ($ (broot) + (-35mm,0) $);
    \end{tikzpicture}
    \caption{Base of induction, case 2: no node $u_1$ lies below a node $p_1$ in $T_0$.}
    \label{fig:sylvbasecase2}
  \end{figure}
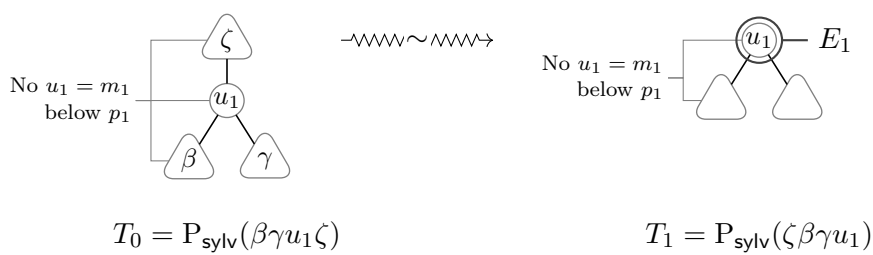

  Thus $T_0 = \psylv{\beta\gamma u_1\zeta}$. Let $T_1 = \psylv{\zeta\beta\gamma u_1}$; note that $T_0 \cyc T_1$.

  In computing $T_1$, the symbol $u_1$ is inserted first and becomes the root node. Since other symbols $u_1$ must
  appear after symbols $p_1$ in $\beta$, all symbols $m_1 = u_1$ are inserted before symbols $p_1$. Thus there is no
  node $m_1$ below a node $p_1$; thus $T_1$ satisfies P4. Since $E_1$ consists only of a node $u_1$, the tree $T_1$
  satisfies P1--P3 by the same reasoning as in Case~1.

  \medskip
  This completes the base of the induction: the tree $T_1$ satisfies P1--P4 and $T_0 \cyc T_1$.

  \bigskip
  \noindent\textit{Induction step.} Let $h \in \set{1,\ldots,n-1}$ and suppose that
  $U_h^\uparrow = \set{u_{i_1},\ldots,u_{i_k}}$ (where $i_1 < \ldots < i_k$). Recall that $h = i_k$; for brevity, let
  $g = i_{k-1}$. Suppose that the tree $T_h$ satisfies conditions P1--P4. The aim is to apply a cyclic shift to $T_h$
  and obtain a tree $T_{h+1}$ that satisfies conditions P1--P4.

  There are four cases, depending on the relatives positions of $u_h$ and $u_{h+1}$ in $U$:
  \begin{enumerate}
  \item $u_h$ is in the left subtree of $v$ and $u_{h+1}$ is in the right subtree of $v$, where $v$ is the lowest common
    ancestor of $u_h$ and $u_{h+1}$ in $U$;
  \item $u_h$ is in the left subtree of $u_{h+1}$;
  \item $u_h$ is in the right subtree of $u_{h+1}$, and there is no node $u_i$ in the left subtree of $u_{h+1}$;
  \item $u_h$ is in the right subtree of $u_{h+1}$, and there is some node $u_i$ in the left subtree of $u_{h+1}$.
  \end{enumerate}

  \medskip
  \noindent\textit{Case 1.} Suppose that, in $U$, the node $u_h$ is in the left subtree of $v$ and $u_{h+1}$ is in the
  right subtree of $v$, where $v$ is the lowest common ancestor of $u_h$ and $u_{h+1}$ in $U$.

  In this case, $U_{h+1}^\uparrow = U_h^\uparrow \cup \set{u_{h+1}}$.  By the definition of the topmost traversal, there
  are no topmost nodes below $u_{h+1}$ in $U$. Since $m_{h+1} \leq u_{h+1}$, the topmost node $m_{h+1}$ is in $B_{h+1}$ by
  \fullref{Lemma}{lem:topmostlessthantopmost}, and thus $m_{h+1} = u_{h+1}$. Thus $B_{h+1}$ consists only of symbols $u_{h+1}$ and
  possibly tertiary nodes $q_{h+1}$. Thus $C_{h+1}$ consists only of the topmost node $m_{h+1} = u_{h+1}$ and possibly tertiary nodes
  $q_{h+1}$, and so $D_{h+1}$ consists only of the single node $u_{h+1}$. Thus $E_{h+1} = D_{h+1}$ since $D_{h+1}$ does not contain a symbol
  $q_{h+1}$.

  There are two sub-cases:

  \smallskip
  \noindent\textit{Sub-case 1(a).} Suppose that there is some node labelled $u_{h+1}$ below some node with label $p_{h+1}$ in
  $T_h$ and that $q_h \neq p_{h+1}$.

  Since $u_{h+1}$ is in the right subtree of $v$, the symbol $p_{h+1}$ is greater than or equal to $v$. The symbol $v$
  is greater than or equal to every symbol in $B_h$. It is impossible that $p_{h+1}$ is equal to some symbol in $B_h$,
  since this would require $p_{h+1} = v = q_h$, which is excluded from this sub-case. Hence $p_{h+1}$ is strictly
  greater than every symbol in $B_h$ and thus in $E_h$.

  The tree $E_h$ appears at the root of $T_h$ by P1. Since $p_{h+1}$ is strictly
  greater than every symbol in $E_h$, every symbol $p_{h+1}$ must be in the right-maximal subtree of $E_h$ in
  $T_h$ by P3. Distinguish the uppermost node $u_{h+1}$ that lies below some node $p_{h+1}$; since $p_{h+1} < u_{h+1}$,
  this node $u_{h+1}$ must be in the right subtree of the node $p_{h+1}$. Thus this node $p_{h+1}$ must be a topmost
  node since it has non-empty right subtree.

  As shown in \fullref{Figure}{fig:sylvinductioncase1a}, let $\lambda$ be a reading of the left-minimal subtree of
  $E_h$; note that $\lambda$ contains all the subtrees $E_{i_j}$ by P2. Let $\zeta$ be a reading of the right-maximal
  subtree of $E_h$ outside the complete subtree at the topmost node $p_{h+1}$, let $\alpha$ be a reading of
  the left subtree of the topmost node $p_{h+1}$, let $\delta$ be a reading of the right subtree of the topmost node $p_{h+1}$
  outside the complete subtree at the distinguished node $u_{h+1}$, and let $\beta$ and $\gamma$ be readings of the left and
  right subtrees of this node $u_{h+1}$. All other nodes $p_{h+1}$ must be in $\alpha$. There are no nodes $u_{h+1}$ in $\delta$, by
  the choice of the distinguished node $u_{h+1}$. It is possible that there may be nodes $u_{h+1}$ in $\beta$ or in $\zeta$. (If
  there is a node $u_{h+1}$ in $\zeta$, then the distinguished node $u_{h+1}$ is not topmost and so $\gamma$ is empty.)
  \begin{figure}[tb]
    \centering
    \begin{tikzpicture}
      \begin{scope}[smallbst]
        \node[triangle] (aroot) at (0,0) {$E_h$}
        child { node[triangle] (a0) {$\lambda$} }
        child { node[triangle] (a1) {$\zeta$}
          child { node (a10) {$p_{h+1}$}
            child { node[triangle] (a100) {$\alpha$} }
            child { node[triangle] (a101) {$\delta$}
              child { node (a1010) {$u_{h+1}$}
                child { node[triangle] (a10100) {$\beta$} }
                child { node[triangle] (a10101) {$\gamma$} }
              }
            }
          }
        };
      \end{scope}
      \node[lcomment] (aroota0comment) at ($ (a0) + (-7mm,4mm) $) {No $m_{i_j}$\\below $p_{i_j}$};
      \node[rcomment] (a1comment) at ($ (a1) + (7mm,0) $) {Possible $u_{h+1}$};
      \node[lcomment] (a100comment) at ($ (a100) + (-7mm,0) $) {All other $p_{h+1}$};
      \node[lcomment] (a10100comment) at ($ (a10100) + (-7mm,0) $) {Possible $u_{h+1}$};
      \node[rcomment] (a101comment) at ($ (a101) + (7mm,0) $) {No $u_{h+1}$};
      \draw[gray] (aroota0comment.east) -- ++ (2mm,0) |- (aroot);
      \draw[gray] (aroota0comment.east) -- ++ (2mm,0) |- (a0);
      \draw[gray] (a1comment) -- (a1);
      \draw[gray] (a100comment) -- (a100);
      \draw[gray] (a10100comment) -- (a10100);
      \draw[gray] (a101comment) -- (a101);
      \node at ($ (aroot) + (0,-50mm) $) {$T_h = \psylv{\beta\gamma u_{h+1}\delta\alpha p_{h+1}\zeta\lambda E_h}$};
      %
      %
      \begin{scope}[smallbst]
        \node (broot) at (70mm,0) {$u_{h+1}$}
        child { node[triangle] (b0) {}
          child { node[triangle] (b00) {$E_h$}
            child { node[triangle] (b000) {$\lambda$} }
            child { node[triangle] (b001) {} }
          }
          child[missing]
        }
        child { node[triangle] (b1) {} };
      \end{scope}
      \node[lcomment] (brootb0comment) at ($ (b0) + (-7mm,4mm) $) {No $u_{h+1} = m_{h+1}$\\below $p_{h+1}$};
      \node[lcomment] (b00b000comment) at ($ (b000) + (-7mm,4mm) $) {No $m_{i_j}$\\below $p_{i_j}$};
      \draw[gray] (brootb0comment.east) -- ++ (2mm,0) |- (broot);
      \draw[gray] (brootb0comment.east) -- ++ (2mm,0) |- (b0);
      \draw[gray] (b00b000comment.east) -- ++ (2mm,0) |- (b00);
      \draw[gray] (b00b000comment.east) -- ++ (2mm,0) |- (b000);
      \coordinate (brootw) at ($ (broot) + (-2mm,0) $);
      \coordinate (broote) at ($ (broot) + (2mm,0) $);
      \draw[bstoutline] \convexpath{brootw,broote}{3mm};
      \node (e1) at ($ (broot) + (13mm,0mm) $) {$E_{h+1}$};
      \draw[bstoutline] ($ (broot) + (5mm,0) $) -- (e1);
      \node at ($ (broot) + (0,-50mm) $) {$T_{h+1} = \psylv{\delta\alpha p_{h+1}\zeta\lambda E_h \beta\gamma u_{h+1}}$};
      \draw ($ (aroot) + (15mm,0) $) edge[mogrifyarrow] node[fill=white,anchor=mid,inner xsep=.25mm,inner ysep=1mm] {$\sim$} ($ (broot) + (-40mm,0) $);
    \end{tikzpicture}
    \caption{Induction step, sub-case 1(a): $E_h = D_h$ and some node $u_{h+1}$ lies below some node $p_{h+1}$ in $T_h$.}
    \label{fig:sylvinductioncase1a}
  \end{figure}

  Thus
  \[
    T_h = \psylv{\beta\gamma u_{h+1}\delta\alpha p_{h+1}\zeta\lambda E_h}.
  \]
  Let
  \[
    T_{h+1} = \psylv{\delta\alpha p_{h+1}\zeta\lambda E_h\beta\gamma u_{h+1}};
  \]
  notice that $T_h \cyc T_{h+1}$.

  In computing $T_{h+1}$, the symbol $u_{h+1}$ is inserted first and becomes the root node. Since other symbols
  $u_{h+1}$ can only appear in $\beta$ and $\zeta$, all symbols $m_{h+1} = u_{h+1}$ are inserted before symbols
  $p_{h+1}$. Thus there is no node $m_{h+1}$ below a node $p_{h+1}$. Since $E_{h+1}$ consists only of a node $u_{h+1}$,
  the tree $T_{h+1}$ satisfies P1. Since every symbol in the trees $\lambda$ and $E_h$ are strictly less than every
  other symbol, these trees reinserted on the path of left child nodes in the same way; hence $T_{h+1}$ satisfies P2
  since $T_h$ does, and satisfies P4 since $T_h$ does and since there is no node $m_{h+1}$ below a node
  $p_{h+1}$. Finally, $T_{h+1}$ satisfies P3 because all the $E_{i_j}$ in $\lambda$ satisfy the condition in P3 (since
  $T_h$ satisfies P3), and trivially $E_{h+1}$ satisfies the condition in P3.

  \smallskip
  \noindent\textit{Sub-case 1(b).} Suppose that that no node $u_{h+1}$ lies below a node $p_{h+1}$ in
  $T_h$, or that $p_{h+1} = q_h$.

  The tree $E_h$ appears at the root of $T_h$ by P1. The symbol $u_{h+1}$ is strictly greater than every symbol in $E_h$
  and so every node $u_{h+1}$ must be in the right-maximal subtree of $E_h$ in $T_h$ by P3. Distinguish the topmost node
  $u_{h+1}$ in $T_h$. As shown in \fullref{Figure}{fig:sylvinductioncase1b}, let $\lambda$ be a reading of the
  left-minimal subtree of $E_h$; note that $\lambda$ contains all the subtrees $E_{i_j}$ by P2. Let $\zeta$ be a reading
  of the right-maximal subtree of $E_h$ outside the complete subtree at the topmost node $u_{h+1}$. Note that no symbol
  $u_{h+1}$ appears in $\zeta$. If no node $u_{h+1}$ lies below a node $p_{h+1}$ in $T_h$, choose a reading $\beta$ of
  the left subtree of the topmost node $u_{h+1}$ in which all symbols $p_{h+1}$ appear before all symbols $u_{h+1}$. On
  the other hand, if $p_{h+1} = z_h$, so that every node with this label in $T_h$ is in the subtree $E_h$, then fix any
  reading $\beta$ of the left subtree of the topmost node $u_{h+1}$; then $\beta$ vacuously has the same property (since
  it contains no symbols $p_{h+1}$ at all). Let $\gamma$ be a reading of the right subtree of the topmost node $u_{h+1}$.
  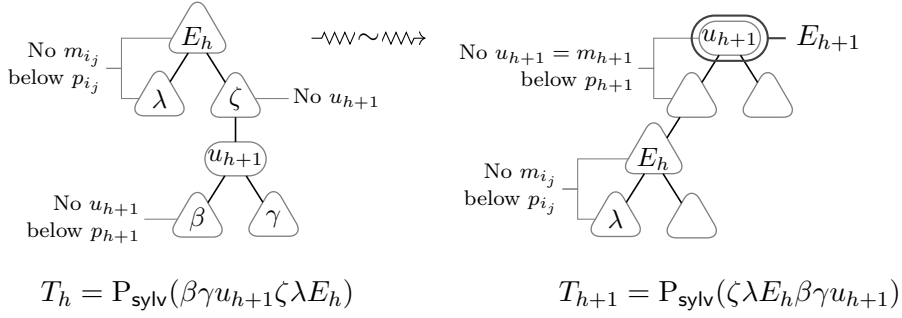
\begin{figure}[tb]
    \centering
    \begin{tikzpicture}
      \begin{scope}[smallbst]
        \node[triangle] (aroot) at (0,0) {$E_h$}
        child { node[triangle] (a0) {$\lambda$} }
        child { node[triangle] (a1) {$\zeta$}
          child { node (a10) {$u_{h+1}$}
            child { node[triangle] (a100) {$\beta$} }
            child { node[triangle] (a101) {$\gamma$} }
          }
        };
      \end{scope}
      \node[lcomment] (aroota0comment) at ($ (a0) + (-7mm,4mm) $) {No $m_{i_j}$\\below $p_{i_j}$};
      \node[rcomment] (a1comment) at ($ (a1) + (7mm,0) $) {No $u_{h+1}$};
      \node[lcomment] (a100comment) at ($ (a100) + (-7mm,0) $) {No $u_{h+1}$\\below $p_{h+1}$};
      \draw[gray] (aroota0comment.east) -- ++ (2mm,0) |- (aroot);
      \draw[gray] (aroota0comment.east) -- ++ (2mm,0) |- (a0);
      \draw[gray] (a1comment) -- (a1);
      \draw[gray] (a100comment) -- (a100);
      \node at ($ (aroot) + (0,-34mm) $) {$T_h = \psylv{\beta\gamma u_{h+1}\zeta\lambda E_h}$};
      %
      %
      \begin{scope}[smallbst]
        \node (broot) at (70mm,0) {$u_{h+1}$}
        child { node[triangle] (b0) {}
          child { node[triangle] (b00) {$E_h$}
            child { node[triangle] (b000) {$\lambda$} }
            child { node[triangle] (b001) {} }
          }
          child[missing]
        }
        child { node[triangle] (b1) {} };
      \end{scope}
      \node[lcomment] (brootb0comment) at ($ (b0) + (-7mm,4mm) $) {No $u_{h+1} = m_{h+1}$\\below $p_{h+1}$};
      \node[lcomment] (b00b000comment) at ($ (b000) + (-7mm,4mm) $) {No $m_{i_j}$\\below $p_{i_j}$};
      \draw[gray] (brootb0comment.east) -- ++ (2mm,0) |- (broot);
      \draw[gray] (brootb0comment.east) -- ++ (2mm,0) |- (b0);
      \draw[gray] (b00b000comment.east) -- ++ (2mm,0) |- (b00);
      \draw[gray] (b00b000comment.east) -- ++ (2mm,0) |- (b000);
      \coordinate (brootw) at ($ (broot) + (-2mm,0) $);
      \coordinate (broote) at ($ (broot) + (2mm,0) $);
      \draw[bstoutline] \convexpath{brootw,broote}{3mm};
      \node (e1) at ($ (broot) + (13mm,0mm) $) {$E_{h+1}$};
      \draw[bstoutline] ($ (broot) + (5mm,0) $) -- (e1);
      \node at ($ (broot) + (0,-34mm) $) {$T_{h+1} = \psylv{\zeta\lambda E_h \beta\gamma u_{h+1}}$};
      \draw ($ (aroot) + (15mm,0) $) edge[mogrifyarrow] node[fill=white,anchor=mid,inner xsep=.25mm,inner ysep=1mm] {$\sim$} ($ (broot) + (-40mm,0) $);
    \end{tikzpicture}
    \caption{Induction step, sub-case (b): $E_h = D_h$ and no node $u_{h+1}$ lies below a node $p_{h+1}$ in $T_h$.}
    \label{fig:sylvinductioncase1b}
  \end{figure}

  Thus
  \[
    T_h = \psylv{\beta\gamma u_{h+1}\zeta\lambda E_h}.
  \]
  Let
  \[
    T_{h+1} = \psylv{\zeta\lambda E_h\beta\gamma u_{h+1}};
  \]
  note that $T_h \cyc T_{h+1}$.

  In computing $T_{h+1}$, the symbol $u_{h+1}$ is inserted first and becomes the root node. Since other symbols
  $u_{h+1}$ can only appear in $\beta$, and any symbols $u_{h+1}$ in $\beta$ must appear after symbols $p_{h+1}$, all
  symbols $m_{h+1} = u_{h+1}$ are inserted before symbols $p_{h+1}$. Thus there is no node $m_{h+1}$ below a node
  $p_{h+1}$. Since $E_{h+1}$ consists only of a node $u_{h+1}$, the tree $T_{h+1}$ satisfies P1. Since every symbol in
  the trees $\lambda$ and $E_h$ are strictly less than every other symbol, these trees are re-inserted on the path of
  left child nodes in the same way; hence $T_{h+1}$ satisfies P2 since $T_h$ does, and satisfies P4 since $T_h$ does and
  since there is no node $m_{h+1}$ below a node $p_{h+1}$. Finally, $T_{h+1}$ satisfies P3 because all the $E_{i_j}$ in
  $\lambda$ satisfy the condition in P3 because $T_h$ satisfies P3, and trivially $E_{h+1}$ satisfies the condition in
  P3.

  \medskip
  \noindent\textit{Case 2.} Suppose that, in $U$, the node $u_h$ is in the left subtree of $u_{h+1}$. By the definition of the topmost
  traversal, the right subtree of $u_{h+1}$ contains no node $u_i$.

  In this case, $U_{h+1}^\uparrow = \parens[\big]{U_h^\uparrow \setminus \set{u_h}} \cup \set{u_{h+1}}$. It is immediate
  from the definitions that $p_{h+1} = p_h$ and $m_{h+1} = m_h$. Finally, $q_h$ is defined and $q_h = u_{h+1}$, since
  $u_{h+1}$ will be the next topmost node that the infix traversal of $U$ visits after visiting all nodes in $B_h$, and
  so $u_{h+1}$ must be the topmost node $q_h$ since the infix traversal visits nodes in weakly increasing order.

  \smallskip
  \noindent\textit{Sub-case 2(a).} Suppose that $E_h \neq D_h$. Then $D_h$ contains nodes with label $q_h = u_{h+1}$. By
  the definition of $q_h$, the symbol $u_{h+1}$ is greater than or equal to every symbol in $D_h$. Thus $u_{h+1}$ is the
  rightmost symbol in $D_h$ and this is the node where the right-maximal subtree of $D_h$ (and $E_h$) is attached in
  $T_h$. The tree $E_h$ consists of $D_h$ with $s$ nodes $u_{h+1}$ inserted, where $s$ is the number of nodes $u_{h+1}$
  that appear in $U$ outside of the complete subtree at $u_h$.

  As shown in \fullref{Figure}{fig:sylvinductioncase2a}, let $\lambda$ be a reading of the left-minimal subtree of
  $D_h$; note that the subtree $\lambda$ contains all the $E_{i_j}$ except $E_h$ by P2. Let $\beta$ be a reading of the right-maximal
  subtree of $D_h$.
  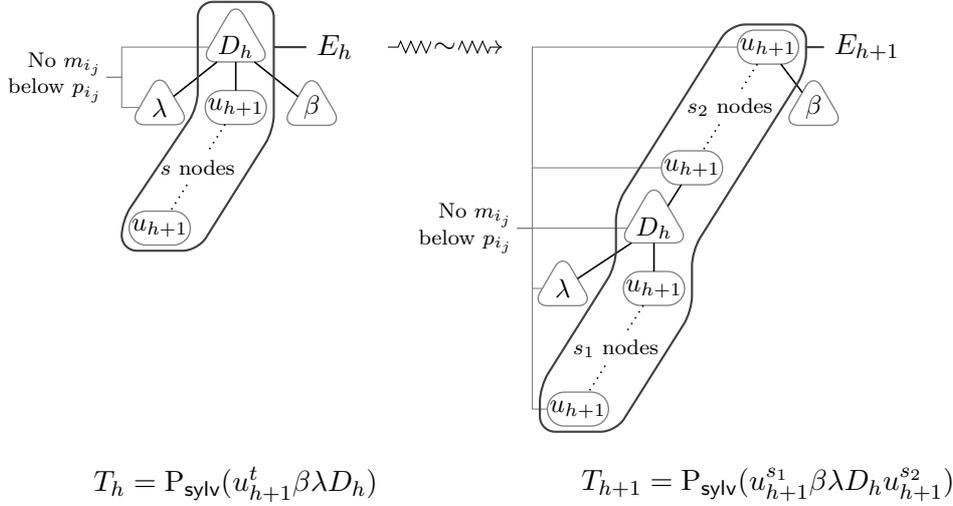
\begin{figure}[tb]
    \centering
    \begin{tikzpicture}
      \begin{scope}[smallbst]
        \node[triangle] (aroot) at (0,0) {$D_h$}
        child { node[triangle] (a0) {$\lambda$} }
        child { node (a1) {$u_{h+1}$}
          child[dotted] { node[nodecount] (a10) {$s$ nodes}
            child { node[solid] (a100) {$u_{h+1}$} }
            child[missing]
          }
          child[missing]
        }
        child { node[triangle] (a2) {$\beta$} };
      \end{scope}
      \node[lcomment] (aroota0comment) at ($ (a0) + (-7mm,4mm) $) {No $m_{i_j}$\\below $p_{i_j}$};
      \draw[gray] (aroota0comment.east) -- ++ (2mm,0) |- (aroot);
      \draw[gray] (aroota0comment.east) -- ++ (2mm,0) |- (a0);
      \draw[bstoutline,rounded corners=2mm]
      ($ (aroot) + (0mm,6mm) $) -|
      ($ (aroot) + (5mm,1mm) $) --
      ($ (a1) + (5mm,-3mm) $) --
      ($ (a100) + (5mm,-3mm) $) --
      ($ (a100) + (0mm,-3mm) $) -|
      ($ (a100) + (-5mm,1mm) $) --
      ($ (a1) + (-5mm,0mm) $) --
      ($ (aroot) + (-5mm,1mm) $) |-
      ($ (aroot) + (0mm,6mm) $);
      \node (e0) at ($ (aroot) + (13mm,0mm) $) {$E_{h}$};
      \draw[bstoutline] ($ (aroot) + (5mm,0) $) -- (e0);
      \node at ($ (aroot) + (0,-58mm) $) {$T_h = \psylv{u_{h+1}^t\beta\lambda D_h}$};
      %
      %
      \begin{scope}[smallbst]
        \node (broot) at (70mm,0) {$u_{h+1}$}
        child[dotted] { node[nodecount] (b0) {$s_2$ nodes}
          child { node[solid] (b00) {$u_{h+1}$}
            child[solid] { node[triangle] (b000) {$D_h$}
              child[sibling distance=12mm] { node[triangle] (b0000) {$\lambda$} }
              child { node (b0001) {$u_{h+1}$}
                child[dotted] { node[nodecount] (b00010) {$s_1$ nodes}
                  child { node[solid] (b000100) {$u_{h+1}$} }
                  child[missing]
                }
                child[missing]
              }
              child[missing]
            }
            child[missing]
          }
          child[missing]
        }
        child[sibling distance=12mm] { node[triangle] (b1) {$\beta$} };
      \end{scope}
      \node[lcomment] (brootb000100comment) at ($ (b000100) + (-8mm,24mm) $) {No $m_{i_j}$\\below $p_{i_j}$};
      \draw[gray] (brootb000100comment.east) -- ++ (2mm,0) |- (broot);
      \draw[gray] (brootb000100comment.east) -- ++ (2mm,0) |- (b00);
      \draw[gray] (brootb000100comment.east) -- ++ (2mm,0) |- (b000);
      \draw[gray] (brootb000100comment.east) -- ++ (2mm,0) |- (b0000);
      \draw[gray] (brootb000100comment.east) -- ++ (2mm,0) |- (b000100);
      \draw[bstoutline,rounded corners=2mm]
      ($ (broot) + (0mm,3mm) $) -|
      ($ (broot) + (5mm,-3mm) $) --
      ($ (b000) + (5mm,-3mm) $) --
      ($ (b0001) + (5mm,-3mm) $) --
      ($ (b000100) + (5mm,-3mm) $) --
      ($ (b000100) + (0mm,-3mm) $) -|
      ($ (b000100) + (-5mm,3mm) $) --
      ($ (b0001) + (-5mm,3mm) $) --
      ($ (b000) + (-5mm,3mm) $) --
      ($ (broot) + (-5mm,3mm) $) --
      ($ (broot) + (0mm,3mm) $);
      \node (e1) at ($ (broot) + (13mm,0mm) $) {$E_{h+1}$};
      \draw[bstoutline] ($ (broot) + (5mm,0) $) -- (e1);
      \node at ($ (broot) + (0,-58mm) $) {$T_{h+1} = \psylv{u_{h+1}^{s_1}\beta\lambda D_hu_{h+1}^{s_2}}$};
      \draw ($ (aroot) + (20mm,0) $) edge[mogrifyarrow] node[fill=white,anchor=mid,inner xsep=.25mm,inner ysep=1mm] {$\sim$} ($ (broot) + (-35mm,0) $);
    \end{tikzpicture}
    \caption{Induction step, sub-case 2(a): $E_h \neq D_h$.}
    \label{fig:sylvinductioncase2a}
  \end{figure}

  Thus
  \[
    T_h = \psylv{u_{h+1}^s\beta\lambda D_h}.
  \]
  Let
  \[
    T_{h+1} = \psylv{u_{h+1}^{s_1}\beta\lambda D_hu_{h+1}^{s_2}},
  \]
  where $s_2$ is the number of primary nodes $u_{h+1}$ in $U$, and where $s_1 = s - s_2$. Notice that
  $T_{h+1} \cyc T_h$.

  In computing $T_{h+1}$, the rightmost symbol $u_{h+1}$ is inserted first and becomes the root
  node. Then the next $s_2-1$ nodes $u_{h+1}$ are attached along the path of left child nodes. Since every symbol in
  $D_h$ or $\lambda$ is less than or equal to $u_{h+1}$, the subtrees $D_h$ and $\lambda$ are re-inserted at the left
  child of the bottommost node $u_{h+1}$. Since every symbol in $\beta$ is strictly greater than $u_{h+1}$, the subtree
  $\beta$ is re-inserted as the right child of the root node $u_{h+1}$. Finally, the remaining $s_1$ symbols $u_{h+1}$
  are re-inserted into $D_h$.

  All secondary nodes $u_{h+1}$ are inside $D_h$, so the $s$ nodes $u_{h+1}$ outside $D_h$ are either primary or
  tertiary. Since there are $s_2$ primary nodes $u_{h+1}$ in $U$, and since the evaluation of $U$ is
  the same as the evaluation of $T_h$, there are $s_1$ tertiary nodes $u_{h+1}$ in $U$. Hence $D_h$ and the other nodes
  $u_{h+1}$ in $T_{h+1}$ together make up $E_{h+1}$; thus $T_{h+1}$ satisfies P1. All the other $E_{i_j}$ are contained
  in $\lambda$, so $T_{h+1}$ satisfies P2. Since $T_h$ satisfies P3, and since $\lambda$ and $\beta$ are the
  left-minimal and right-maximal subtrees of $E_{h+1}$, it follows that $T_{h+1}$ satisfies P3. Finally, the relative
  positions of the $m_{i_j}$ and $p_{i_j}$ have not been altered, so $T_{h+1}$ satisfies P4.

  \smallskip
  \noindent\textit{Sub-case 2(b).} Suppose that $E_h = D_h$. Thus $D_h$ does not contain nodes labelled $q_h =
  u_{h+1}$. Since $u_{h+1}$ is greater than every node in $D_h$, it follows that all nodes $u_{h+1}$ are in the
  right-maximal subtree of $D_h$ (and $E_h$) in $T_h$. Furthermore, since $u_{h+1}$ is the smallest symbol greater than
  every symbol in $D_h$, only another node $u_{h+1}$ can be the left child of a node $u_{h+1}$ in $T_h$.

  As shown in \fullref{Figure}{fig:sylvinductioncase2b}, let $\lambda$ be a reading of the left-minimal subtree of
  $D_h$; note that the subtree $\lambda$ contains all the $E_{i_j}$ except $E_h$ by P2. Let $\delta$ be a reading of the
  right-maximal subtree of $D_h$ outside of the complete subtree at the topmost node $u_{h+1}$. (Note that $\delta$ may
  be empty.) Let $\beta$ be a reading of the right subtree of the topmost node $u_{h+1}$; note that the left subtree of
  this node can only contain other nodes $u_{h+1}$. Suppose there are $s$ nodes $u_{h+1}$ in $U$ (and so in $T_h$).
  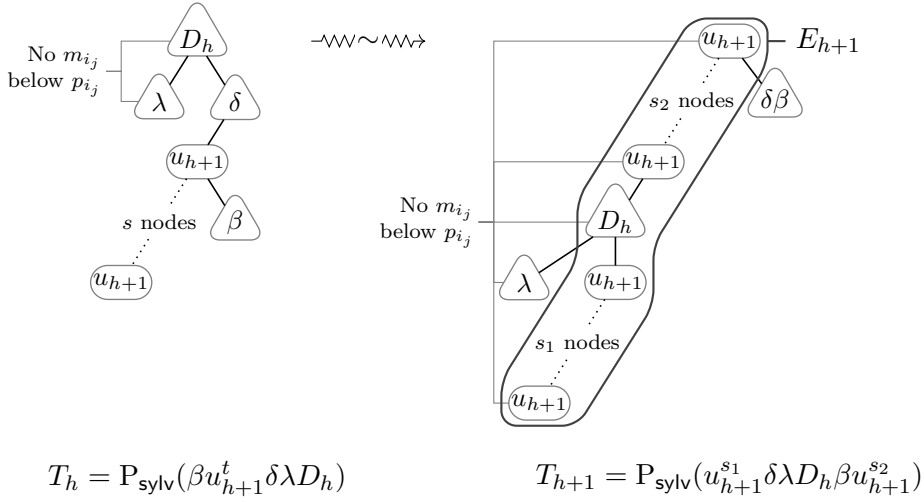
\begin{figure}[tb]
    \centering
    \begin{tikzpicture}
      \begin{scope}[smallbst]
        \node[triangle] (aroot) at (0,0) {$D_h$}
        child { node[triangle] (a0) {$\lambda$} }
        child { node[triangle] (a1) {$\delta$}
          child { node (a10) {$u_{h+1}$}
            child[dotted] { node[nodecount] (a100) {$s$ nodes}
              child { node[solid] (a1000) {$u_{h+1}$} }
              child[missing]
            }
            child { node[triangle] (a101) {$\beta$} }
          }
          child[missing]
        };
      \end{scope}
      \node[lcomment] (aroota0comment) at ($ (a0) + (-7mm,4mm) $) {No $m_{i_j}$\\below $p_{i_j}$};
      \draw[gray] (aroota0comment.east) -- ++ (2mm,0) |- (aroot);
      \draw[gray] (aroota0comment.east) -- ++ (2mm,0) |- (a0);
      \node at ($ (aroot) + (0,-58mm) $) {$T_h = \psylv{\beta u_{h+1}^t\delta\lambda D_h}$};
      %
      %
      \begin{scope}[smallbst]
        \node (broot) at (70mm,0) {$u_{h+1}$}
        child[dotted] { node[nodecount] (b0) {$s_2$ nodes}
          child { node[solid] (b00) {$u_{h+1}$}
            child[solid] { node[triangle] (b000) {$D_h$}
              child[sibling distance=12mm] { node[triangle] (b0000) {$\lambda$} }
              child { node (b0001) {$u_{h+1}$}
                child[dotted] { node[nodecount] (b00010) {$s_1$ nodes}
                  child { node[solid] (b000100) {$u_{h+1}$} }
                  child[missing]
                }
                child[missing]
              }
              child[missing]
            }
            child[missing]
          }
          child[missing]
        }
        child[sibling distance=12mm] { node[triangle] (b1) {$\delta\beta$} };
      \end{scope}
      \node[lcomment] (brootb000100comment) at ($ (b000100) + (-8mm,24mm) $) {No $m_{i_j}$\\below $p_{i_j}$};
      \draw[gray] (brootb000100comment.east) -- ++ (2mm,0) |- (broot);
      \draw[gray] (brootb000100comment.east) -- ++ (2mm,0) |- (b00);
      \draw[gray] (brootb000100comment.east) -- ++ (2mm,0) |- (b000);
      \draw[gray] (brootb000100comment.east) -- ++ (2mm,0) |- (b0000);
      \draw[gray] (brootb000100comment.east) -- ++ (2mm,0) |- (b000100);
      \draw[bstoutline,rounded corners=2mm]
      ($ (broot) + (0mm,3mm) $) -|
      ($ (broot) + (5mm,-3mm) $) --
      ($ (b000) + (5mm,-3mm) $) --
      ($ (b0001) + (5mm,-3mm) $) --
      ($ (b000100) + (5mm,-3mm) $) --
      ($ (b000100) + (0mm,-3mm) $) -|
      ($ (b000100) + (-5mm,3mm) $) --
      ($ (b0001) + (-5mm,3mm) $) --
      ($ (b000) + (-5mm,3mm) $) --
      ($ (broot) + (-5mm,3mm) $) --
      ($ (broot) + (0mm,3mm) $);
      \node (e1) at ($ (broot) + (13mm,0mm) $) {$E_{h+1}$};
      \draw[bstoutline] ($ (broot) + (5mm,0) $) -- (e1);
      \node at ($ (broot) + (0,-58mm) $) {$T_{h+1} = \psylv{u_{h+1}^{s_1}\delta\lambda D_h\beta u_{h+1}^{s_2}}$};
      \draw ($ (aroot) + (15mm,0) $) edge[mogrifyarrow] node[fill=white,anchor=mid,inner xsep=.25mm,inner ysep=1mm] {$\sim$} ($ (broot) + (-40mm,0) $);
    \end{tikzpicture}
    \caption{Induction step, sub-case 2(b): $E_h = D_h$.}
    \label{fig:sylvinductioncase2b}
  \end{figure}

  Thus $T_h = \psylv{\beta u_{h+1}^s\delta\lambda D_h}$. Let
  $T_{h+1} = \psylv{u_{h+1}^{s_1}\delta\lambda D_h\beta u_{h+1}^{s_2}}$, where $s_2$ is the number of primary nodes
  $u_{h+1}$ in $U$, and where $s_1 = s - s_2$. In computing $T_{h+1}$, the rightmost symbol $u_{h+1}$ is inserted first
  and becomes the root node. Then the next $s_2-1$ nodes $u_{h+1}$ are attached along the path of left child
  nodes. Since every symbol in $\beta$ is strictly greater than $u_{h+1}$, the subtree $\beta$ is re-inserted as the
  right child of the root node $u_{h+1}$. Since every symbol in $D_h$ or $\lambda$ is less than or equal to $u_{h+1}$,
  the subtrees $D_h$ and $\lambda$ are re-inserted as the left child of the bottommost node $u_{h+1}$. Since every
  symbol in $\delta$ is strictly greater than $u_{h+1}$, the symbols in $\delta$ are also inserted into the right
  subtree of the root node $u_{h+1}$. Finally, the remaining $s_1$ symbols $u_{h+1}$ are re-inserted into $D_h$.

  By the definition of $D_h$ and the fact that there are $s_2$ primary nodes $u_{h+1}$ in $U$, and the fact that the
  evaluation of $U$ is the same as the evaluation of $T_h$, there are $s_1$ tertiary nodes $u_{h+1}$ in $U$. Hence $D_h$
  and the other nodes $u_{h+1}$ in $T_{h+1}$ together make up $E_{h+1}$; thus $T_{h+1}$ satisfies P1. All the other
  $E_{i_j}$ are contained in $\lambda$, so $E_{h+1}$ satisfies P2. Since $T_h$ satisfies P3, and since $\lambda$ and
  $\delta\beta$ are the left-minimal and right-maximal subtrees of $E_{h+1}$, it follows that $T_{h+1}$ satisfies P3. Finally,
  the relative positions of $m_{i_j}$ and $p_{i_j}$ have not been altered, so $T_{h+1}$ satisfies P4.

  \medskip
  \noindent\textit{Case 3.} Suppose that, in $U$, the node $u_h$ is in the right subtree of $u_{h+1}$, and there is no node $u_i$ in the
  left subtree of $u_{h+1}$. Note that $q_h$ is
  defined if and only if $q_{h+1}$ is defined, in which case $q_h = q_{h+1}$.

  In this case, $U_{h+1}^\uparrow = \parens[\big]{U_h^\uparrow \setminus \set{u_h}} \cup \set{u_{h+1}}$, and $p_h = u_{h+1}$.

  By definition, $B_{h+1}$ is the complete subtree of $U$ at $u_{h+1}$. So $C_{h+1}$ is $B_{h+1}$ with all but the
  topmost node $m_{h+1}$ deleted. Since the left subtree of $u_{h+1}$ in $U$ contains no nodes $u_i$, by
  \fullref{Lemma}{lem:topmostlessthantopmost} it only contains other nodes $u_{h+1}$ and so $m_{h+1} = u_{h+1}$. Thus
  $C_{h+1}$ consists of $u_{h+1}$ and its right subtree. By defintion, $q_{h+1}$ (if defined) is the least symbol
  greater than every topmost symbol in $B_{h+1}$; since $u_{h+1}$ is less than every symbol in $B_h$, this implies that
  $q_{h+1} = q_h$. By definition of the topmost traversal, there is no topmost node between $u_h$ and $u_{h+1}$, so only
  (non-topmost) nodes $q_{h+1}$ can lie between them. Thus $C_{h+1}$ consists of $u_{h+1}$, some number $s_2$ (possibly zero)
  of secondary nodes $q_{h+1}$, and the subtree $B_h$. Hence $D_{h+1}$ consists of $u_{h+1}$, these nodes $q_{h+1}$, and the
  subtree $B_h$ with its tertiary nodes $q_{h+1} = q_h$ deleted, since these tertiary nodes are the same in both $B_{h+1}$ and $B_h$.

  Notice that, in $T_h$, all nodes $m_h$ that are not in $E_h$ itself are in the left-minimal subtree of $E_h$, but may
  not be consecutive. Since no $m_h$ is below $p_h = u_{h+1}$, which is the greatest symbol less than $m_h$, it follows
  that the uppermost node $u_{h+1}$ in $T_h$ is in the left subtree of the lowest node $m_h$.

  Furthermore, as shown in \fullref{Figures}{fig:sylvinductioncase3a} and \ref{fig:sylvinductioncase3b} below, this
  uppermost node $u_{h+1}$ in $T_h$ must be on the path of right child nodes from the left child of the lowest node $m_h$,
  for otherwise it would be in a left subtree of some node $x$ below the lowest node $m_h$, which would imply
  $u_{h+1} = p_h < x < m_h$, which contradicts $p_h$ being the greatest symbol less than $m_h$.

  \smallskip
  \noindent\textit{Sub-case 3(a).} Suppose that $E_h \neq D_h$. Then $q_h$ is defined
  and $D_h$ contains nodes $q_h$. Since $q_h$ is greater than or equal to every node in $D_h$, it follows that the
  uppermost node $q_h$ is where the right-maximal subtree of $D_h$ (and $E_h$) is attached in $T_h$. The tree $E_h$
  consists of $D_h$ with $s$ nodes $q_{h}$ inserted, where $s$ is the number of nodes $q_{h}$ that appear in $U$
  outside of the subtree $D_h$. Suppose there are $r+1$ nodes $m_h$ in $U$, so there are $r$ nodes $m_h$
  outside $D_h$ in $U$.)

  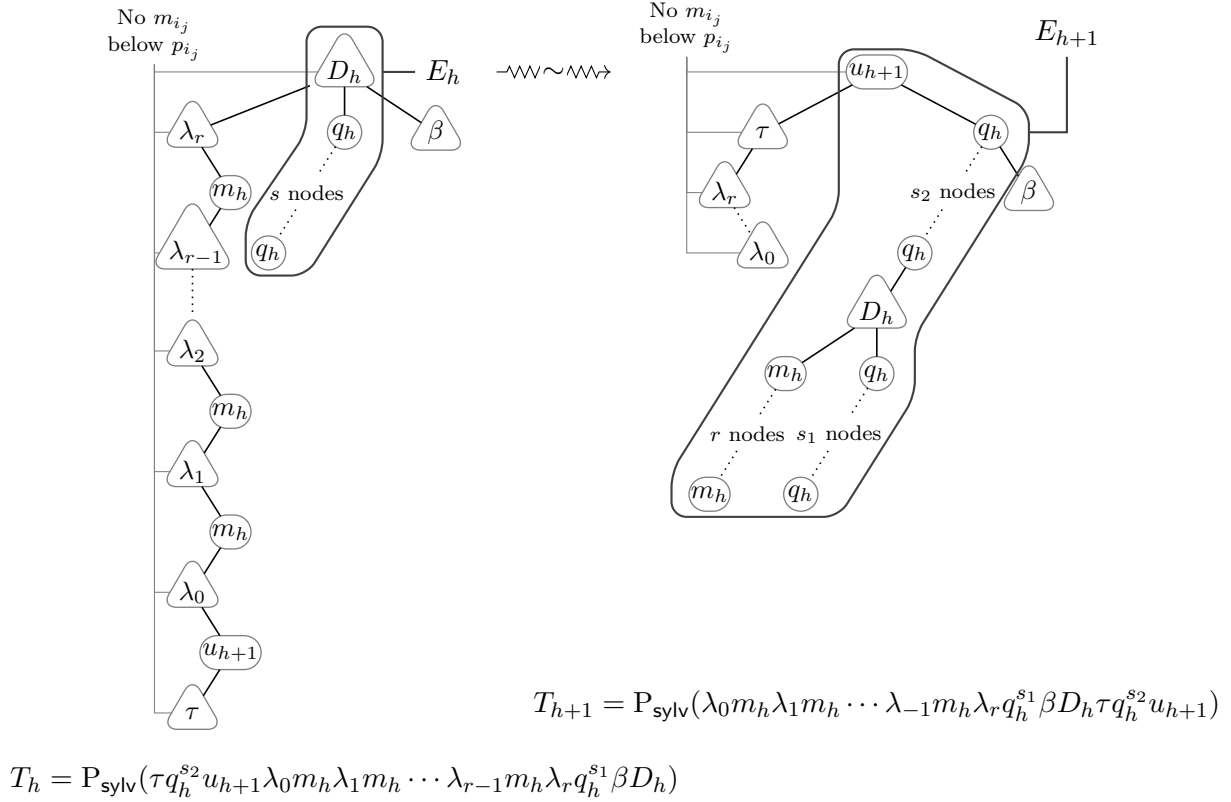
\begin{figure}[tb]
    \centerline{%
    \begin{tikzpicture}
      \begin{scope}[smallbst]
        \node[triangle] (aroot) at (0,0) {$D_h$}
        child[sibling distance=20mm] { node[triangle] (a0) {$\lambda_r$}
          child[missing]
          child { node (a01) {$m_h$}
            child[solid] { node[triangle] (a010) {$\lambda_{r{-}1}$\kern-1mm}
              child[dotted,level distance=13mm] { node[solid,triangle] (a0100) {$\lambda_2$}
                child[missing]
                child[solid,level distance=8mm] { node (a01001) {$m_h$}
                  child { node[solid,triangle] (a010010) {$\lambda_1$}
                    child[missing]
                    child[solid] { node (a0100101) {$m_h$}
                      child[solid] { node[triangle] (a01001010) {$\lambda_0$}
                        child[missing]
                        child { node (a010010101) {$u_{h+1}$}
                          child { node[triangle] (a0100101010) {$\tau$} }
                          child[missing]
                        }
                      }
                      child[missing]
                    }
                  }
                  child[missing]
                }
              }
            }
            child[missing]
          }
        }
        child[sibling distance=20mm] { node (a1) {$q_h$}
          child[dotted] { node[nodecount] (a10) {$s$ nodes}
            child { node[solid] (a100) {$q_h$} }
            child[missing]
          }
          child[missing]
        }
        child[sibling distance=12mm] { node[triangle] (a2) {$\beta$} };
      \end{scope}
      \node[tcomment] (aroota0100101010comment) at ($ (a0100101010) + (-5mm,86mm) $) {No $m_{i_j}$\\below $p_{i_j}$};
      \draw[gray] (aroota0100101010comment.south) |- (aroot);
      \draw[gray] (aroota0100101010comment.south) |- (a0);
      \draw[gray] (aroota0100101010comment.south) |- (a010);
      \draw[gray] (aroota0100101010comment.south) |- (a0100);
      \draw[gray] (aroota0100101010comment.south) |- (a010010);
      \draw[gray] (aroota0100101010comment.south) |- (a01001010);
      \draw[gray] (aroota0100101010comment.south) |- (a0100101010);
      \draw[bstoutline,rounded corners=2mm]
      ($ (aroot) + (0mm,6mm) $) -|
      ($ (aroot) + (5mm,1mm) $) --
      ($ (a1) + (5mm,-3mm) $) --
      ($ (a100) + (5mm,-3mm) $) --
      ($ (a100) + (1mm,-3mm) $) -|
      ($ (a100) + (-3mm,4.2mm) $) --
      ($ (a1) + (-5mm,0mm) $) --
      ($ (aroot) + (-5mm,1mm) $) |-
      ($ (aroot) + (0mm,6mm) $);
      \node (e0) at ($ (aroot) + (13mm,0mm) $) {$E_{h}$};
      \draw[bstoutline] ($ (aroot) + (5mm,0) $) -- (e0);
      \node at ($ (aroot) + (0,-94mm) $) {$T_h = \psylv{\tau q_h^{s_2}u_{h+1}\lambda_0m_h\lambda_1m_h\cdots\lambda_{r-1}m_h\lambda_rq_h^{s_1}\beta D_h}$};
      %
      %
      \begin{scope}[smallbst]
        \node (broot) at (70mm,0) {$u_{h+1}$}
        child[sibling distance=30mm] { node[triangle] (b0) {$\tau$}
          child { node[triangle] (b00) {$\lambda_r$}
            child[missing]
            child[dotted] { node[solid,triangle] (b001) {$\lambda_0$} }
          }
          child[missing]
        }
        child[sibling distance=30mm] { node[solid] (b1) {$q_h$}
          child[dotted] { node[nodecount] (b10) {$s_2$ nodes}
            child { node[solid] (b100) {$q_{h}$}
              child[solid] { node[triangle] (b1000) {$D_h$}
                child[sibling distance=12mm] { node (b10000) {$m_h$}
                  child[sibling distance=10mm,dotted] { node[nodecount] (b100000) {$r$ nodes}
                    child { node[solid] (b1000000) {$m_h$} }
                    child[missing]
                  }
                  child[missing]
                }
                child { node (b10001) {$q_h$}
                  child[dotted] { node[nodecount] (b100010) {$s_1$ nodes}
                    child { node[solid] (b1000100) {$q_h$} }
                    child[missing]
                  }
                  child[missing]
                }
                child[missing]
              }
              child[missing]
            }
            child[missing]
          }
          child { node[triangle] (b2) {$\beta$} }
        };
      \end{scope}
      \node[tcomment] (b0b000comment) at ($ (b00) + (-5mm,18mm) $) {No $m_{i_j}$\\below $p_{i_j}$};
      \draw[gray] (b0b000comment.south) |- (broot);
      \draw[gray] (b0b000comment.south) |- (b0);
      \draw[gray] (b0b000comment.south) |- (b00);
      \draw[gray] (b0b000comment.south) |- (b001);
      \draw[bstoutline,rounded corners=2mm]
      ($ (broot) + (0,3mm) $) --
      ($ (broot) + (5mm,3mm) $) --
      ++(15mm,-8mm) --
      ($ (b1) + (5mm,-3mm) $) --
      ++(-15mm,-24mm) --
      ($ (b10001) + (5mm,-3mm) $) --
      ($ (b1000100) + (5mm,-3mm) $) -|
      ($ (b1000000) + (-5mm,3mm) $) --
      ++(22mm,35.2mm) |-
      ($ (broot) + (0,3mm) $);
      \node[anchor=south] (e1) at ($ (b1) + (10mm,10mm) $) {$E_{h+1}$};
      \draw[bstoutline] ($ (b1) + (5mm,0) $) -| (e1);
      \node at ($ (broot) + (0,-84mm) $)
      {$T_{h+1} = \psylv{\lambda_0m_h\lambda_1m_h\cdots \lambda_{-1}m_h\lambda_rq_h^{s_1}\beta
          D_h\tau q_h^{s_2}u_{h+1}}$};
      \draw ($ (aroot) + (20mm,0) $) edge[mogrifyarrow] node[fill=white,anchor=mid,inner xsep=.25mm,inner ysep=1mm] {$\sim$} ($ (broot) + (-35mm,0) $);
    \end{tikzpicture}}
    \caption{Induction step, sub-case 3(a): $E_h \neq D_h$.}
    \label{fig:sylvinductioncase3a}
  \end{figure}

  As shown in \fullref{Figure}{fig:sylvinductioncase3a}, let $\lambda_r$ be reading of the left-minimal subtree of $D_h$
  outside of the complete subtree at the uppermost $m_h$ below $D_h$. (Note that $\lambda_r$ may be empty.) For
  $i = r-1,\ldots,2$, let $\lambda_i$ be readings of the left subtree of the $i$-th node $m_h$ (counting from the
  lowermost to the uppermostmost) outside of the complete subtree of the $i-1$-th. (Note that $\lambda_i$ will be empty
  if the $i$-th node $m_h$ is the left child of the $i+1$-th.) Let $\lambda_0$ be a reading of the left subtree of the
  bottommost node $m_h$ outside the complete subtree at the uppermost node $u_{h+1}$. Let $\tau$ be a reading of the
  left subtree of the uppermost node $u_{h+1}$.  Note that the uppermost non-empty subtree $\lambda_i$ or $\tau$
  contains all the $E_{i_j}$ except $E_h$ by P2. Let $\beta$ be a reading of the right-maximal subtree of $D_h$.

  Thus
  \[
    T_h = \psylv{\tau q_h^{s_2}u_{h+1}\lambda_0m_h\lambda_1m_h\cdots\lambda_{r-1}m_h\lambda_rq_h^{s_1}\beta D_h};
  \]
  where $s_1 = s - s_2$. (Recall that $s_2$ was defined above as the number of secondary nodes $q_h$ between $u_h$ and
  $p_h = u_{h+1}$ in $U$.) Let
  \[
    T_{h+1} = \psylv{\lambda_0m_h\lambda_1m_h\cdots \lambda_{r-1}m_h\lambda_rq_h^{s_1}\beta D_h\tau q_h^{s_2}u_{h+1}},
  \]
  Notice that $T_h \cyc T_{h+1}$.

  In computing $T_{h+1}$, the symbol $u_{h+1}$ is inserted first and becomes the root node. Then the $s_2$ symbols $q_h$
  are inserted into the right subtree of the root nodes, since $q_h > p_h = u_{h+1}$. Every symbol in $\tau$ is less than or
  equal to $u_{h+1}$ and so is inserted into the left subtree of the root note $u_{h+1}$. Every symbol in $D_h$ is
  greater than $u_{h+1}$ and less than or equal to $q_h$, so the tree $D_h$ is re-inserted at the left child of the
  bottommost $q_h$. Every symbol in $\beta$ is greater than $q_h$ (since in $T_h$ the subtree $\beta$ is the right child
  of a node $q_h$, as discussed above), so $\beta$ is inserted as the right subtree of the topmost $q_h$. The remaining
  $s_1$ symbols $q_h$ are re-inserted below $D_h$, attached at the same position as in $T_h$. The symbol $m_h$ is the
  smallest symbol greater than $u_{h+1}$ and all symbols $m_h$ are inserted into the left-minimal subtree of $D_h$
  (which is attached at the single node $m_h$ in $D_h$). Every symbol in every $\lambda_i$ is less than every symbol in
  $\tau$ and so are inserted into the left-minimal subtree of $\tau$. Since every symbol in $\lambda_i$ is greater than
  every symbol in $\lambda_{i+1}$, the former is inserted into the right-maximal subtree of the latter.

  As noted above, $D_{h+1}$ consists of $u_{h+1}$ with empty left subtree, $s_2$ nodes $q_{h+1} = q_h$, and the subtree
  $B_h$ with its tertiary nodes $q_{h+1} = q_h$ deleted. Hence, since $D_h$ contains secondary nodes $q_h$, so does $D_{h+1}$, and
  so $E_{h+1}$ consists of $D_{h+1}$ with $s_1$ nodes $q_{h+1} = q_h$ inserted, as shown in \fullref{Figure}{fig:sylvinductioncase3a}.

  Therefore $E_{h+1}$ appears at the root of $U$ and so $T_{h+1}$ satisfies P1. The other trees $E_{i_j}$ in $T_h$ were in
  the uppermost non-empty $\lambda_i$ or $\tau$; this still holds and so $T_{h+1}$ satisfies P2. Since $T_h$ satisfies
  P3 and the insertions into the left-minimal and right-maximal subtrees of $E_{h+1}$, the tree $T_{h+1}$ satisfies
  P3. Finally, $T_{h+1}$ satisfies P4 since $T_h$ does. (Note that $m_h$ is below $p_{h} = u_{h+1}$, but this does not
  matter since $u_h$ is not in $U_{h+1}^\uparrow$.)

  \smallskip
  \noindent\textit{Sub-case 3(b).} Suppose that $E_h = D_h$. Then either $q_h$ is undefined, or is defined but $D_h$ does not
  contain nodes $q_h$.

  Suppose that $q_h$ is defined. Let $s$ be the number of nodes $q_h$ that appear in $U$ outside of the subtree
  $D_h$. 

  Since $q_h$ is greater than every node in $D_h$, it follows that the $s$ nodes $q_h$ are all
  in the right-maximal subtree of $D_h$ (and $E_h$) in $T_h$. Since $q_h$ is the least symbol greater than every symbol in $D_h$, in $T_h$
  the left child of a node $q_h$ can only be another node $q_h$.

  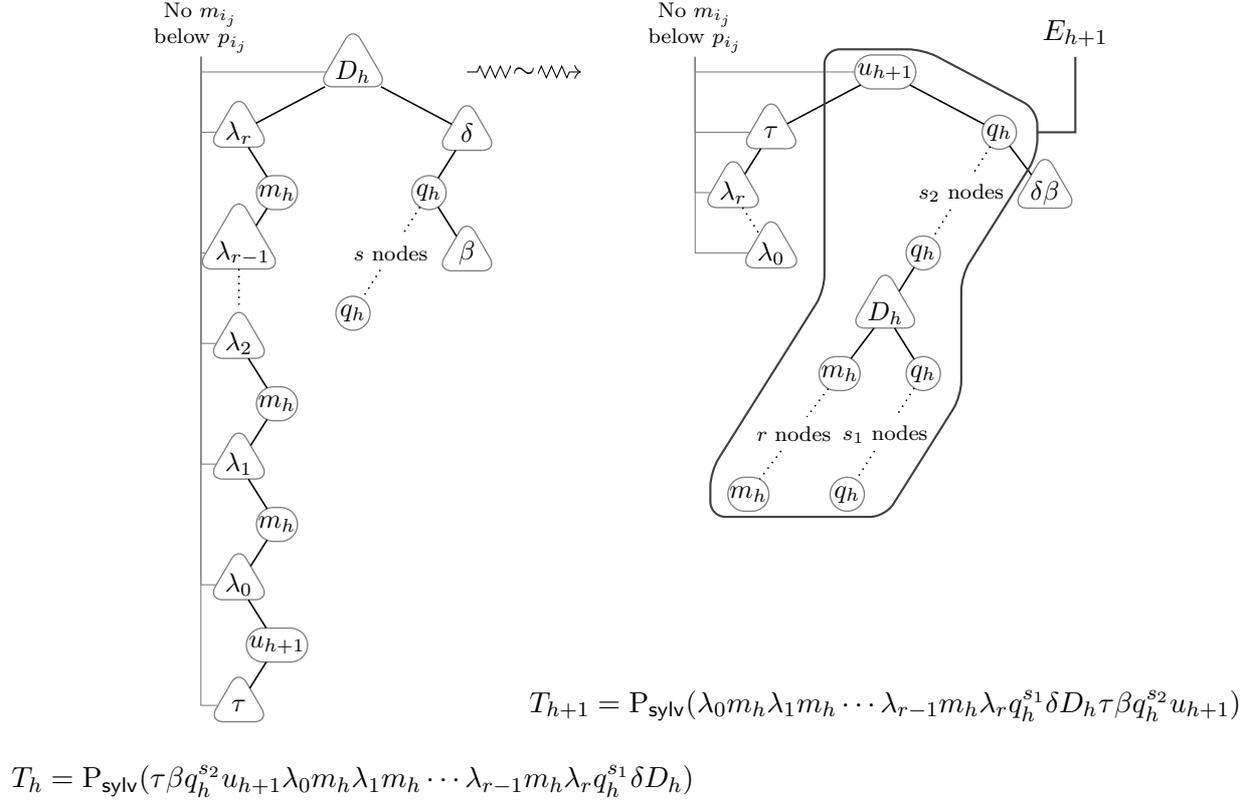
\begin{figure}[tb]
    \centerline{%
    \begin{tikzpicture}
      \begin{scope}[smallbst]
        \node[triangle] (aroot) at (0,0) {$D_h$}
        child[sibling distance=30mm] { node[triangle] (a0) {$\lambda_r$}
          child[missing]
          child { node (a01) {$m_h$}
            child[solid] { node[triangle] (a010) {$\lambda_{r{-}1}$\kern-1mm}
              child[dotted,level distance=12mm] { node[solid,triangle] (a0100) {$\lambda_2$}
                child[missing]
                child[solid,level distance=8mm] { node (a01001) {$m_h$}
                  child { node[solid,triangle] (a010010) {$\lambda_1$}
                    child[missing]
                    child[solid] { node (a0100101) {$m_h$}
                      child[solid] { node[triangle] (a01001010) {$\lambda_0$}
                        child[missing]
                        child { node (a010010101) {$u_{h+1}$}
                          child { node[triangle] (a0100101010) {$\tau$} }
                          child[missing]
                        }
                      }
                      child[missing]
                    }
                  }
                  child[missing]
                }
              }
            }
            child[missing]
          }
        }
        child[sibling distance=30mm] { node[triangle] (a1) {$\delta$}
          child { node (a10) {$q_h$}
            child[dotted] { node[nodecount] (a100) {$s$ nodes}
              child { node[solid] (a1000) {$q_h$} }
              child[missing]
            }
            child { node[triangle] (a101) {$\beta$} }
          }
          child[missing]
        };
      \end{scope}
      \node[tcomment] (aroota0100101010comment) at ($ (a0100101010) + (-5mm,86mm) $) {No $m_{i_j}$\\below $p_{i_j}$};
      \draw[gray] (aroota0100101010comment.south) |- (aroot);
      \draw[gray] (aroota0100101010comment.south) |- (a0);
      \draw[gray] (aroota0100101010comment.south) |- (a010);
      \draw[gray] (aroota0100101010comment.south) |- (a0100);
      \draw[gray] (aroota0100101010comment.south) |- (a010010);
      \draw[gray] (aroota0100101010comment.south) |- (a01001010);
      \draw[gray] (aroota0100101010comment.south) |- (a0100101010);
      \node at ($ (aroot) + (0,-94mm) $) {$T_h = \psylv{\tau\beta q_h^{s_2}u_{h+1}\lambda_0m_h\lambda_1m_h\cdots\lambda_{r-1}m_h\lambda_rq_h^{s_1}\delta D_h}$};
      %
      %
      \begin{scope}[smallbst]
        \node (broot) at (70mm,0) {$u_{h+1}$}
        child[sibling distance=30mm] { node[triangle] (b0) {$\tau$}
          child { node[triangle] (b00) {$\lambda_r$}
            child[missing]
            child[dotted] { node[solid,triangle] (b001) {$\lambda_0$} }
          }
          child[missing]
        }
        child[sibling distance=30mm] { node[solid] (b1) {$q_h$}
          child[dotted] { node[nodecount] (b10) {$s_2$ nodes}
            child { node[solid] (b100) {$q_{h}$}
              child[solid] { node[triangle] (b1000) {$D_h$}
                child[sibling distance=12mm] { node (b10000) {$m_h$}
                  child[dotted] { node[nodecount] (b100000) {$r$ nodes}
                    child { node[solid] (b1000000) {$m_h$} }
                    child[missing]
                  }
                  child[missing]
                }
                child { node (b10001) {$q_h$}
                  child[dotted] { node[nodecount] (b100010) {$s_1$ nodes}
                    child { node[solid] (b1000100) {$q_h$} }
                    child[missing]
                  }
                  child[missing]
                }
              }
              child[missing]
            }
            child[missing]
          }
          child[sibling distance=12mm] { node[triangle] (b2) {$\delta\beta$} }
        };
      \end{scope}
      \node[tcomment] (b0b000comment) at ($ (b00) + (-5mm,18mm) $) {No $m_{i_j}$\\below $p_{i_j}$};
      \draw[gray] (b0b000comment.south) |- (broot);
      \draw[gray] (b0b000comment.south) |- (b0);
      \draw[gray] (b0b000comment.south) |- (b00);
      \draw[gray] (b0b000comment.south) |- (b001);
      \draw[bstoutline,rounded corners=2mm]
      ($ (broot) + (0,3mm) $) --
      ($ (broot) + (5mm,3mm) $) --
      ++(15mm,-8mm) --
      ($ (b1) + (5mm,-3mm) $) --
      ++(-10mm,-16mm) --
      ($ (b10001) + (5mm,-3mm) $) --
      ($ (b1000100) + (5mm,-3mm) $) -|
      ($ (b1000000) + (-5mm,3mm) $) --
      ++(15mm,24mm) |-
      ($ (broot) + (0,3mm) $);
      \node[anchor=south] (e1) at ($ (b1) + (10mm,10mm) $) {$E_{h+1}$};
      \draw[bstoutline] ($ (b1) + (5mm,0) $) -| (e1);
      \node at ($ (broot) + (0,-84mm) $)
      {$T_{h+1} = \psylv{\lambda_0m_h\lambda_1m_h\cdots\lambda_{r-1}m_h\lambda_rq_h^{s_1}\delta D_h\tau\beta q_h^{s_2}u_{h+1}}$};
      \draw ($ (aroot) + (15mm,0) $) edge[mogrifyarrow] node[fill=white,anchor=mid,inner xsep=.25mm,inner ysep=1mm] {$\sim$} ($ (broot) + (-40mm,0) $);
    \end{tikzpicture}}
    \caption{Induction step, sub-case 3(b): $E_h = D_h$.}
    \label{fig:sylvinductioncase3b}
  \end{figure}

  Now return to the general situation, where $q_h$ may or may not be defined.  As shown in
  \fullref{Figure}{fig:sylvinductioncase3b}, let $\lambda_i$ and $\tau$ be as in sub-case~3(b). Let $\delta$ be a
  reading of the right-maximal subtree of $D_h$ outside the complete subtree at the uppermost node $q_h$, and let
  $\beta$ be a reading of the right subtree of the topmost node $q_h$. (If $q_h$ is undefined, $\beta$ is empty.)

  Now,
  \[
    T_h = \psylv{\tau\beta q_h^{s_2}u_{h+1}\lambda_0m_h\lambda_1m_h\cdots\lambda_{r-1}m_h\lambda_rq_h^{s_1}\delta D_h},
  \]
  where $s_1 = s - s_2$. (Recall that $s_2$ was defined above as the number of secondary nodes $q_h$ between $u_h$ and
  $p_h = u_{h+1}$ in $U$.) Let
  \[
    T_{h+1} = \psylv{\lambda_0m_h\lambda_1m_h\cdots\lambda_{r-1}m_h\lambda_rq_h^{s_1}\delta D_h\tau\beta q_h^{s_2}u_{h+1}}.
  \]
  Notice that $T_h \cyc T_{h+1}$.

  Assume for the moment that $q_h$ is defined and $s_2 > 0$.  In computing $T_{h+1}$, the symbol $u_{h+1}$ is inserted
  first and becomes the root node. Then the $s_2$ symbols $q_h$ are inserted into the right subtree of the root nodes,
  since $q_h > u_{h+1}$. Every symbol in $\beta$ is greater than $q_h$, so $\beta$ is inserted as the right subtree of
  the topmost $q_h$. Every symbol in $\tau$ is less than or equal to $u_{h+1}$ and so is inserted into the left subtree
  of the root node $u_{h+1}$. Every symbol in $D_h$ is greater than $u_{h+1}$ and less than $q_h$, so the tree $D_h$ is
  re-inserted as the left child of the bottommost $q_h$. The remaining $s_1$ symbols $q_h$ are re-inserted below $D_h$,
  attached at the right-maximal subtree of $D_h$. The symbol $m_h$ is the smallest symbol greater than $p_h = u_{h+1}$
  and so all symbols $m_h$ are inserted into the the left-minimal subtree of $D_h$ (which is attached at the single node
  $m_h$ in $T_h$). Every symbol in every $\lambda_i$ is less than every symbol in $\tau$ and so the $\lambda_i$ are
  inserted into the left-minimal subtree of $\tau$. Since every symbol in $\lambda_i$ is greater than every symbol in
  $\lambda_{i+1}$, the former is inserted into the right-maximal subtree of the latter.

  As noted above, $D_{h+1}$ consists of $u_{h+1}$ with empty left subtree, $s_2$ nodes $q_{h+1} = q_h$, and the subtree
  $B_h$ with its tertiary nodes $q_{h+1} = q_h$ deleted. The tree $D_{h+1}$ contains $s_2$ nodes $q_{h+1} = q_h$, and so
  $E_{h+1}$ consists of $D_{h+1}$ with $s_1$ nodes $q_{h+1} = q_h$ inserted, as shown in
  \fullref{Figure}{fig:sylvinductioncase3b}. Thus $E_{h+1}$ appears at the root of $T_{h+1}$ and every node below it is
  in its left-minimal or right-maximal subtrees.

  If, on the other hand, $q_h$ is defined and $s_2 = 0$, then, by near-identical reasoning, $D_h$ is inserted at the
  right child of $u_{h+1}$ and $\delta\beta$ becomes the right-maximal subtree of $D_h$, and the $s = s_1$ nodes $q_h$
  are inserted into the tree $\delta\beta$.

  In this case, $D_{h+1}$ does not contain any nodes $q_{h+1} = q_h$ and so $E_{h+1} = D_{h+1}$ consists of $u_{h+1}$
  and the subtree $B_h$ with its tertiary nodes $q_{h+1} = q_h$ deleted and so $E_{h+1}$ appears at the root of
  $T_{h+1}$ and every node below it is in its left-minimal or right-maximal subtrees.

  Finally, if $q_h$ is undefined, then again $D_h$ is inserted at the right child of $u_{h+1}$ and $\delta\beta$ becomes
  the right-maximal subtree of $D_h$.

  In this case, $q_{h+1}$ is also undefined and so there are no tertiary nodes in $B_h$. Hence $E_{h+1} = D_{h+1}$ consists precisely of
  $u_{h+1}$ with right subtree $B_h$. So $E_{h+1}$ appears at the root of
  $T_{h+1}$ and every node below it is in its left-minimal or right-maximal subtrees.

  Therefore $T_{h+1}$ satisfies P1. The other trees $E_{i_j}$ in $T_h$ were in the uppermost non-empty $\lambda_i$ or
  $\tau$; this still holds and so $T_{h+1}$ satisfies P2. Since $T_h$ satisfies P3 and all nodes below $E_{h+1}$ are in
  its left-minimal and right-maximal subtrees, the tree $T_{h+1}$ satisfies P3. Finally, $T_{h+1}$ satisfies P4 since
  $T_h$ does. (Note that $m_h$ is below $p_{h+1}$, but this does not matter since $u_h$ is not in $U_{h+1}^\uparrow$.)

  \medskip
  \noindent\textit{Case 4.} Suppose that, in $U$, the node $u_h$ is in the right subtree of $u_{h+1}$, and there is some node $u_i$ in the
  left subtree of $u_{h+1}$.

  Suppose there are topmost nodes $u_{j}$ and $u_{j'}$ in this left subtree. Then their lowest common ancestor $v$ is
  also in this left subtree; $v$ must have both subtrees non-empty and so be a topmost node and thus lie in $U_h$. Thus
  there is a unique node in $U_h^\uparrow$ in the left subtree of $v$; clearly this is $u_g$, which is the rightmost
  node in $U_h^\uparrow \setminus \set{u_h}$. (Recall that, for brevity, $h = i_k$ and $g = i_{k-1}$.) Furthermore, this
  node is on the path of left child nodes from $u_{h+1}$ (since only topmost nodes have right subtrees). Therefore,
  $U_{h+1}^\uparrow = \parens[\big]{U_h^\uparrow \setminus \set{u_h,u_i}} \cup \set{u_{h+1}}$, where $u_i$ is the unique
  node from $U_h^\uparrow$ in the left subtree of $u_{h+1}$. As in case~3, $p_h = u_{h+1}$. The smallest symbol in the
  tree $B_{h+1}$ is the smallest symbol in the left subtree of $u_{h+1}$ (which is certainly non-empty in this case) and
  so $m_{h+1} = m_g$. The least symbol less than or equal to every symbol in $B_{h+1}$ is the least symbol less than or
  equal to every symbol in the right subtree of $u_{h+1}$ (which is certainly non-empty in this case) and so
  $q_h = q_{h+1}$ (or $q_{h+1}$ is undefined if $q_h$ is undefined). The only nodes between $u_g$ and the topmost node
  $u_{h+1}$ are other nodes $u_{h+1}$, so $q_g = u_{h+1}$ (in particular, $q_g$ is defined). Finally, $u_{h+1}$ is the
  first node entered from a right child on ascending from $u_h$ to the root (by the definition of the topmost traversal)
  so and $u_{h+1} = p_h$.

  Let $s_2$ be the number of secondary nodes $q_h$ between $u_h$ and $u_{h+1}$ in $U$, and let $t_2$ is the number of
  primary nodes $u_{h+1}$ and let $s_1 = s - s_2$.

  \smallskip
  \noindent\textit{Sub-case 4(a).} Suppose that $E_h \neq D_h$ and $E_g \neq D_g$. Then $q_h$ is defined,
  $D_h$ contains nodes $q_h$, and $D_g$ contains nodes $q_g$. Since $q_h$ is greater than or equal to every node in
  $D_h$, it follows that the topmost node $q_h$ is where the right-maximal subtree of $D_h$ (and $E_h$) is attached in
  $T_h$. Similarly, since $q_g$ is greater than or equal to every node in $D_g$, it follows that the topmost node $q_g$
  is where the right-maximal subtree of $D_g$ (and $E_g$) is attached in $T_h$. The tree $E_h$ consists of $D_h$ with
  $s$ nodes $q_h$ inserted, where $s$ is the number of nodes $q_h$ that appear in $U$ outside of the subtree $D_h$. The
  tree $E_g$ consists of $D_g$ with $t$ nodes $q_g$ inserted, where $t$ is the number of nodes $q_g = u_{h+1}$ that
  appear in $U$ outside of the subtree $D_g$.

  By P2, the subtree $E_g$ appears on the path of left child nodes and thus in the left-minimal subtree of $D_h$ (and
  $E_h$). The only symbols that are less than or equal to $m_h$ (the minimum symbol in $E_h$) and greater than or equal
  to $q_g = p_h = u_{h+1}$ (the maximum symbol in $E_g$) are the symbols $m_h$ and $q_g = p_h = u_{h+1}$
  themselves.

  By P4, no node $m_h$ appears below a node
  $p_h = q_g$, so nodes $m_h$ cannot appear in the right-maximal subtree of $E_g$. Thus the nodes $m_h$ (except for the single
  node $m_h$ in $E_h$) are precisely the nodes on the path of left child nodes between $E_h$ and $E_g$.

  As shown in \fullref{Figure}{fig:sylvinductioncase4a}, let $\lambda$ be a reading of the left-minimal subtree of $D_g$
  and let $\beta$ be a reading of the right-maximal subtree of $D_h$.

  \begin{figure}[tb]
    \centerline{%
    \begin{tikzpicture}
      \begin{scope}[smallbst]
        \node[triangle] (aroot) at (0,0) {$D_h$}
        child[sibling distance=12mm] { node (a0) {$m_h$}
          child[dotted] { node[nodecount] (a00) {$r$ nodes}
            child[dotted] { node[solid] (a000) {$m_h$}
              child[solid] { node[triangle] (a0000) {$D_g$}
                child { node[triangle] (a00000) {$\lambda$} }
                child { node[empty] (a00001) {$u_{h+1}$}
                  child[dotted] { node[nodecount] (a000010) {$t$ nodes}
                    child[dotted] { node[solid] {$u_{h+1}$} }
                    child[missing]
                  }
                  child[missing]
                }
                child[missing]
              }
              child[missing]
            }
            child[missing]
          }
          child[missing]
        }
        child[sibling distance=12mm] { node (a1) {$q_h$}
          child[dotted] { node[nodecount] (a10) {$s$ nodes}
            child { node[solid] (a100) {$q_h$} }
            child[missing]
          }
          child[missing]
        }
        child[sibling distance=12mm] { node[triangle] (a2) {$\beta$} };
      \end{scope}
      \node[tcomment] (aroota00010comment) at ($ (a0000) + (-5mm,34mm) $) {No $m_{i_j}$\\below $p_{i_j}$};
      \draw[gray] (aroota00010comment.south) |- (aroot);
      \draw[gray] (aroota00010comment.south) |- (a0);
      \draw[gray] (aroota00010comment.south) |- (a00);
      \draw[gray] (aroota00010comment.south) |- (a000);
      \draw[gray] (aroota00010comment.south) |- (a0000);
      %
      \node at ($ (aroot) + (0,-74mm) $) {$T_h = \psylv{q_h^{s_2}u_{h+1}^t\lambda D_g \beta m_h^rq_h^{s_1}D_h}$};
      %
      %
      \begin{scope}[smallbst]
        \node (broot) at (70mm,0) {$u_{h+1}$}
        child[sibling distance=25mm,dotted] { node[nodecount] (b0) {$t_2$ nodes}
          child { node[solid] (b00) {$u_{h+1}$}
            child[solid] { node[triangle] (b000) {$D_g$}
              child { node[triangle] (b0000) {$\lambda$} }
              child { node (b0001) {$u_{h+1}$}
                child[dotted] { node[nodecount] (b00010) {$t_1$ nodes}
                  child { node[solid] (b000100) {$u_{h+1}$} }
                  child[missing]
                }
                child[missing]
              }
              child[missing]
            }
            child[missing]
          }
          child[missing]
        }
        child[sibling distance=25mm] { node (b1) {$q_h$}
          child[dotted] { node[nodecount] (b10) {$s_2$ nodes}
            child[dotted] { node[solid] (b100) {$q_h$}
              child[solid] { node[triangle] (b1000) {$D_h$}
                child[sibling distance=12mm] { node (b10000) {$m_h$}
                  child[dotted] { node[nodecount] (b100000) {$r$ \strut nodes}
                    child[dotted] { node[solid] (b1000000) {$m_h$} }
                    child[missing]
                  }
                  child[missing]
                }
                child[sibling distance=12mm] { node (b10001) {$q_h$}
                  child[dotted] { node[nodecount] (b100010) {$s_1$ nodes}
                    child[dotted] { node[solid] (b1000100) {$q_h$} }
                    child[missing]
                  }
                  child[missing]
                }
                child[missing]
              }
              child[missing]
            }
            child[missing]
          }
          child { node[triangle] (b11) {$\beta$} }
        };
      \end{scope}
      \node[lcomment] (brootb000100comment) at ($ (b000100) + (-9mm,24mm) $) {No $m_{i_j}$\\below $p_{i_j}$};
      \draw[gray] (brootb000100comment.east) -- ++ (2mm,0) |- (broot);
      \draw[gray] (brootb000100comment.east) -- ++ (2mm,0) |- (b00);
      \draw[gray] (brootb000100comment.east) -- ++ (2mm,0) |- (b000);
      \draw[gray] (brootb000100comment.east) -- ++ (2mm,0) |- (b0000);
      \draw[gray] (brootb000100comment.east) -- ++ (2mm,0) |- (b000100);
      \draw[bstoutline,rounded corners=2mm]
      ($ (broot) + (0,3mm) $) --
      ($ (broot) + (5mm,3mm) $) --
      ($ (b1) + (5mm,0) $) --
      +(-4mm,-6.6mm) |-
      ($ (b1000000) + (5mm,-3mm) $) --
      ($ (b1000000) + (-5mm,-3mm) $) -|
      ($ (b000100) + (-5mm,-3mm) $) --
      ($ (b000100) + (-5mm,3mm) $) --
      ($ (b00010) + (-5mm,3mm) $) -|
      ($ (b000) + (-5mm,3mm) $) --
      ($ (b0) + (-5mm,3mm) $) --
      ($ (broot) + (-5mm,3mm) $) --
      ($ (broot) + (0,3mm) $);
      \node (e1) at ($ (b1) + (10mm,00mm) $) {$E_{h+1}$};
      \draw[bstoutline] ($ (b1) + (5mm,0) $) -- (e1);
      \node at ($ (broot) + (0,-74mm) $) {$T_{h+1} = \psylv{u_{h+1}^{t_1}\lambda D_g \beta m_h^rq_h^{s_1}D_hq_h^{s_2}u_{h+1}^{t_2}}$};
    \end{tikzpicture}}
    \caption{Induction step, sub-case 4(a): $E_h \neq D_h$ and $E_g \neq D_g$.}
    \label{fig:sylvinductioncase4a}
  \end{figure}

  Then
  \[
    T_h = \psylv{q_h^{s_2}u_{h+1}^t\lambda D_g \beta m_h^rq_h^{s_1}D_h}.
  \]
  Let
  \[
    T_{h+1} = \psylv{u_{h+1}^{t_1}\lambda D_g \beta m_h^rq_h^{s_1}D_hq_h^{s_2}u_{h+1}^{t_2}}.
  \]
  Note that $T_h \cyc T_{h+1}$.

  In computing $T_{h+1}$, the rightmost symbol $u_{h+1}$ is inserted first and becomes the root
  node, with the remaining $t_2 - 1$ symbols descending from it on the path of left child nodes. Since
  $q_h = q_{h+1} > u_{h+1}$, the $s_2$ symbols $q_h$ are inserted into the right subtree of the root node
  $u_{h+1}$. Every symbol in $D_h$ is greater than $u_{h+1}$ and less than or equal to $q_h$, so the subtree $D_h$ is
  re-inserted as the left child node of the bottommost node $q_h$. The remaining $s_1$ symbols $q_h$ are inserted below
  $D_h$ (attached at the same place as in $T_h$). The symbol $m_h$ is the smallest symbol greater than $u_{h+1}$ and so
  the $r$ symbols $m_h$ is re-inserted into the left-minimal subtree of $D_h$ (attached as the left child of the node
  $m_h$ in $D_h$). Every symbol in $\beta$ is greater than $q_h$, so $\beta$ is re-inserted as the right child of the
  topmost $q_h$. Every symbol in $D_g$ and $\lambda$ is less than or equal to $u_{h+1}$ and so $D_g$ and $\lambda$ are
  re-inserted at the left child of the bottommost node $u_{h+1}$. The remaining $t_1$ nodes $u_{h+1}$ are into the right-maximal subtree of $D_g$ (in
  the same place as they were attached in $T_h$).

  The subtree $D_h$, the nodes $m_h$, and possibly some of the $q_h$ below $D_h$ make up $B_h$. This tree $B_h$,
  together with the nodes $q_h$ above $D_h$, the nodes $u_{h+1}$ and $D_{g}$ together make up $D_{h+1}$. Adding the
  remaining nodes $q_h = q_{h+1}$ below $D_h$ gives the tree $E_{h+1}$. So $T_{h+1}$ satisfies P1. The other trees
  $E_{i_j}$ in $T_h$ were in $\lambda$; this still holds and so $T_{h+1}$ satisfies P2. Every node not in $E_{h+1}$ is
  in its left-minimal or right-maximal subtree; together with the fact that $T_h$ satisfies P3, this shows that
  $T_{h+1}$ satisfies P3. Finally, $T_{h+1}$ satisfies P4 since $T_h$ does. (Note that $m_h$ is
  below $p_h = u_{h+1}$, but this does not matter since $u_h$ is not in $U_{h+1}^\uparrow$.)

  \smallskip
  \noindent\textit{Sub-case 4(b).} Suppose that $E_h \neq D_h$ and $E_g = D_g$. Then $q_h$ is defined and
  $D_h$ contains nodes $q_h$. Further, $q_g$ is defined and $q_g = u_{h+1}$, but $D_g$ does not contain nodes
  $q_g$. Since $q_h$ is greater than or equal to every node in $D_h$, it follows that the topmost node $q_h$ is where
  the right-maximal subtree of $D_h$ (and $E_h$) is attached. The tree $E_h$ consists of $D_h$ with $s$ nodes
  $q_h$ inserted, where $s$ is the number of nodes $q_h$ that appear in $U$ outside of the complete subtree at
  $u_h$.

  By P2, in $T_h$ the subtree $E_g$ appears on the path of left child nodes and thus in the left-minimal subtree of
  $E_h$ (and thus of $D_h$), which appears at the root of $T_h$ by P1. The only symbols that are less than or equal to
  $m_h$ (the minimum symbol in $E_h$) and greater than or equal to every symbol in $E_g$ are $m_h$ and
  $p_h = u_{h+1} = q_g$. Thus nodes $m_h$ and $u_{h+1}$ must appear either on the path of left child nodes between $E_h$
  and $E_g$, or in the right-maximal subtree of $E_g$.

  As shown in \fullref{Figure}{fig:sylvinductioncase4b}, let $\lambda$ be a reading of the left-minimal subtree of $D_g$
  and let $\beta$ be a reading of the right-maximal subtree of $D_h$. The nodes shown as empty consist of $r$ nodes $m_h$
  and $t$ nodes $u_{h+1}$ (for some $r$ and $t$), with the $m_h$ above the $u_{h+1}$ by condition P4 since $u_{h+1} = p_h$. Note, however, that
  the boundary between the $m_h$ and the $u_{h+1}$ may be either above or below $D_g$.

  \begin{figure}[tb]
    \centerline{%
    \begin{tikzpicture}
      \begin{scope}[smallbst]
        \node[triangle] (aroot) at (0,0) {$D_h$}
        child[sibling distance=12mm] { node[empty] (a0) {}
          child[dotted] { node[nodecount] (a00) {}
            child[dotted] { node[solid,empty] (a000) {}
              child[solid] { node[triangle] (a0000) {$D_g$}
                child { node[triangle] (a00000) {$\lambda$} }
                child { node[empty] (a00001) {}
                  child[dotted] { node[nodecount] (a000010) {}
                    child[dotted] { node[solid,empty] {} }
                    child[missing]
                  }
                  child[missing]
                }
              }
              child[missing]
            }
            child[missing]
          }
          child[missing]
        }
        child[sibling distance=12mm] { node (a1) {$q_h$}
          child[dotted] { node[nodecount] (a10) {$s$ nodes}
            child { node[solid] (a100) {$q_h$} }
            child[missing]
          }
          child[missing]
        }
        child[sibling distance=12mm] { node[triangle] (a2) {$\beta$} };
      \end{scope}
      \node[tcomment] (aroota00010comment) at ($ (a0000) + (-5mm,34mm) $) {No $m_{i_j}$\\below $p_{i_j}$};
      \draw[gray] (aroota00010comment.south) |- (aroot);
      \draw[gray] (aroota00010comment.south) |- (a0);
      \draw[gray] (aroota00010comment.south) |- (a00);
      \draw[gray] (aroota00010comment.south) |- (a000);
      \draw[gray] (aroota00010comment.south) |- (a0000);
      %
      \node[align=left] at ($ (aroot) + (-20mm,-74mm) $) {%
      (1)~$T_h = \psylv{q_h^{s_2}u_{h+1}^tm_h^{r_1}\lambda D_g m_h^{r_2}\beta q_h^{s_1}D_h}$;\\
      (2)~$T_h = \psylv{q_h^{s_2}u_{h+1}^{q_1}\lambda D_g u_{h+1}^{q_2}m_h^{r}\beta q_h^{s_1}D_h}$;\\
      (3)~$T_h = \psylv{q_h^{s_2}u_{h+1}^{q_1}\lambda D_g u_{h+1}^{q_2}m_h^{r}\beta q_h^{s_1}D_h}$.
      };
      %
      %
      \begin{scope}[smallbst]
        \node (broot) at (70mm,0) {$u_{h+1}$}
        child[sibling distance=30mm,dotted] { node[nodecount] (b0) {$t_2$ nodes}
          child { node[solid] (b00) {$u_{h+1}$}
            child[solid] { node[triangle] (b000) {$D_g$}
              child[sibling distance=12mm] { node[triangle] (b0000) {$\lambda$} }
              child[sibling distance=12mm] { node (b0001) {$u_{h+1}$}
                child[dotted] { node[nodecount] (b00010) {$t_1$ nodes}
                  child { node[solid] (b000100) {$u_{h+1}$} }
                  child[missing]
                }
                child[missing]
              }
            }
            child[missing]
          }
          child[missing]
        }
        child[sibling distance=30mm] { node (b1) {$q_h$}
          child[dotted] { node[nodecount] (b10) {$s_2$ nodes}
            child[dotted] { node[solid] (b100) {$q_h$}
              child[solid] { node[triangle] (b1000) {$D_h$}
                child[sibling distance=12mm] { node (b10000) {$m_h$}
                  child[dotted,sibling distance=10mm] { node[nodecount] (b100000) {$r$ \strut nodes}
                    child[dotted] { node[solid] (b1000000) {$m_h$} }
                    child[missing]
                  }
                  child[missing]
                }
                child[sibling distance=12mm] { node (b10001) {$q_h$}
                  child[dotted,sibling distance=10mm] { node[nodecount] (b100010) {$s_1$ nodes}
                    child[dotted] { node[solid] (b1000100) {$q_h$} }
                    child[missing]
                  }
                  child[missing]
                }
                child[missing]
              }
              child[missing]
            }
            child[missing]
          }
          child { node[triangle] (b11) {$\beta$} }
        };
      \end{scope}
      \node[lcomment] (brootb000100comment) at ($ (b000100) + (-9mm,24mm) $) {No $m_{i_j}$\\below $p_{i_j}$};
      \draw[gray] (brootb000100comment.east) -- ++ (2mm,0) |- (broot);
      \draw[gray] (brootb000100comment.east) -- ++ (2mm,0) |- (b00);
      \draw[gray] (brootb000100comment.east) -- ++ (2mm,0) |- (b000);
      \draw[gray] (brootb000100comment.east) -- ++ (2mm,0) |- (b0000);
      \draw[gray] (brootb000100comment.east) -- ++ (2mm,0) |- (b000100);
      \draw[bstoutline,rounded corners=2mm]
      ($ (broot) + (0,3mm) $) --
      ($ (broot) + (5mm,3mm) $) --
      ($ (b1) + (5mm,0) $) --
      +(-4mm,-6.6mm) |-
      ($ (b1000000) + (5mm,-3mm) $) --
      ($ (b1000000) + (-5mm,-3mm) $) -|
      ($ (b000100) + (-5mm,-3mm) $) --
      ($ (b000100) + (-5mm,3mm) $) --
      ($ (b0001) + (-5mm,3mm) $) --
      ($ (b000) + (-5mm,-3mm) $) --
      ($ (b000) + (-5mm,3mm) $) --
      ($ (b0) + (-5mm,3mm) $) --
      ($ (broot) + (-5mm,3mm) $) --
      ($ (broot) + (0,3mm) $);
      \node (e1) at ($ (b1) + (10mm,0mm) $) {$E_{h+1}$};
      \draw[bstoutline] ($ (b1) + (5mm,0) $) -- (e1);
      \node[align=left] at ($ (broot) + (0,-74mm) $) {%
      (1)~$T_{h+1} = \psylv{u_{h+1}^{t_1}m_h^{r_1}\lambda D_g m_h^{r_2}\beta q_h^{s_1}D_hq_h^{s_2}u_{h+1}^{t_2}}$; \\
      (2)~$T_{h+1} = \psylv{u_{h+1}^{q_2 - t_2}m_h^{r}\beta q_h^{s_1}D_hq_h^{s_2}u_{h+1}^{q_1}\lambda D_g u_{h+1}^{t_2}}$; \\
      (3)~$T_{h+1} = \psylv{u_{h+1}^{q_1-t_2+q_2}\lambda D_g u_{h+1}^{q_2}m_h^{r}\beta q_h^{s_1}D_hq_h^{s_2}u_{h+1}^{t_2-q_2}}$.
      };
    \end{tikzpicture}}
    \caption{Induction step, sub-case 4(b): $E_h \neq D_h$ and $E_g = D_g$; sub-sub-cases (1)--(3) use different cyclic shifts.}
    \label{fig:sylvinductioncase4b}
  \end{figure}

  Consider three sub-sub-cases:
  \begin{enumerate}

  \item There are no nodes $u_{h+1}$ above $D_g$. Suppose there are $r_1$ nodes $m_h$ below $D_g$ and $r_2$ above. Then
    \[
      T_h = \psylv{q_h^{s_2}u_{h+1}^tm_h^{r_1}\lambda D_g m_h^{r_2}\beta q_h^{s_1}D_h}.
    \]
    Let
    \[
      T_{h+1} = \psylv{u_{h+1}^{t_1}m_h^{r_1}\lambda D_g m_h^{r_2}\beta q_h^{s_1}D_hq_h^{s_2}u_{h+1}^{t_2}}.
    \]
    Note that $T_h \cyc T_{h+1}$.

    In computing $T_{h+1}$, the rightmost symbol $u_{h+1}$ is inserted first and becomes the root node, with the
    remaining $t_2 - 1$ symbols descending from it on the path of left child nodes. Since $q_h = q_{h+1} > u_{h+1}$, the
    $s_2$ symbols $q_h$ are inserted into the right subtree of the root node $u_{h+1}$. Every symbol in $D_h$ is greater
    than $u_{h+1}$ and less than or equal to $q_h$, so if $s_2 > 0$, the subtree $D_h$ is re-inserted as the left child
    node of the bottommost node $q_h$, while if $s_2 = 0$, the subtree $D_h$ is inserted as the right child of
    $u_{h+1}$. The remaining $s_1$ symbols $q_h$ are inserted below $D_h$ (attached at the same place as in
    $T_h$). Every symbol in $\beta$ is greater than $q_h$, so if $s_2 > 0$, the subtree $\beta$ is re-inserted as the
    right child of the topmost $q_h$, while if $s_2 = 0$, the subtree $\beta$ is inserted as the right-maximal
    subtree of $D_h$. The symbol $m_h$ is the smallest symbol greater than $u_{h+1}$ and so the $r_2$ symbols $m_h$ are
    inserted into the left-minimal subtree of $D_h$ (attached as the left child of the node $m_h$ in $D_h$). Every
    symbol in $D_g$ and $\lambda$ is less than or equal to $u_{h+1}$ and so $D_g$ and $\lambda$ are re-inserted at the
    left child of the bottommost node $u_{h+1}$. The remaining $r_1$ nodes $m_h$ are inserted at the left child of the
    bottommost node $m_h$ (so that the nodes $m_h$ are now consecutive). The remaining $t_1$ nodes $u_{h+1}$ are into
    the right-maximal subtree of $D_g$.

  \item There are $o_1$ nodes $u_{h+1}$ below $D_g$ and $o_2$ above, where $o_2 \geq t_2$. Then
    \[
      T_h = \psylv{q_h^{s_2}u_{h+1}^{o_1}\lambda D_g u_{h+1}^{o_2}m_h^{r}\beta q_h^{s_1}D_h}.
    \]
    Let
    \[
      T_{h+1} = \psylv{u_{h+1}^{o_2 - t_2}m_h^{r}\beta q_h^{s_1}D_hq_h^{s_2}u_{h+1}^{o_1}\lambda D_g u_{h+1}^{t_2}}.
    \]
    Note that $T_h \cyc T_{h+1}$.

    In computing $T_{h+1}$, the rightmost symbol $u_{h+1}$ is inserted first and becomes the root node, with the
    remaining $t_2 - 1$ symbols descending from it on the path of left child nodes. Every symbol in $D_g$ and $\lambda$
    is less than or equal to $u_{h+1}$ and so $D_g$ and $\lambda$ are re-inserted at the left child of the bottommost
    node $u_{h+1}$. The next $o_1$ nodes $u_{h+1}$ are inserted into the right-maximal subtree $D_g$. Since
    $q_h = q_{h+1} > u_{h+1}$, the $s_2$ symbols $q_h$ are inserted into the right subtree of the root node
    $u_{h+1}$. Every symbol in $D_h$ is greater than $u_{h+1}$ and less than or equal to $q_h$, so if $s_2 > 0$, the
    subtree $D_h$ is re-inserted as the left child node of the bottommost node $q_h$, while if $s_2 = 0$, the subtree
    $D_h$ is inserted as the right child of $u_{h+1}$. The remaining $s_1$ symbols $q_h$ are inserted below $D_h$
    (attached at the same place as in $T_h$). Every symbol in $\beta$ is greater than $q_h$, so if $s_2 > 0$, the
    subtree $\beta$ is re-inserted as the right child of the topmost $q_h$, while if $s_2 = 0$, the subtree $\beta$ is
    attached at as the right-maximal subtree of $D_h$. The symbol $m_h$ is the smallest symbol greater than $u_{h+1}$
    and so the $r$ symbols $m_h$ are inserted into the left-minimal subtree of $D_h$ (attached as the left child of the
    node $m_h$ in $D_h$). The remaining $o_2-t_2$ nodes $u_{h+1}$ are inserted at the left child of the bottommost node
    $u_{h+1}$ (so that the $o_1+o_2 - t_2 = t_1$ nodes $u_{h+1}$ below $D_g$ are now consecutive).

  \item There are $o_1$ nodes $u_{h+1}$ below $D_g$ and $o_2$ above, where $o_2 < t_2$. Then
    \[
      T_h = \psylv{q_h^{s_2}u_{h+1}^{o_1}\lambda D_g u_{h+1}^{o_2}m_h^{r}\beta q_h^{s_1}D_h}.
    \]
    Let
    \[
      T_{h+1} = \psylv{u_{h+1}^{o_1-t_2+o_2}\lambda D_g u_{h+1}^{o_2}m_h^{r}\beta q_h^{s_1}D_hq_h^{s_2}u_{h+1}^{t_2-o_2}}.
    \]
    Note that $o_1-t_2+o_2 = t_1$ and that $T_h \cyc T_{h+1}$.

    In computing $T_{h+1}$, the rightmost symbol $u_{h+1}$ is inserted first and becomes the root node, with the
    remaining $t_2-o_2 - 1$ symbols descending from it on the path of left child nodes. Since $q_h = q_{h+1} > u_{h+1}$,
    the $s_2$ symbols $q_h$ are inserted into the right subtree of the root node $u_{h+1}$. Every symbol in $D_h$ is
    greater than $u_{h+1}$ and less than or equal to $q_h$, so if $s_2 > 0$, the subtree $D_h$ is re-inserted as the
    left child node of the bottommost node $q_h$, while if $s_2 = 0$, the subtree $D_h$ is inserted as the right child
    of $u_{h+1}$. The remaining $s_1$ symbols $q_h$ are inserted below $D_h$ (attached at the same place as in
    $T_h$). Every symbol in $\beta$ is greater than $q_h$, so if $s_2 > 0$, the subtree $\beta$ is re-inserted as the
    right child of the topmost $q_h$, while if $s_2 = 0$, the subtree $\beta$ is attached at as the right-maximal
    subtree of $D_h$. The symbol $m_h$ is the smallest symbol greater than $u_{h+1}$ and so the $r$ symbols $m_h$ are
    inserted into the left-minimal subtree of $D_h$ (attached as the left child of the node $m_h$ in $D_h$). The next
    $o_2$ symbols $u_{h+1}$ are inserted into the left subtree of the bottommost node $u_{n+1}$, so that there are
    $t_2-o_2+o_2 = t_2$ consecutive nodes $u_{h+1}$.  Every symbol in $D_g$ and $\lambda$ is less than or equal to
    $u_{h+1}$ and so $D_g$ and $\lambda$ are re-inserted at the left child of the bottommost node $u_{h+1}$. The
    remaining $o_1-t_2+o_2 = t_1$ nodes $u_{h+1}$ are inserted into the right-maximal subtree of $D_g$.

  \end{enumerate}

  The subtree $D_h$, the nodes $m_h$, and possibly some of the $q_h$ below $D_h$ make up $B_h$. This tree $B_h$,
  together with the nodes $q_h$ above $D_h$, the nodes $u_{h+1}$ and $D_{g}$ together make up $D_{h+1}$. Adding the
  remaining nodes $q_h = q_{h+1}$ below $D_h$ gives the tree $E_{h+1}$. So $T_{h+1}$ satisfies P1. The other trees
  $E_{i_j}$ in $T_2$ were in $\lambda$; this still holds and so $T_{h+1}$ satisfies P2. Every node not in $E_{h+1}$ is
  in its left-minimal or right-maximal subtree; together with the fact that $T_h$ satisfies P3, this shows that
  $T_{h+1}$ satisfies P3. Finally, $T_{h+1}$ satisfies P4 since $T_h$ does. (Note that $m_h$ is
  below $p_h = u_{h+1}$, but this does not matter since $u_h$ is not in $U_{h+1}^\uparrow$.)

  \begin{figure}[tb]
    \centerline{%
    \begin{tikzpicture}
      \begin{scope}[smallbst]
        \node[triangle] (aroot) at (0,0) {$D_h$}
        child { node (a0) {$m_h$}
          child[dotted] { node[nodecount] (a00) {$r$ nodes}
            child[dotted] { node[solid] (a000) {$m_h$}
              child[solid] { node[triangle] (a0000) {$D_g$}
                child { node[triangle] (a00000) {$\lambda$} }
                child { node[empty] (a00001) {$u_{h+1}$}
                  child[dotted] { node[nodecount] (a000010) {$t$ nodes}
                    child[dotted] { node[solid] {$u_{h+1}$} }
                    child[missing]
                  }
                  child[missing]
                }
                child[missing]
              }
              child[missing]
            }
            child[missing]
          }
          child[missing]
        }
        child { node[triangle] (a1) {$\delta$}
          child { node (a10) {$q_h$}
            child[dotted] { node[nodecount] (a100) {$s$ nodes}
              child { node[solid] (a1000) {$q_h$} }
              child[missing]
            }
            child { node[triangle] (a101) {$\beta$} }
          }
          child[missing]
        };
      \end{scope}
      \node[tcomment] (aroota00010comment) at ($ (a0000) + (-5mm,34mm) $) {No $m_{i_j}$\\below $p_{i_j}$};
      \draw[gray] (aroota00010comment.south) |- (aroot);
      \draw[gray] (aroota00010comment.south) |- (a0);
      \draw[gray] (aroota00010comment.south) |- (a000);
      \draw[gray] (aroota00010comment.south) |- (a0000);
      %
      \node[align=left] at ($ (aroot) + (-10mm,-74mm) $) {
      (1)~$T_h = \psylv{\beta q_h^{s_2}u_{h+1}^t\lambda D_g m_h^rq_h^{s_1}\delta D_h}$;\\
      (2)~$T_h = \psylv{u_{h+1}^t\lambda D_g m_h^r\beta q_h^{s}\delta D_h}$.
      };
      %
      %
      \begin{scope}[smallbst]
        \node (broot) at (70mm,0) {$u_{h+1}$}
        child[sibling distance=25mm,dotted] { node[nodecount] (b0) {$t_2$ nodes}
          child { node[solid] (b00) {$u_{h+1}$}
            child[solid] { node[triangle] (b000) {$D_g$}
              child { node[triangle] (b0000) {$\lambda$} }
              child { node (b0001) {$u_{h+1}$}
                child[dotted] { node[nodecount] (b00010) {$t_1$ nodes}
                  child { node[solid] (b000100) {$u_{h+1}$} }
                  child[missing]
                }
                child[missing]
              }
              child[missing]
            }
            child[missing]
          }
          child[missing]
        }
        child[sibling distance=25mm] { node (b1) {$q_h$}
          child[dotted] { node[nodecount] (b10) {$s_2$ nodes}
            child[dotted] { node[solid] (b100) {$q_h$}
              child[solid] { node[triangle] (b1000) {$D_h$}
                child[sibling distance=12mm] { node (b10000) {$m_h$}
                  child[dotted,sibling distance=10mm] { node[nodecount] (b100000) {$r$ \strut nodes}
                    child[dotted] { node[solid] (b1000000) {$m_h$} }
                    child[missing]
                  }
                  child[missing]
                }
                child[sibling distance=12mm] { node (b10001) {$q_h$}
                  child[dotted,sibling distance=10mm] { node[nodecount] (b100010) {$s_1$ nodes}
                    child[dotted] { node[solid] (b1000100) {$q_h$} }
                    child[missing]
                  }
                  child[missing]
                }
              }
              child[missing]
            }
            child[missing]
          }
          child { node[triangle] (b11) {$\delta\beta$} }
        };
      \end{scope}
      \node[lcomment] (brootb000100comment) at ($ (b000100) + (-9mm,24mm) $) {No $m_{i_j}$\\below $p_{i_j}$};
      \draw[gray] (brootb000100comment.east) -- ++ (2mm,0) |- (broot);
      \draw[gray] (brootb000100comment.east) -- ++ (2mm,0) |- (b00);
      \draw[gray] (brootb000100comment.east) -- ++ (2mm,0) |- (b000);
      \draw[gray] (brootb000100comment.east) -- ++ (2mm,0) |- (b0000);
      \draw[gray] (brootb000100comment.east) -- ++ (2mm,0) |- (b000100);
      \draw[bstoutline,rounded corners=2mm]
      ($ (broot) + (0,3mm) $) --
      ($ (broot) + (5mm,3mm) $) --
      ($ (b1) + (5mm,0) $) --
      +(-4mm,-6.6mm) |-
      ($ (b1000000) + (5mm,-3mm) $) --
      ($ (b1000000) + (-5mm,-3mm) $) -|
      ($ (b000100) + (-5mm,-3mm) $) --
      ($ (b000100) + (-5mm,3mm) $) --
      ($ (b00010) + (-5mm,3mm) $) -|
      ($ (b000) + (-5mm,3mm) $) --
      ($ (b0) + (-5mm,3mm) $) --
      ($ (broot) + (-5mm,3mm) $) --
      ($ (broot) + (0,3mm) $);
      \node (e1) at ($ (b1) + (10mm,00mm) $) {$E_{h+1}$};
      \draw[bstoutline] ($ (b1) + (5mm,0) $) -- (e1);
      \node[align=left] at ($ (broot) + (0,-74mm) $) {
      (1)~$T_{h+1} = \psylv{u_{h+1}^{t_1}\lambda D_g m_h^r\beta q_h^{s}\delta D_h u_{h+1}^{t_2}}$;\\
      (2)~$T_{h+1} = \psylv{u_{h+1}^{t_1}\lambda D_g m_h^r\beta q_h^{s}\delta D_h u_{h+1}^{t_2}}$.
      };
    \end{tikzpicture}}
    \caption{Induction step, sub-case 4(c): $E_h = D_h$ and $E_g \neq D_g$; sub-sub-cases (1) and (3) use different cyclic shifts.}
    \label{fig:sylvinductioncase4c}
  \end{figure}

  \smallskip
  \noindent\textit{Sub-case 4(c).} Suppose that $E_h = D_h$ and $E_g \neq D_g$. Then $D_g$ contains nodes $q_g$. Since
  $q_g$ is greater than or equal to every node in $D_g$, it follows that the
  topmost node $q_g$ is where the right-maximal subtree of $D_g$ (and $E_g$) is attached in $T_h$. The tree $E_g$
  consists of $D_g$ with $t$ nodes $q_g$ inserted, where $t$ is the number of nodes $q_g = u_{h+1}$ that appear in $U$
  outside of the complete subtree at $u_g$.

  The subtree $E_g$ appears on the path of left child nodes and thus in the left-minimal subtree of $E_h$ (and
  $E_h$). The only symbols that is less than or equal to $m_h$ and greater than or equal to every symbol
  in $q_g$ are $m_h$ and $p_h = u_{h+1} = q_g$. By P4, no node $m_h$ appears below a node $p_h$, so nodes $m_h$ cannot appear
  in the right-maximal subtree of $E_g$. Thus the nodes $m_h$ (except for the single node $m_h$ in $D_h$) are precisely
  the nodes on the path of left child nodes between $D_h$ and $D_g$.

  As shown in \fullref{Figure}{fig:sylvinductioncase4c}, let $\lambda$ be a reading of the left-minimal subtree of
  $D_g$. Let $\delta$ be a reading of the right-maximal subtree of $D_h$ outside the complete subtree at the topmost
  node $q_h$ (if $q_h$ is defined) and let $\beta$ be a reading of the right-maximal subtree of the topmost $q_h$. (Note
  that $\beta$ is empty if $q_h$ is not defined.)

  Let $s_2$ be the number of secondary nodes $q_h$ between $u_h$ and $u_{h+1}$ in $U$, and let $t_2$ be the number of
  primary nodes $u_{h+1}$. Note that if $q_h$ is defined, then $s_2 < s$ since there must be at least one primary node
  $q_h$. Consider two sub-sub-cases:

  \begin{enumerate}
  \item Suppose $s_2 > 0$ and $q_h$ is defined.
    Then
    \[
      T_h = \psylv{\beta q_h^{s_2}u_{h+1}^t\lambda D_g m_h^rq_h^{s_1}\delta D_h}.
    \]
    Let
    \[
      T_{h+1} = \psylv{u_{h+1}^{t_1}\lambda D_g m_h^rq_h^{s_1}\delta D_h\beta q_h^{s_2}u_{h+1}^{t_2}}.
    \]
    Note that $T_h \cyc T_{h+1}$.

    In computing $T_{h+1}$, the rightmost symbol $u_{h+1}$ is inserted first and becomes the root node, with the
    remaining $t_2 - 1$ symbols descending from it on the path of left child nodes. Since $q_h = q_{h+1} > u_{h+1}$, the
    $s_2$ symbols $q_h$ are inserted into the right subtree of the root node $u_{h+1}$. Every symbol in $\beta$ is
    greater than $q_h$, so $\beta$ is re-inserted as the right child of the topmost $q_h$. Every symbol in $D_h$ is
    greater than $u_{h+1}$ and less than or equal to $q_h$, so the subtree $D_h$ is re-inserted as the left child node
    of the bottommost node $q_h$. Every symbol in $\delta$ is greater than $q_h$, so $\delta$ is inserted into the
    subtree $\beta$. The remaining $s_1$ symbols $q_h$ are inserted into the right-maximal subtree of $D_h$. The symbol
    $m_h$ is the smallest symbol greater than $u_{h+1}$ and so the $r$ symbols $m_h$ are re-inserted into the
    left-minimal subtree of $D_h$ (attached as the left child of the node $m_h$ in $D_h$). Every symbol in $D_g$ and
    $\lambda$ is less than or equal to $u_{h+1}$ and so $D_g$ and $\lambda$ are re-inserted at the left child of the
    bottommost node $u_{h+1}$. The remaining $t_1$ nodes $u_{h+1}$ are into $D_g$ (in the same place as they were
    attached in $T_h$).

  \item Suppose $s_2 = 0$ or $q_h$ is not defined.
    Then
    \[
      T_h = \psylv{u_{h+1}^t\lambda D_g m_h^r\beta q_h^{s}\delta D_h}.
    \]
    (If $q_h$ is undefined, formally treat $q_h$ and $\beta$ as empty and $s$ as $0$.) Let
    \[
      T_{h+1} = \psylv{u_{h+1}^{t_1}\lambda D_g m_h^r\beta q_h^{s}\delta D_h u_{h+1}^{t_2}},
    \]
    Note that $T_h \cyc T_{h+1}$.

    In computing $T_{h+1}$, the rightmost symbol $u_{h+1}$ is inserted first and becomes the root node, with the
    remaining $t_2 - 1$ symbols descending from it on the path of left child nodes. Every symbol in $D_h$ and $\delta$
    is greater than $u_{h+1}$, so the subtrees $D_h$ and $\delta$ are re-inserted as the right child node of the root
    node $u_{h+1}$. The $s$ symbols $q_h$ are inserted into te right-maximal subtree of $D_h$. Every
    symbol in $\beta$ is greater than every symbol in $D_h$, so $\beta$ is inserted into the subtree $\delta$. The symbol
    $m_h$ is the smallest symbol greater than $u_{h+1}$ and so the $r$ symbols $m_h$ are re-inserted into the
    left-minimal subtree of $D_h$ (attached as the left child of the node $m_h$ in $D_h$). Every symbol in $D_g$ and
    $\lambda$ is less than or equal to $u_{h+1}$ and so $D_g$ and $\lambda$ are re-inserted at the left child of the
    bottommost node $u_{h+1}$. The remaining $t_1$ nodes
    $u_{h+1}$ are into $D_g$ (in the same place as they were attached in $T_h$).

  \end{enumerate}

  (The key difference between the sub-sub-case is the different placement of $\beta$ in the reading of $T_h$. In
  sub-sub-case (2), the subword $\beta q_h^{s}\delta D_h$ makes up the entire part of $T_h$ outside of the left-minimal
  subtree of $D_h$, and since this subword is intact in the reading of $T_{h+1}$, it will appear as the right subtree of
  the root, as shown by the reasoning below. This observation will be used to abbreviate some of the sub-sub-cases in
  sub-case~4(d) below.)

  The subtree $D_h$, the nodes $m_h$, and possibly some of the $q_h$ below $D_h$ make up $B_h$. This tree $B_h$,
  together with any nodes $q_h$ above $D_h$, the nodes $u_{h+1}$ and $D_{g}$ together make up $D_{h+1}$. Adding the
  remaining nodes $q_h = q_{h+1}$ below $D_h$ gives the tree $E_{h+1}$. So $T_{h+1}$ satisfies P1. The other trees
  $E_{i_j}$ in $T_2$ were in $\lambda$; this still holds and so $T_{h+1}$ satisfies P2. Every node not in $E_{h+1}$ is
  in its left-minimal or right-maximal subtree; together with the fact that $T_h$ satisfies P3, this shows that
  $T_{h+1}$ satisfies P3. Finally, $T_{h+1}$ satisfies P4 since $T_h$ does. (Note that $m_h$ is
  below $p_h = u_{h+1}$, but this does not matter since $u_h$ is not in $U_{h+1}^\uparrow$.)

  \begin{figure}[tb]
    \centerline{%
    \begin{tikzpicture}
      \begin{scope}[smallbst]
        \node[triangle] (aroot) at (0,0) {$D_h$}
        child { node[empty] (a0) {}
          child[dotted] { node[nodecount] (a00) {}
            child[dotted] { node[solid,empty] (a000) {}
              child[solid] { node[triangle] (a0000) {$D_g$}
                child { node[triangle] (a00000) {$\lambda$} }
                child { node[empty] (a00001) {}
                  child[dotted] { node[nodecount] (a000010) {}
                    child[dotted] { node[solid,empty] {} }
                    child[missing]
                  }
                  child[missing]
                }
              }
              child[missing]
            }
            child[missing]
          }
          child[missing]
        }
        child { node[triangle] (a1) {$\delta$}
          child { node (a10) {$q_h$}
            child[dotted] { node[nodecount] (a100) {$s$ nodes}
              child { node[solid] (a1000) {$q_h$} }
              child[missing]
            }
            child { node[triangle] (a101) {$\beta$} }
          }
          child[missing]
        };
      \end{scope}
      \node[tcomment] (aroota00010comment) at ($ (a0000) + (-5mm,34mm) $) {No $m_{i_j}$\\below $p_{i_j}$};
      \draw[gray] (aroota00010comment.south) |- (aroot);
      \draw[gray] (aroota00010comment.south) |- (a0);
      \draw[gray] (aroota00010comment.south) |- (a000);
      \draw[gray] (aroota00010comment.south) |- (a0000);
      %
      \node[align=left] at ($ (aroot) + (-20mm,-90mm) $) {
        (1)~$T_h = \psylv{\beta q_h^{s_2}u_{h+1}^tm_h^{r_1}\lambda D_g m_h^{r_2}q_h^{s_1}\delta D_h}$; \\
        (2)~$T_h = \psylv{\beta q_h^{s_2}u_{h+1}^{o_1}\lambda D_g u_{h+1}^{o_2}m_h^{r} q_h^{s_1}\delta D_h}$; \\
        (3)~$T_h = \psylv{\beta q_h^{s_2}u_{h+1}^{o_1}\lambda D_g u_{h+1}^{o_2}m_h^{r} q_h^{s_1}\delta D_h}$; \\
        (4)~$T_h = \psylv{u_{h+1}^tm_h^{r_1}\lambda D_g m_h^{r_2}\beta q_h^{s}\delta D_h}$; \\
        (5)~$T_h = \psylv{u_{h+1}^{o_1}\lambda D_g u_{h+1}^{o_2}m_h^{r}\beta q_h^{s}\delta D_h}$; \\
        (6)~$T_h = \psylv{u_{h+1}^{o_1}\lambda D_g u_{h+1}^{o_2}m_h^{r} \beta q_h^{s}\delta D_h}$. \\
      };
      %
      %
      \begin{scope}[smallbst]
        \node (broot) at (70mm,0) {$u_{h+1}$}
        child[sibling distance=30mm,dotted] { node[nodecount] (b0) {$t_2$ nodes}
          child { node[solid] (b00) {$u_{h+1}$}
            child[solid] { node[triangle] (b000) {$D_g$}
              child[sibling distance=12mm] { node[triangle] (b0000) {$\lambda$} }
              child[sibling distance=12mm] { node (b0001) {$u_{h+1}$}
                child[dotted] { node[nodecount] (b00010) {$t_1$ nodes}
                  child { node[solid] (b000100) {$u_{h+1}$} }
                  child[missing]
                }
                child[missing]
              }
            }
            child[missing]
          }
          child[missing]
        }
        child[sibling distance=30mm] { node (b1) {$q_h$}
          child[dotted] { node[nodecount] (b10) {$s_2$ nodes}
            child[dotted] { node[solid] (b100) {$q_h$}
              child[solid] { node[triangle] (b1000) {$D_h$}
                child[sibling distance=12mm] { node (b10000) {$m_h$}
                  child[dotted,sibling distance=10mm] { node[nodecount] (b100000) {$r$ \strut nodes}
                    child[dotted] { node[solid] (b1000000) {$m_h$} }
                    child[missing]
                  }
                  child[missing]
                }
                child[sibling distance=12mm] { node (b10001) {$q_h$}
                  child[dotted,sibling distance=10mm] { node[nodecount] (b100010) {$s_1$ nodes}
                    child[dotted] { node[solid] (b1000100) {$q_h$} }
                    child[missing]
                  }
                  child[missing]
                }
              }
              child[missing]
            }
            child[missing]
          }
          child { node[triangle] (b11) {$\delta\beta$} }
        };
      \end{scope}
      \node[lcomment] (brootb000100comment) at ($ (b000100) + (-9mm,24mm) $) {No $m_{i_j}$\\below $p_{i_j}$};
      \draw[gray] (brootb000100comment.east) -- ++ (2mm,0) |- (broot);
      \draw[gray] (brootb000100comment.east) -- ++ (2mm,0) |- (b00);
      \draw[gray] (brootb000100comment.east) -- ++ (2mm,0) |- (b000);
      \draw[gray] (brootb000100comment.east) -- ++ (2mm,0) |- (b0000);
      \draw[gray] (brootb000100comment.east) -- ++ (2mm,0) |- (b000100);
      \draw[bstoutline,rounded corners=2mm]
      ($ (broot) + (0,3mm) $) --
      ($ (broot) + (5mm,3mm) $) --
      ($ (b1) + (5mm,0) $) --
      +(-4mm,-6.6mm) |-
      ($ (b1000000) + (5mm,-3mm) $) --
      ($ (b1000000) + (-5mm,-3mm) $) -|
      ($ (b000100) + (-5mm,-3mm) $) --
      ($ (b000100) + (-5mm,3mm) $) --
      ($ (b0001) + (-5mm,3mm) $) --
      ($ (b000) + (-5mm,-3mm) $) --
      ($ (b000) + (-5mm,3mm) $) --
      ($ (b0) + (-5mm,3mm) $) --
      ($ (broot) + (-5mm,3mm) $) --
      ($ (broot) + (0,3mm) $);
      \node (e1) at ($ (b1) + (10mm,0mm) $) {$E_{h+1}$};
      \draw[bstoutline] ($ (b1) + (5mm,0) $) -- (e1);
      \node[align=left] at ($ (broot) + (0,-90mm) $) {
      (1)~$T_{h+1} = \psylv{u_{h+1}^{t_1}m_h^{r_1}\lambda D_g m_h^{r_2}q_h^{s_1}\delta D_h\beta q_h^{s_2}u_{h+1}^{t_2}}$; \\
      (2)~$T_{h+1} = \psylv{u_{h+1}^{o_2 - t_2}m_h^{r}q_h^{s_1}\delta D_h\beta q_h^{s_2}u_{h+1}^{o_1}\lambda D_g u_{h+1}^{t_2}}$; \\
      (3)~$T_{h+1} = \psylv{u_{h+1}^{o_1-t_2+o_2}\lambda D_g u_{h+1}^{o_2}m_h^{r}q_h^{s_1}\delta D_h\beta q_h^{s_2}u_{h+1}^{t_2-o_2}}$; \\
      (4)~$T_{h+1} = \psylv{u_{h+1}^{t_1}m_h^{r_1}\lambda D_g m_h^{r_2}\beta q_h^{s}\delta D_hu_{h+1}^{t_2}}$; \\
      (5)~$T_{h+1} = \psylv{u_{h+1}^{o_2 - t_2}m_h^{r}\beta q_h^{s}\delta D_hu_{h+1}^{o_1}\lambda D_g u_{h+1}^{t_2}}$; \\
      (6)~$T_{h+1} = \psylv{u_{h+1}^{o_1-t_2+o_2}\lambda D_g u_{h+1}^{o_2}m_h^{r}\beta q_h^{s}\delta D_hu_{h+1}^{t_2-o_2}}$. \\
      };
    \end{tikzpicture}}
    \caption{Induction step, sub-case 4(d): $E_h = D_h$ and $E_g = D_g$; sub-sub-cases (1)--(6) use different cyclic shifts.}
    \label{fig:sylvinductioncase4d}
  \end{figure}

  \smallskip
  \noindent\textit{Sub-case 4(d).} Suppose that $E_h = D_h$ and $E_g = D_g$. If $q_h$ is defined then it is the least
  symbol greater than any symbol in $D_h$ and so all nodes $q_h$ must appear in the right-maximal subtree of $E_h$ in
  $T_h$, and the left child of a node $q_h$ can only be another node $q_h$.

  It is possible that $q_g$ is defined, in which case $q_g = u_{h+1}$, but $D_g$ does not contain nodes $q_g$. The
  subtree tree $D_g$ appears on the path of left child nodes and thus in the left-minimal subtree of $D_h$ (and
  $E_h$). The only symbols that is less than or equal to every symbol in $D_h$ and greater than or equal to every symbol
  in $D_g$ are $m_h$ and $p_h = u_{h+1}$ (which is equal to $q_g$, if $q_g$ is defined). Thus nodes $m_h$ and $u_{h+1}$
  must appear either on the path of left child nodes between $D_h$ and $D_g$, or in the right-maximal subtree of $D_g$.

  As shown in \fullref{Figure}{fig:sylvinductioncase4d}, let $\lambda$ be a reading of the left-minimal subtree of
  $D_g$. Let $\delta$ be a reading of the right-maximal subtree of $D_h$ outside the complete subtree at the topmost
  node $q_h$ (if $q_h$ is defined) and let $\beta$ be a reading of the right-maximal subtree of the topmost $q_h$. (Note
  that $\beta$ is empty if $q_h$ is not defined.) The empty nodes are filled with $r$ nodes $m_h$ and $t$ nodes
  $u_{h+1}$, with the $m_h$ above the $u_{h+1}$. Note, however, that the boundary between the $m_h$ and the $u_{h+1}$
  may be either above or below $D_g$.

  Let $s_2$ be the number of secondary nodes $q_h$ between $u_h$ and $u_{h+1}$ in $U$, and let $t_2$ is the number of
  primary nodes $u_{h+1}$. There are six sub-sub-cases: consider first the three sub-sub-cases where $s_2 > 0$:
  \begin{enumerate}

  \item There are no $u_{h+1}$ above $D_g$. Suppose there are $r_1$ nodes $m_h$ below $D_g$ and $r_2$ above. Then
    \[
      T_h = \psylv{\beta q_h^{s_2}u_{h+1}^tm_h^{r_1}\lambda D_g m_h^{r_2}q_h^{s_1}\delta D_h}.
    \]
    Let
    \[
      T_{h+1} = \psylv{u_{h+1}^{t_1}m_h^{r_1}\lambda D_g m_h^{r_2}q_h^{s_1}\delta D_h\beta q_h^{s_2}u_{h+1}^{t_2}},
    \]
    Note that $T_h \cyc T_{h+1}$.

    In computing $T_{h+1}$, the rightmost symbol $u_{h+1}$ is inserted first and becomes the root node, with the
    remaining $t_2 - 1$ symbols descending from it on the path of left child nodes. Since $q_h = q_{h+1} > u_{h+1}$, the
    $s_2$ symbols $q_h$ are inserted into the right subtree of the root node $u_{h+1}$. Every symbol in $\beta$ is
    greater than $q_h$, so the subtree $\beta$ is re-inserted as the right child of the topmost $q_h$.  Every symbol in
    $D_h$ is greater than $u_{h+1}$ and less than or equal to $q_h$, so the subtree $D_h$ is re-inserted as the left
    child node of the bottommost node $q_h$. Every symbol in $\delta$ is greater than $q_h$, so $\delta$ is inserted
    into the subtree $\beta$. The remaining $s_1$ symbols $q_h$ are inserted into the right-maximal subtree of
    $D_h$. The symbol $m_h$ is the smallest symbol greater than $u_{h+1}$ and so the $r_2$ symbols $m_h$ are inserted
    into the left-minimal subtree of $D_h$ (attached as the left child of the node $m_h$ in $D_h$). Every symbol in
    $D_g$ and $\lambda$ is less than or equal to $u_{h+1}$ and so $D_g$ and $\lambda$ are re-inserted at the left child
    of the bottommost node $u_{h+1}$. The remaining $r_1$ nodes $m_h$ are inserted at the left child of the bottommost
    node $m_h$ (so that the nodes $m_h$ are now consecutive). The remaining $t_1$ nodes $u_{h+1}$ are into the
    right-maximal subtree of $D_g$.

  \item There are $o_1$ nodes $u_{h+1}$ below $D_g$ and $o_2$ above, where $o_2 \geq t_2$. Then
    \[
      T_h = \psylv{\beta q_h^{s_2}u_{h+1}^{o_1}\lambda D_g u_{h+1}^{o_2}m_h^{r} q_h^{s_1}\delta D_h}.
    \]
    Let
    \[
      T_{h+1} = \psylv{u_{h+1}^{o_2 - t_2}m_h^{r}q_h^{s_1}\delta D_h\beta q_h^{s_2}u_{h+1}^{o_1}\lambda D_g u_{h+1}^{t_2}}.
    \]
    Note that $T_h \cyc T_{h+1}$.

    In computing $T_{h+1}$, the rightmost symbol $u_{h+1}$ is inserted first and becomes the root node, with the
    remaining $t_2 - 1$ symbols descending from it on the path of left child nodes. Every symbol in $D_g$ and $\lambda$
    is less than or equal to $u_{h+1}$ and so $D_g$ and $\lambda$ are re-inserted at the left child of the bottommost
    node $u_{h+1}$. The next $o_1$ nodes $u_{h+1}$ are inserted into the right-maximal subtree of $D_g$. Since
    $q_h = q_{h+1} > u_{h+1}$, the $s_2$ symbols $q_h$ are inserted into the right subtree of the root node
    $u_{h+1}$. Every symbol in $\beta$ is greater than $q_h$, so subtree $\beta$ is re-inserted as the right child of
    the topmost $q_h$. Every symbol in $D_h$ is greater than $u_{h+1}$ and less than or equal to $q_h$, so the subtree
    $D_h$ is re-inserted as the left child node of the bottommost node $q_h$. Every symbol in $\delta$ is greater than
    $q_h$ and so $\delta$ is inserted into the subtree $\beta$. The remaining $s_1$ symbols $q_h$ are inserted into the
    right-maximal subtree of $D_h$. The symbol $m_h$ is the smallest symbol greater than $u_{h+1}$ and so the $r$
    symbols $m_h$ are inserted into the left-minimal subtree of $D_h$ (attached as the left child of the node $m_h$ in
    $D_h$). The remaining $o_2-t_2$ nodes $u_{h+1}$ are inserted at the left child of the bottommost node $u_{h+1}$ (so
    that the $o_1+o_2 - t_2 = t_1$ nodes $u_{h+1}$ below $D_g$ are now consecutive).

  \item There are $o_1$ nodes $u_{h+1}$ below $D_g$ and $o_2$ above, where $o_2 < t_2$. Then
    \[
      T_h = \psylv{\beta q_h^{s_2}u_{h+1}^{o_1}\lambda D_g u_{h+1}^{o_2}m_h^{r} q_h^{s_1}\delta D_h}.
    \]
    Let
    \[
      T_{h+1} = \psylv{u_{h+1}^{o_1-t_2+o_2}\lambda D_g u_{h+1}^{o_2}m_h^{r}q_h^{s_1}\delta D_h\beta q_h^{s_2}u_{h+1}^{t_2-o_2}}.
    \]
    Note that $o_1-t_2+o_2 = t_1$ and that $T_h \cyc T_{h+1}$.

    In computing $T_{h+1}$, the rightmost symbol $u_{h+1}$ is inserted first and becomes the root node, with the
    remaining $t_2-o_2 - 1$ symbols descending from it on the path of left child nodes. Since $q_h = q_{h+1} > u_{h+1}$,
    the $s_2$ symbols $q_h$ are inserted into the right subtree of the root node $u_{h+1}$. Every symbol in $\beta$ is
    greater than $q_h$, so the subtree $\beta$ is re-inserted as the right child of the topmost $q_h$. Every symbol in
    $D_h$ is greater than $u_{h+1}$ and less than or equal to $q_h$, so the subtree $D_h$ is re-inserted as the left
    child node of the bottommost node $q_h$. Every symbol in $\delta$ is greater than $q_h$, so $\delta$ is inserted
    into the subtree $\beta$. The remaining $s_1$ symbols $q_h$ are inserted into the right-maximal subtree of
    $D_h$. The symbol $m_h$ is the smallest symbol greater than $u_{h+1}$ and so the $r$ symbols $m_h$ are inserted into
    the left-minimal subtree of $D_h$ (attached as the left child of the node $m_h$ in $D_h$). The next $o_2$ symbols
    $u_{h+1}$ are inserted into the left subtree of the bottommost node $u_{n+1}$, so that there are $t_2-o_2+o_2 = t_2$
    consecutive nodes $u_{h+1}$.  Every symbol in $D_g$ and $\lambda$ is less than or equal to $u_{h+1}$ and so $D_g$
    and $\lambda$ are re-inserted at the left child of the bottommost node $u_{h+1}$. The remaining $o_1-t_2+o_2 = t_1$
    nodes $u_{h+1}$ are inserted into the right-maximal subtree of $D_g$.

  \end{enumerate}

  The three sub-sub-cases where $s_2 = 0$ (see below) differ from the above three sub-sub-cases in the same way that the two
  sub-sub-cases in sub-case~4(c) differ: instead of taking a reading of $T_h$ with $\beta$ at the start, take one where
  $\beta$ appears just before the string $q_h^s$. Since the reasoning is so similar, these sub-sub-cases are thus
  treated in an abbreviated form:

  \begin{itemize}

  \item[(4)] There are no $u_{h+1}$ above $D_g$. Suppose there are $r_1$ nodes $m_h$ below $D_g$ and $r_2$ above. Then
    \[
      T_h = \psylv{u_{h+1}^tm_h^{r_1}\lambda D_g m_h^{r_2}\beta q_h^{s}\delta D_h}.
    \]
    Let
    \[
      T_{h+1} = \psylv{u_{h+1}^{t_1}m_h^{r_1}\lambda D_g m_h^{r_2}\beta q_h^{s}\delta D_hu_{h+1}^{t_2}},
    \]
    Note that $T_h \cyc T_{h+1}$.

  \item[(5)] There are $o_1$ nodes $u_{h+1}$ below $D_g$ and $o_2$ above, where $o_2 \geq t_2$. Then
    \[
      T_h = \psylv{u_{h+1}^{o_1}\lambda D_g u_{h+1}^{o_2}m_h^{r}\beta q_h^{s}\delta D_h}.
    \]
    Let
    \[
      T_{h+1} = \psylv{u_{h+1}^{o_2 - t_2}m_h^{r}\beta q_h^{s}\delta D_hu_{h+1}^{o_1}\lambda D_g u_{h+1}^{t_2}}.
    \]

  \item[(6)] There are $o_1$ nodes $u_{h+1}$ below $D_g$ and $o_2$ above, where $o_2 < t_2$. Then
    \[
      T_h = \psylv{u_{h+1}^{o_1}\lambda D_g u_{h+1}^{o_2}m_h^{r} \beta q_h^{s}\delta D_h}.
    \]
    Let
    \[
      T_{h+1} = \psylv{u_{h+1}^{o_1-t_2+o_2}\lambda D_g u_{h+1}^{o_2}m_h^{r}\beta q_h^{s}\delta D_hu_{h+1}^{t_2-o_2}}.
    \]
    Note that $o_1-t_2+o_2 = t_1$ and that $T_h \cyc T_{h+1}$.

  \end{itemize}

  The subtree $D_h$, the nodes $m_h$, and possibly some of the $q_h$ below $D_h$ make up $B_h$. This tree $B_h$,
  together with the nodes $q_h$ above $D_h$, the nodes $u_{h+1}$ and $D_{g}$ together make up $D_{h+1}$. Adding the
  remaining nodes $q_h = q_{h+1}$ below $D_h$ gives the tree $E_{h+1}$. So $T_{h+1}$ satisfies P1. The other trees
  $E_{i_j}$ in $T_2$ were in $\lambda$; this still holds and so $T_{h+1}$ satisfies P2. Every node not in $E_{h+1}$ is
  in its left-minimal or right-maximal subtree; together with the fact that $T_h$ satisfies P3, this shows that
  $T_{h+1}$ satisfies P3. Finally, $T_{h+1}$ satisfies P4 since $T_h$ does. (Note that $m_h$ is
  below $p_h = u_{h+1}$, but this does not matter since $u_h$ is not in $U_{h+1}^\uparrow$.)

  \bigskip
  \noindent{Conclusion.} The tree $T_n$ satisfies the conditions P1--P4. In particular, $E_n$ appears at the root of
  $T_n$. Since $u_n$ is the root of $U$ (since it is obviously the last node visited by the topmost traversal),
  $B_n = U$. Thus $C_n$ is $U$ with all nodes $m_n$ deleted except the topmost. Since $q_n$ is undefined,
  $E_n = D_n = C_n$. Hence $C_n$ appears at the roots of $T_n$ and $U$. However, the number of nodes $m_n$ in the trees
  $T_n$ and $U$ are equal, and since $m_n$ is the smallest symbol appearing in $T_n$ or $U$, all the nodes $m_n$ in
  $T_n$ and $U$ must appear in the paths of left child nodes in $T_n$ and $U$, in the left subtree of the single node
  $m_n$ in $E_n$. Hence $T_n = U$.

  Thus there is a sequence $T = T_0, T_1, \ldots, T_n = U$ with $T_i \cyc T_{i+1}$ for $i \in \set{1,\ldots,n-1}$. Since
  $T$ and $U$ were arbitrary elements of an abritrary connected component of $K(\sylv_n)$, the diameter of any connected
  component of $K(\sylv_n)$ is at most $n$. This completes the proof.
\end{proof}

\section{Stalactic monoid}
\label{sec:stalactic}

The stalactic monoid is primarily used in the definition of the more interesting taiga monoid (see
\fullref{Section}{sec:taiga}), but its cyclic shift graph exhibits a particularly simple structure. As usual, this section
recalls only the essentials background; for further reading, see \cite{priez_lattice}.

A \defterm{stalactic tableau} is a finite array of symbols from $\aA$ in which columns are top-aligned, and two symbols
appear in the same column if and only if they are equal. For example,
\begin{equation}
\label{eq:egstaltab1}
\tikz[tableau]\matrix{
3 \& 1 \& 2 \& 6 \& 5 \\
3 \& 1 \&   \& 6 \& 5 \\
  \& 1 \&   \&   \& 5 \\
  \& 1 \\
};
\end{equation}
is a stalactic tableau. The insertion algorithm is very straightforward:

\begin{algorithm}
\label{alg:stalinsertone}
~\par\nobreak
\textit{Input:} A stalactic tableau $T$ and a symbol $a \in \aA$.

\textit{Output:} A stalactic tableau $T \leftarrow a$.

\textit{Method:} If $a$ does not appear in $T$, add $a$ to the left of the top row of $T$. If $a$ does appear in $T$,
add $a$ to the bottom of the (by definition, unique) column in which $a$ appears. Output the new tableau.
\end{algorithm}

Thus one can compute, for any word $u \in \aA^*$, a stalactic tableau $\pstal{u}$ by starting with an empty stalactic
tableau and successively inserting the symbols of $u$, proceeding right-to-left through the word. For example
$\pstal{361135112565}$ is \eqref{eq:egstaltab1}. Notice that the order in which the symbols appear along the first row
in $\pstal{u}$ is the same as the order of the rightmost instances of the symbols that appear in $u$. Define the relation
$\stalcong$ by
\[
  u \stalcong v \iff \pstal{u} = \pstal{v}
\]
for all $u,v \in \aA^*$. The relation $\stalcong$ is a congruence, and the \defterm{stalactic monoid}, denoted $\stal$,
is the factor monoid $\aA^*\!/{\stalcong}$; The \defterm{stalactic monoid of rank $n$}, denoted $\stal_n$, is the factor
monoid $\aA_n^*/{\stalcong}$ (with the natural restriction of $\stalcong$). Each element $[u]_{\stalcong}$ (where
$u \in \aA^*$) can be identified with the stalactic tableau $\pstal{u}$. Note that if $T$ is a stalactic tableau
consisting of a single row (that is, will all columns having height $1$), then there is a unique word $u \in \aA^*$,
formed by reading the entries of $T$ left-to-right, such that $\pstal{u} = T$. Thus, if $T = \pstal{a_1\cdots a_k}$ and
$U \in stal$ is such that $U \cyc T$, then $U = \pstal{a_i\cdots a_ka_1\cdots a_{i-1}}$ for some $i$.

The monoid $\stal$ is presented by
$\pres{\aA}{\drel{R}_\stal}$, where
\[
\drel{R}_\stal = \gset[\big]{(bavb,abvb)}{a,b\in A,\; v \in A^*}.
\]
The monoid $\stal_n$ is presented by $\pres{\aA_n}{\drel{R}_\stal}$, where the set of defining relations
$\drel{R}_\stal$ is naturally restricted to $\aA_n^*\times \aA_n^*$. Notice that $\stal$ and $\stal_n$ are multihomogeneous.

[The stalactic monoid was originally defined by Hivert et al. \cite[\S~3.7]{hivert_commutative} using the defining
relations $\gset[\big]{(bvb,bbv)}{b\in A, v \in A^*}$; this would yield a monoid that is anti-isomorphic to $\stal$. The
definition of $\stal$ here follows Priez \cite[Example~3]{priez_lattice} so as to be compatible with the construction of
the taiga monoid below.]

\begin{proposition}
  \label{prop:stalproperlycontained}
  Connected components of $K(\stal)$ are properly contained in $\evrel$-classes.
\end{proposition}

\begin{proof}
  By \fullref[(1)]{Lemma}{lem:basicprops}, ${\cyc} \subseteq {\evrel}$, so it remains to prove that equality does not
  hold.  Since no defining relations in $\drel{R}_\stal$ can be applied to a word without repeated letters, a stalactic
  tableau consisting of a single row is represented by exactly one word over $\aA^*$. In particular, an element of the
  form $\tikz[tableau]\matrix{k \& |[dottedentry]| \& n \& 1 \& |[dottedentry]| \& k{-}1 \\};$ is represented by the
  unique word $k\cdots n1\cdots (k{-}1)$. Thus the elements of this form are all $\cyc$-related and form a
  $\cyc^*$-class and thus a connected component of $K(\stal)$. Thus, for $n \geq 3$, the elements
  $\tikz[tableau]\matrix{1 \& 2 \& 3 \& |[dottedentry]| \& n \\};$ and
  $\tikz[tableau]\matrix{2 \& 1 \& 3 \& |[dottedentry]| \& n \\};$ are $\evrel$-related but in different connected
  components of $K(\stal)$.
\end{proof}

However, it is possible to characterize connected components of $K(\stal)$. For any stalactic tableau $T$, let $\iota(T)$ be the
word obtained by reading from left to right the symbols that appear in columns of height $1$. Define $\kappa(T)$ to be
the pair $\parens[\big]{[\iota(T)]_{\cyc},\ev{T}}$. (Notice that $[\iota(T)]_{\cyc} = [\iota]_{\cyc^*}$.)

\begin{proposition}
  \begin{enumerate}
  \item Two elements of $\stal$ lie in the same connected component of $K(\stal)$ if and only if they have the same
    image under the map $\kappa$.
  \item The maximum diameter of a connected component of $K(\stal)$ is $3$.
  \item The maximum diameter of a connected component of $K(\stal_n)$ is $3$ if $n \geq 3$, and is respectively $1$ and
    $0$ for $n = 2$ and $n = 1$.
  \end{enumerate}
\end{proposition}

\begin{proof}
  Suppose $T,U \in \stal$ are such that $T \cyc U$. Then there exist $x,y \in \aA^*$ such that $xy$ represents $T$ and
  $yx$ represents $U$. Deleting from $x$ and $y$ every symbol $a$ such that $|x|_a + |y|_a > 1$ yields two words $x'$
  and $y'$ with $\iota(T) = x'y'$ and $\iota(U) = y'x'$; thus $\iota(T) \cyc \iota(U)$. Furthermore, since
  $T \cyc U$ it follows that $\ev{T} = \ev{U}$. Thus $\kappa(T) = \kappa(U)$. Iterating this reasoning shows that if
  $T$ and $U$ lie in the same connected component of $K(\stal)$, then $\kappa(T) = \kappa(U)$.

  Now suppose that $T,U \in \stal$ are such that $\kappa(T) = \kappa(U)$. Let $B = \set{b_1,\ldots,b_k}$ consist of
  exactly the symbols in $\aA$ that appear more than once in $U$ and thus in $T$. Choose any word $t$ such that $\pstal{t} = T$, and
  delete the leftmost appearance of each symbol in $B$ from $t$; call the resulting word $t'$. Let
  $t_0 = b_1\cdots b_k t' \in \aA^*$.  Since the order of columns in a stalactic tableau corresponding to a word is
  determined by the rightmost appearance of each symbol in that word, and since $\ev{t} = \ev{t_0}$, it follows that
  $\pstal{t_0} = T$. Let $T_1 = \pstal{t'b_1\cdots b_k}$, so that $T \cyc T_1$. Then $T_1$ is of the form
  \begin{equation}
    \label{eq:stalbounds1}
    \begin{tikzpicture}
      \matrix[tableaumatrix,name=maintableau,topalign]{
        a_1 \& |[dottedentry]| \null \& a_m \& b_1 \& b_2 \& |[dottedentry]| \null \& b_k \\
        \& \& \& |[dottedentry]| \null \& |[dottedentry]| \null \& |[dottedentry]| \null \& |[dottedentry]| \null \\
        \& \& \& b_1 \& \& |[dottedentry]| \null \& b_k \\
        \& \& \& b_1 \\
      };
      \node at ($ (maintableau-1-2) + (0mm,3.5mm) $) {$\overbrace{\hbox{\vrule width 14mm height 0cm depth 0cm}}$};
      \node[anchor=south] at ($ (maintableau-1-2) + (0mm,4mm) $) {Symbols that appear once};
    \end{tikzpicture}.
  \end{equation}

  Similarly, choose any word $u$ with $\pstal{u} = U$, delete the leftmost appearance of each symbol in $B$ to obtain a
  word $u'$, and let $u_0 = b_1\cdots b_ku'$; then $\pstal{u_0} = U$. Let $U_1 = \pstal{b_1\cdots b_ku'}$. Then
  $\iota(U_1) \cyc \iota(U) \cyc \iota(T) \cyc  \iota(T_1)$, and so $\iota(U_1) \cyc \iota(T_1)$. Thus $U_1$ of the form
  \begin{equation}
    \label{eq:stalbounds2}
    \begin{tikzpicture}
      \matrix[tableaumatrix,name=maintableau,topalign]{
        a_{h+1} \& |[dottedentry]| \null \& a_m \& a_1 \& |[dottedentry]| \& a_h \& b_1 \& b_2 \& |[dottedentry]| \null \& b_k \\
        \& \& \& \& \& \& |[dottedentry]| \null \& |[dottedentry]| \null \& |[dottedentry]| \null \& |[dottedentry]| \null \\
        \& \& \& \& \& \& b_1 \& \& |[dottedentry]| \null \& b_k \\
        \& \& \& \& \& \& b_1 \\
      };
      \node at ($ (maintableau-1-3) + (.5mm,3.5mm) $) {$\overbrace{\hbox{\vrule width 29mm height 0cm depth 0cm}}$};
      \node[anchor=south] at ($ (maintableau-1-3) + (.5mm,4mm) $) {Symbols that appear once};
    \end{tikzpicture}.
  \end{equation}
  Note that \eqref{eq:stalbounds1} and \eqref{eq:stalbounds2} differ only by a cyclic permutation of the columns of
  height $1$.

  Let $s \in \aA^*$ be such that $\pstal{a_{h+1}\cdots a_ma_1\cdots a_hs} = U_1$. Delete the leftmost appearance of each
  symbol in $B$ from the word $s$; call the resulting word $s'$. Again using the fact that the rightmost appearance of
  each symbol determines the order of columns, $\pstal{a_{h+1}\cdots a_mb_1\cdots b_ka_1\cdots a_hs'} = U_1$. Similarly,
  $\pstal{a_1\cdots a_hs'a_{h+1}\cdots a_mb_1\cdots b_k} = T_1$. Hence $U_1 \cyc T_1$.

  Thus $T \cyc T_1 \cyc U_1 \cyc U$, and so there is a path of length at most $3$ from $T$ to $U$. Hence $T$ and $U$ lie
  in the same connected component. This completes the proof of part~1).

  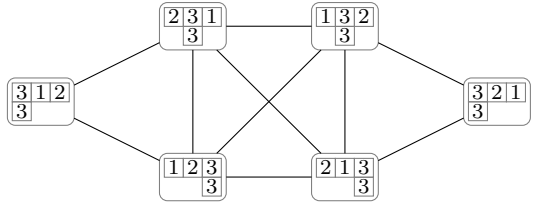
\begin{figure}[t]
    \centering
    \begin{tikzpicture}[x=20mm,y=20mm]
      %
      \begin{scope}[vertex/.style={draw=gray,rectangle with rounded corners,inner sep=-.5mm}]
        \node[vertex] (3312) at (-1.5,0) {\smalltableau{3 \& 1 \& 2\\3\\}};
        \node[vertex] (1233) at (-.5,-.5) {\smalltableau{1 \& 2 \& 3\\\&\& 3\\}};
        \node[vertex] (2331) at (-.5,.5) {\smalltableau{2 \& 3 \& 1\\\& 3\\}};
        \node[vertex] (1332) at (.5,.5) {\smalltableau{1 \& 3 \& 2\\\& 3\\}};
        \node[vertex] (2133) at (.5,-.5) {\smalltableau{2 \& 1 \& 3\\\&\& 3\\}};
        \node[vertex] (3321) at (1.5,0) {\smalltableau{3 \& 2 \& 1\\3\\}};
      \end{scope}
      \begin{scope}
        \draw (1233) edge (1332);
        \draw (1233) edge (2331);
        \draw (1233) edge (2133);
        \draw (1233) edge (3312);
        \draw (1332) edge (2331);
        \draw (1332) edge (2133);
        \draw (1332) edge (3321);
        \draw (3312) edge (2331);
        \draw (2331) edge (2133);
        \draw (3321) edge (2133);
      \end{scope}
    \end{tikzpicture}
    \caption{The connected component $K(\stal,\pstal{1233})$, which has diameter $3$.}
    \label{fig:kstal1233}
  \end{figure}
  Furthermore, this shows that connected components of $K(\stal)$ have diameter at most $3$. Direct calculation shows
  that the connected component $K(\stal,\pstal{1233})$ is as shown in \fullref{Figure}{fig:kstal1233} and thus has diameter
  $3$. This proves part~2). For part~3), note that connected components of $K(\stal_1)$ are singleton vertices, and the
  connected components of $K(\stal_2)$ have at most two vertices corresponding to the two possible orders of columns
  filled by symbols $1$ and $2$.
\end{proof}

\section{Taiga monoid}
\label{sec:taiga}

The taiga monoid is a quotient of the sylvester monoid that is associated with a modified notion of binary search
tree. As usual, this section only recalls the essential facts; see \cite[\S~5]{priez_lattice} for further
background.

A \defterm{binary search tree with multiplicities} is a labelled search tree in which each label appears at most once,
where the label of each node is greater than the label of every node in its left subtree, and less than the label of
every node in its right subtree, and where each node label is assigned a positive integer called its
\defterm{multiplicity}. An example of a binary search tree is:
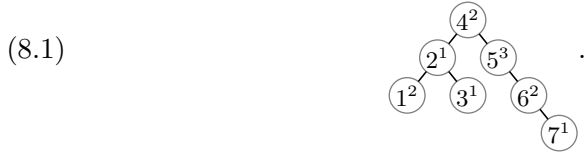
\begin{equation}
\label{eq:bstmulteg}
\begin{tikzpicture}[tinybst,baseline=(0)]
  \node (root) {$4^2$}
    child { node (0) {$2^1$}
      child { node (00) {$1^2$} }
      child { node (01) {$3^1$} }
    }
    child { node (1) {$5^3$}
      child[missing]
      child { node (11) {$6^2$}
        child[missing]
        child { node (110) {$7^1$} }
      }
    };
\end{tikzpicture}.
\end{equation}
(The superscripts on the labels in each node denote the multiplicities.)

\begin{algorithm}
\strut\par\nobreak
\textit{Input:} A binary search tree with multiplicities $T$ and a symbol $a \in \aA$.

\textit{Output:} A binary search tree with multiplicities $T \leftarrow a$.

\textit{Method:} If $T$ is empty, create a node, label it by $a$, and assign it multiplicity $1$. If $T$ is non-empty,
examine the label $x$ of the root node; if $a < x$, recursively insert $a$ into the left subtree of the root node; if $a
> x$, recursively insert $a$ into the right subtree of the root node; if $a=x$, increment by $1$ the multiplicity of the node
label $x$.
\end{algorithm}

Thus one can compute, for any word $u \in \aA^*$, a binary search tree with multiplicities $\ptaig{u}$ by starting with an
empty binary search tree with multiplicities and successively inserting the symbols of $u$, proceeding right-to-left
through the word. For example $\ptaig{65117563254}$ is \eqref{eq:bstmulteg}.

Define the relation $\taigcong$ by
\[
u \taigcong v \iff \ptaig{u} = \ptaig{v},
\]
for all $u,v \in \aA^*$. The relation $\taigcong$ is a congruence, and the \defterm{taiga monoid}, denoted $\taig$, is
the factor monoid $\aA^*\!/{\taigcong}$; the \defterm{taiga monoid of rank $n$}, denoted $\taig_n$, is the factor monoid
$\aA_n^*/{\taigcong}$ (with the natural restriction of $\taigcong$). Each element $[u]_{\taigcong}$ can be identified
with the binary search tree with multiplicities $\ptaig{u}$.

As with [ordinary] binary search trees, a \defterm{reading} of a binary search tree with multiplicities $T$ is a word
$u$ such that $\ptaig{u} = T$. It is easy to see that a reading of $T$ is a word formed from the symbols that appear in
the nodes of $T$, with the number of times each symbol appears being its multiplicity, arranged so that the rightmost
symbol from a parent node appears to the right of the rightmost symbols from its children.  For example, $135671456254$
is a reading of \eqref{eq:bstmulteg}.


The monoid $\taig$ is presented by $\pres{\aA}{\drel{R}_\taig}$, where
$\drel{R}_\taig = \drel{R}_\sylv \cup \drel{R}_\stal$; the monoid $\taig_n$ is presented by
$\pres{\aA_n}{\drel{R}_\taig}$, where the set of defining relations $\drel{R}_\taig$ is naturally restricted to
$\aA_n^*\times \aA_n^*$. Notice that $\taig$ and $\taig_n$ are multihomogeneous.

The taiga monoid is a quotient of the sylvester monoid under the homomorphism $\tau : \sylv \to \taig$ sending
$[u]_\sylv$ to $[u]_\taig$ (or equivalently, $\psylv{u}$ to $\ptaig{u}$) for all $u \in \aA^*$. This homomorphism
naturally restricts to a surjective homomorphism $\tau : \sylv_n \to \taig_n$. This connection between $\sylv$ and
$\taig$ makes it possible to use the reasoning about the diameters of connected components in $K(\sylv_n)$ to prove the
corresponding results for $K(\taig_n)$.

Since defining relations in $\drel{R}_\stal$ involve repeated symbols, only relations in $\drel{R}_\sylv$ apply to
standard words. That is, standard words are related by $\taigcong$ if and only if they are related by $\sylvcong$. Thus
the proof of \fullref{Lemma}{lem:sylvesterlowerbound} applies in $\taig_n$ to establish the following result:

\begin{lemma}
  \label{lem:taiglowerbound}
  There is a connected component in $K(\taig_n)$ of diameter at least $n-1$.
\end{lemma}

\begin{lemma}
  \label{lem:taigupperbound}
  Every connected component of $K(\taig_n)$ has diameter at most $n$.
\end{lemma}

\begin{proof}
  Define a map $\psi : \taig_n \to \sylv_n$ that maps a binary search tree with multiplicities $T$ to the standard
  binary search tree obtained by deleting the multiplicities of $T$. Note that, two elements $T,U \in \taig_n$ are equal
  if and only if $\psi(T) = \psi(U)$ and $\ev{T} = \ev{U}$.

  Suppose $P,Q \in \taig_n$ are such that $\psi(P) \cyc \psi(Q)$ (in $\sylv_n$) and $\ev{P} = \ev{Q}$. Then there
  are readings $xy$ of $\psi(P)$ and $yx$ of $\psi(Q)$. Replacing each symbol $a$ with $a^{k_a}$ for each $a \in \aA_n$,
  where $k_a$ is the $a$-th component of $\ev{P} = \ev{Q}$, gives readings $\hat{x}\hat{y}$ of $P$ and
  $\hat{y}\hat{x}$ of $Q$, so that $P \cyc Q$ (in $\taig_n$).

  Let $T$ and $U$ be elements of the same connected component of $K(\taig_n)$. Then $T \evrel U$. By the strategy for
  bulding a path in $K(\sylv_n)$ in the proof of \fullref{Proposition}{prop:sylvupperbound}, there is a
  path $\psi(T) = T_0,\ldots,T_n = \psi(U)$ in $K(\sylv_n)$. By the reasoning in the previous paragraph, this path lifts
  to a path $T = \hat{T}_0,\ldots,\hat{T}_n = U$ in $K(\taig_n)$. Thus the diameter of $K(\taig_n,T)$ is at most $n$.
\end{proof}

Combining \fullref{Lemmata}{lem:taiglowerbound} and \ref{lem:taigupperbound} gives the result:

\begin{theorem}
  \label{thm:taigbounds}
  \begin{enumerate}
  \item Connected components of $K(\taig)$ coincide with $\evrel$-classes in $\taig$.
  \item The maximum diameter of a connected component of $K(\taig_n)$ is $n-1$ or $n$.
  \end{enumerate}
\end{theorem}

\section{Baxter monoid}
\label{sec:baxter}

The Baxter monoid is a monoid of pairs of twin binary search trees. As in previous sections, only the essential facts
are given here; see \cite{giraudo_baxter2} for further background.

A \defterm{left strict binary search tree} is a labelled rooted binary tree where the label of each node is strictly
greater than the label of every node in its left subtree, and less than or equal to every node in its right subtree; see
the left tree shown in \eqref{eq:baxtpaireg} below for an example.


The \defterm{canopy} of a binary tree $T$ is the word over $\set{0,1}$ obtained by traversing the empty subtrees of
the nodes of $T$ from left to right, except the first and the last, labelling an empty left subtree by $1$ and an empty
right subtree by $0$. (See \eqref{eq:baxtpaireg} below for examples of canopies.)

A \defterm{pair of twin binary search trees} consist of a left strict binary search tree $T_L$ and a right strict binary
search tree $T_R$, such that $T_L$ and $T_R$ contains the same symbols, and the canopies of $T_L$ and $T_R$ are
complementary, in the sense that the $i$-th symbol of the canopy of $T_L$ is $0$ (respectively $1$) if and only if the
$i$-th symbol of the canopy of $T_L$ is $1$ (respectively $0$). The following is an example of a pair of twin binary
search trees, with the complementary canopies $0110101$ and $1001010$ shown in grey:
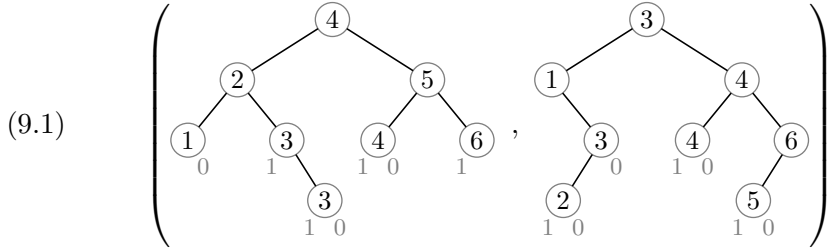
\begin{equation}
  \label{eq:baxtpaireg}
  \left(
  \begin{tikzpicture}[baseline=-15mm]
    \begin{scope}[
      smallbst,
      level 1/.style={sibling distance=25mm},
      level 2/.style={sibling distance=13mm},
      ]
      \node (root) {$4$}
      child[sibling distance=25mm] { node (0) {$2$}
        child { node (00) {$1$} }
        child { node (01) {$3$}
          child[missing]
          child { node (011) {$3$} } } }
      child[sibling distance=25mm] { node (1) {$5$}
        child { node (10) {$4$} }
        child { node (11) {$6$} } };
    \end{scope}
    \begin{scope}[every node/.style={gray,font=\footnotesize}]
      \node at ($ (00) + (2mm,-3.5mm) $) {$0$};
      \node at ($ (01) + (-2mm,-3.5mm) $) {$1$};
      \node at ($ (011) + (-2mm,-3.5mm) $) {$1$};
      \node at ($ (011) + (2mm,-3.5mm) $) {$0$};
      \node at ($ (10) + (-2mm,-3.5mm) $) {$1$};
      \node at ($ (10) + (2mm,-3.5mm) $) {$0$};
      \node at ($ (11) + (-2mm,-3.5mm) $) {$1$};
    \end{scope}
  \end{tikzpicture}
  \;\;
  ,
  \;
  \begin{tikzpicture}[baseline=-15mm]
    \begin{scope}[
      smallbst,
      level 1/.style={sibling distance=25mm},
      level 2/.style={sibling distance=13mm},
      ]
      \node (root) {$3$}
      child { node (0) {$1$}
        child[missing]
        child { node (01) {$3$}
          child { node (010) {$2$} }
          child[missing]
          }
        }
      child { node (1) {$4$}
        child { node (10) {$4$} }
        child { node (11) {$6$}
          child { node (110) {$5$} }
          child[missing]
        }
      };
    \end{scope}
    \begin{scope}[every node/.style={gray,font=\footnotesize}]
      \node at ($ (010) + (-2mm,-3.5mm) $) {$1$};
      \node at ($ (010) + (2mm,-3.5mm) $) {$0$};
      \node at ($ (01) + (2mm,-3.5mm) $) {$0$};
      \node at ($ (10) + (-2mm,-3.5mm) $) {$1$};
      \node at ($ (10) + (2mm,-3.5mm) $) {$0$};
      \node at ($ (110) + (-2mm,-3.5mm) $) {$1$};
      \node at ($ (110) + (2mm,-3.5mm) $) {$0$};
    \end{scope}
  \end{tikzpicture}
  \right)
\end{equation}

The insertion algorithm for left strict binary search trees is symmetric to \fullref{Algorithm}{alg:sylvinsertone}:

\begin{algorithm}[Left strict leaf insertion]
\label{alg:leftbstinsertone}
~\par\nobreak
\textit{Input:} A left strict binary search tree $T$ and a symbol $a \in \aA$.

\textit{Output:} A left strict binary search tree $a \rightarrow T$.

\textit{Method:} If $T$ is empty, create a node and label it $a$. If $T$ is non-empty, examine the label $x$ of the root
node; if $a < x$, recursively insert $a$ into the left subtree of the root node; otherwise recursively insert $a$
into the right subtree of the root node. Output the resulting tree.
\end{algorithm}

Thus one can compute, for any word $u \in \aA^*$, a pair of twin binary search trees $\pbaxt{u} = (T_L,T_R)$, where
$T_R$ is $\psylv{u}$ and $T_L$ is obtained by starting with an empty left strcit binary search tree and successively
inserting the symbols of $u$, proceeding left-to-right through the word. For example, $\pbaxt{42531643}$ is
\eqref{eq:baxtpaireg}.

A \defterm{reading} of a pair of twin binary search trees $(T_L,T_R)$ is a word $u$ such that $\pbaxt{u} =
(T_L,T_R)$. It is easy to see that a reading of $(T_L,T_R)$ is a word formed from the symbols appearing in the two
binary trees $T_L$ and $T_R$ (which, by definition, contain the same symbols), ordered so that every symbol from a
parent node in $T_L$ appears to the left of those from its children in $T_L$, and every symbol from parent node in $T_R$
appears to the right of those from its children.

Define the relation $\baxtcong$ by
\[
u \baxtcong v \iff \pbaxt{u} = \pbaxt{v},
\]
for all $u,v \in \aA^*$. The relation $\baxtcong$ is a congruence, and the \defterm{Baxter monoid}, denoted $\baxt$, is
the factor monoid $\aA^*\!/{\baxtcong}$; the \defterm{Baxter monoid of rank $n$}, denoted $\baxt_n$, is the factor
monoid $\aA_n^*/{\baxtcong}$ (with the natural restriction of $\baxtcong$). Each element $[u]_{\baxtcong}$ (where
$u \in \aA^*$) can be identified with the pair of twin binary search trees $\pbaxt{u}$. The words in $[u]_{\baxtcong}$
are precisely the readings of $\pbaxt{u}$.

The monoid $\baxt$ is presented by
$\pres{\aA}{\drel{R}_\baxt}$, where
\begin{align*}
\drel{R}_\baxt ={}&\gset{(cudavb,cuadvb)}{a \leq b < c \leq d,\; u,v \in A^*} \\
&\cup \gset{(budavc,buadvc)}{a < b \leq c < d,\; u,v \in A^*};
\end{align*}
see \cite[Definition~3.1]{giraudo_baxter2}. The monoid $\baxt_n$ is presented by $\pres{\aA_n}{\drel{R}_\baxt}$, where
the set of defining relations $\drel{R}_\baxt$ is naturally restricted to $\aA_n^* \times \aA_n^*$. Note that $\baxt$
and $\baxt_n$ are multihomogeneous.

There is a straightforward method for extracting every possible reading from a pair of binary search trees
$(T_L,T_R)$:

\begin{method}
\label{method:baxterreading}

\textit{Input:} A pair of twin binary search trees $(T_L,T_R)$.

\textit{Output:} A reading of $(T_L,T_R)$.
\begin{enumerate}
\item Set $(U_L,U_R)$ to be $(T_L,T_R)$. (Throughout this computation, $U_L$ is a forest of
  left strict binary search trees and $U_R$ is a right strict binary search tree.)
\item If $U_L$ and $U_R$ are empty, halt.
\item Given some $(U_L,U_R)$, choose and output some symbol $a$ that labels a root of some tree in the forest $U_L$ and a leaf
  of the tree $U_R$.
\item Deleting the corresponding root vertex of $U_L$ and the corresponding leaf vertex of $U_L$.
\end{enumerate}
\end{method}

This is essentially \cite[Algorithm on p.133]{giraudo_baxter2}, except that the method given here is non-deterministic
in that there may be several choices for $a$ in step~3. As these choices vary, all possible readings of $(T_L,T_R)$ are obtained.

\begin{proposition}
  \label{prop:baxtproperlycontained}
  Connected components of $K(\baxt)$ are properly contained in $\evrel$-classes.
\end{proposition}

\begin{proof}
  Connected components of $K(\baxt)$ are $\cyc^*$-classes, and by \fullref{Lemma}{lem:basicprops},
  ${\cyc}^* \subseteq {\evrel}$. It thus remains to prove that equality does not hold. Since all the defining relations
  in $\drel{R}_\baxt$ have length at least $4$, it follows that none of these relations can be applied to words of
  length $3$. Thus all length-$3$ words represent distinct elements of $\baxt$. Therefore the words in
  $\set{123,231,312}$ represent all the elements in one $\cyc^*$-class in $\baxt$ (and thus one connected component of
  $K(\baxt)$. The word $132$ is not in this set, but is in the same $\evrel$-class. This completes the proof.
\end{proof}

A natural question at this point is whether ${\cyc^*}$ and ${\evrel}$ do not coincide in $\baxt$ only for the slightly
trivial reason that relations in $\drel{R}_\baxt$ do not apply to words of length $3$, and that perhaps ${\cyc^*}$ and
${\evrel}$ coincide for elements represented by words of length $4$ or more. However, consider the elements of
$\baxt$ represented by the words $1243$, $2431$, $4312$, and $3124$:
\begin{align*}
  \pbaxt{1243} &= \left(
    \begin{tikzpicture}[tinybst,baseline=-8mm]
      \node (root) {$1$}
      child[missing]
      child { node (1) {$2$}
        child[missing]
        child { node (11) {$4$}
          child { node (110) {$3$} }
          child[missing]
        }
      };
    \end{tikzpicture}
    \;\;,\;
    \begin{tikzpicture}[tinybst,baseline=-8mm]
      \node (root) {$3$}
      child { node (0) {$2$}
        child { node (00) {$1$} }
        child[missing]
      }
      child { node (1) {$4$} };
    \end{tikzpicture}
    \right) \displaybreak[0]\\
  \pbaxt{2431} &= \left(
    \begin{tikzpicture}[tinybst,baseline=-5mm]
      \node (root) {$2$}
      child { node (0) {$1$} }
      child { node (1) {$4$}
        child { node (10) {$3$} }
        child[missing]
      };
    \end{tikzpicture}
    \;\;,\;
    \begin{tikzpicture}[tinybst,baseline=-5mm]
      \node (root) {$1$}
      child[missing]
      child { node (1) {$3$}
        child { node (10) {$2$} }
        child { node (11) {$4$} }
      };
    \end{tikzpicture}
    \right) \displaybreak[0]\\
  \pbaxt{4312} &= \left(
    \begin{tikzpicture}[tinybst,baseline=-8mm]
      \node (root) {$4$}
      child { node (0) {$3$}
        child { node (00) {$1$}
          child[missing]
          child { node (001) {$2$} }
        }
        child[missing]
      }
      child[missing];
    \end{tikzpicture}
    \;\;,\;
    \begin{tikzpicture}[tinybst,baseline=-8mm]
      \node (root) {$2$}
      child { node (0) {$1$} }
      child { node (1) {$3$}
        child[missing]
        child { node (11) {$4$} }
      };
    \end{tikzpicture}
    \right) \displaybreak[0]\\
  \pbaxt{3124} &= \left(
    \begin{tikzpicture}[tinybst,baseline=-5mm]
      \node (root) {$3$}
      child { node (1) {$1$}
        child[missing]
        child { node (11) {$2$} }
      }
      child { node (110) {$4$} };
    \end{tikzpicture}
    \;\;,\;
    \begin{tikzpicture}[tinybst,baseline=-5mm]
      \node (root) {$4$}
      child { node (0) {$2$}
        child { node (00) {$1$} }
        child { node (01) {$3$} }
      }
      child[missing];
    \end{tikzpicture}
    \right)
\end{align*}
It is straightforward to prove that there is exactly one reading of each of these pairs of twin binary search trees: for
example, consider extracting a reading from $\pbaxt{2431}$. Following \fullref{Method}{method:baxterreading},
$(U_L,U_R)$ is initially $\pbaxt{2431}$. The first output symbol must be $2$, since this is the unique root in
$U_L$. Deleting the corresponding vertices yields
\[
  (U_L,U_R) = \left(
    \begin{tikzpicture}[tinybst,baseline=-5mm]
      \node (0) at (-4mm,-5mm) {$1$};
      \node (1) at (4mm,-5mm) {$4$}
        child { node (10) {$3$} }
        child[missing];
    \end{tikzpicture}
    \;\;,\;
    \begin{tikzpicture}[tinybst,baseline=-5mm]
      \node (root) {$1$}
      child[missing]
      child { node (1) {$3$}
        child[missing]
        child { node (11) {$4$} }
      };
    \end{tikzpicture}
    \right)
\]
From this point onwards, there will be exactly one leaf vertex in $U_R$, and so only one choice for the symbol to
output. Hence the method must output $4$, $3$, $1$. Hence the unique reading of $\pbaxt{2431}$ is $2431$.

Hence the element represented by the words in $\set{1243,2431,4312,3124}$ form a single $\cyc^*$-class and so (for
example) $1234$ and $1243$ are not related by ${\cyc^*}$.

\begin{question}
  \begin{enumerate}
  \item Is there a characterization of $\cyc^*$-classes in $\baxt$?
  \item Is there a bound on the diameter of $\cyc^*$-classes in $\baxt_n$?
  \end{enumerate}
\end{question}

\section{Questions}

\begin{question}
  For each monoid $\mathsf{M} \in \set{\plac,\hypo,\sylv,\stal,\taig,\baxt}$, is there an efficient algorithm that takes
  two elements $T,U \in \mathsf{M}$ such that $T \evrel U$ and computes the distance between them in $K(\mathsf{M})$?
\end{question}

With regard to the previous question, note that it is always possible to compute the distance via a brute-force
computation: one could build the entire connected component $K(\mathsf{M},T)$, then find the shortest path from $T$ to
$U$. The question is whether this can be done \emph{efficiently}. Note that the strategies for constructing paths in the
various proofs in this paper do \emph{not} in general find shortest paths between two elements; see
\fullref{Example}{eg:hypoplacticupperbound}.



\section{Appendix: Conjugacy}

In a group, the relation $\cyc$ is simply the usual notion of conjugacy. The concept of cyclic shifts can thus be viewed
in an algebraic way as a generalization to monoids of the concept of conjugacy in groups. Another possible
generalization, introduced by Otto~\cite{otto_conjugacy}, is $o$-conjugacy, defined by
\begin{equation}
\label{eq:oconjdef}
x \sim_o y \iff (\exists g,h \in M)(xg = gy \land hx = yh).
\end{equation}
The relation $\sim_o$ is an equivalence relation. The following result describes how $\sim_o$ is related to $\cyc$ and $\evrel$:

\begin{proposition}
  \label{prop:basicpropsoconj}
  \begin{enumerate}
  \item In any monoid, ${\cyc^*} \subseteq {\sim_o}$.
  \item In any multihomogeneous monoid ${\sim_o} \subseteq {\evrel}$.
  \end{enumerate}
\end{proposition}

\begin{proof}
For the first part, see \cite[\S~1]{araujo_conjugation}. For the second part, see \cite[Lemma~3.2]{cm_conjugacy}.
\end{proof}

Thus in a multihomogeneous monoid, ${\cyc^*} \subseteq {\sim_o} \subseteq {\evrel}$. Since ${\cyc^*} = {\evrel}$ in the
plactic, hypoplactic, sylvester, and taiga monoids, in these settings $\sim_o$ coincides with ${\cyc^*}$ and ${\evrel}$
and thus is not of independent interest. However, it turns out that $\sim_o$ and $\evrel$ coincide in the stalactic and
Baxter monoids. (Recall that ${\cyc^*}$ is strictly contained in ${\evrel}$ in both these monoids; see
\fullref{Propositions}{prop:stalproperlycontained} and \ref{prop:baxtproperlycontained}.)

\begin{proposition}
  In $\stal$, the relations $\sim_o$ and $\evrel$ coincide.
\end{proposition}

\begin{proof}
  Let $u,v \in \aA^*$ be such that $u \evrel v$. In particular, $\pstal{u}$ and $\pstal{v}$ both have $m$ columns,
  for some $m \in \nset$. For $i \in \set{1,\ldots,m}$, let $a_i \in A$ be the symbol that appears in the $i$-th column
  of $\pstal{u}$ and let $b_i \in A$ be the symbol that appears in the $i$-th column of $\pstal{v}$. Let
  $g = a_1\cdots a_m$ and $h = b_1\cdots b_m$. Notice that every symbol that appears in $u$ and $v$ appears exactly once
  in $g$ and $h$. Hence $gu \evrel vg$ and $uh \evrel hv$.  Furthermore, the order of rightmost appearances of
  symbols in $gu$ and $vg$ is identical; together with $gu \evrel vg$, this implies that
  $\pstal{gu} = \pstal{vg}$. Thus $gu \stalcong vg$. Similarly, $uh \stalcong hv$. Hence $u \sim_o v$. This
  proves that ${\evrel} \subseteq {\sim_o}$. The opposite inclusion follows from \fullref[(2)]{Proposition}{prop:basicpropsoconj}.
\end{proof}

\begin{proposition}
  In $\baxt$, the relations $\sim_o$ and $\evrel$ coincide.
\end{proposition}

\begin{proof}
  Let $p,q \in \aA^*$ be such that $p \evrel q$. By \cite[Proposition~3.8]{cm_identities}, $ppq \baxtcong pqq$ and
  $qpp \baxtcong qqp$. Hence $pg \baxtcong gq$ with $g = pq$, and $hp \baxtcong qh$ with $h = qp$. Thus $p \sim_o
  q$. This proves that ${\evrel} \subseteq {\sim_o}$. The opposite inclusion follows from
  \fullref[(2)]{Proposition}{prop:basicpropsoconj}.
\end{proof}

\bibliography{\jobname}
\bibliographystyle{alphaabbrv}

\end{document}